\documentclass[]{article}

\PassOptionsToPackage{x11names}{xcolor}

\usepackage[T1]{fontenc}

\usepackage{biolinum}

\usepackage[utf8]{inputenc}
\usepackage[english]{babel}
\usepackage[backend=biber,url=false,style=alphabetic]{biblatex}

\usepackage[protrusion=true,expansion=true,final,babel]{microtype}

\usepackage{amsmath}
\usepackage{amssymb}
\usepackage{bbm}
\usepackage{bbold}
\usepackage{mathrsfs}
\usepackage{mathtools}
\usepackage[new]{old-arrows}
\usepackage{quiver}
\usepackage{adjustbox}

\usepackage{scalerel}
\usepackage[nointegrals]{wasysym}
\usepackage{stackengine}

\newcommand\reallywidetilde[1]{\ThisStyle{%
  \setbox0=\hbox{$\SavedStyle#1$}%
  \stackengine{-.1\LMpt}{$\SavedStyle#1$}{%
    \stretchto{\scaleto{\SavedStyle\mkern.2mu\AC}{.5150\wd0}}{.6\ht0}%
    }{O}{c}{F}{T}{S}%
}}

\usepackage{csquotes}

\usepackage{mydiagrams}
\usetikzlibrary{decorations.pathmorphing}

\usepackage{geometry}
\usepackage[colorinlistoftodos]{todonotes}

\usepackage[]{enumitem}
\setlist{nosep}

\usepackage{theorems}

\newcommand\ie{i.\,e.\ }
\newcommand\eg{e.\,g.\ }
\newcommand\define[1]{\textit{#1}}
\newcommand\pullback[1]{\underset{#1}{\times}}
\newcommand\pushout[1]{\underset{#1}{\union}}

\renewcommand\equiv{\simeq}
\newcommand{\xto}{\xrightarrow}
\newcommand\xmono{\xhookrightarrow}
\newcommand\mono{\longhookrightarrow}
\newcommand\epi{\longrightarrow\mathrel{\mkern-25mu}\longrightarrow}
\newcommand\xdhrightarrow[2][]{%
  \mathrel{\ooalign{$\xrightarrow[#1\mkern4mu]{#2\mkern4mu}$\cr%
  \hidewidth$\rightarrow\mkern4mu$}}
}
\newcommand\xepi{\xdhrightarrow}
\newcommand\fibration{\epi}
\newcommand\xfibration{\xepi}

\newcommand\union{\cup}

\newcommand\N{\mathbb{N}_0}
\DeclareMathOperator{\id}{id}
\DeclareMathOperator{\cocart}{cocart}
\DeclareMathOperator{\cart}{cart}
\DeclareMathOperator{\Str}{Str}
\DeclareMathOperator{\comp}{comp}
\DeclareMathOperator{\matecorrbetwadjs}{matecorr}

\DeclareMathOperator{\pr}{pr}
\DeclareMathOperator{\incl}{incl}
\DeclareMathOperator{\const}{const}
\DeclareMathOperator{\codisc}{codisc}
\DeclareMathOperator{\disc}{disc}

\DeclareMathOperator{\mate}{mate}

\DeclareMathOperator{\Span}{Span}

\DeclareMathOperator{\ev}{ev}
\DeclareMathOperator{\CoCartFib}{CoCart}

\renewcommand{\epsilon}{\varepsilon}

\renewcommand{\to}{\longrightarrow}
\newcommand{\twoto}{\implies}

\newcommand\op{\text{op}}

\newcommand\adj{\dashv}
\newcommand\CAT{\text{CAT}}

\newcommand\Cat{\text{Cat}}

\newcommand\Set{\text{Set}}

\newcommand\simplexcat{\boldsymbol{\Delta}}

\newcommand\Spaces{\mathcal{S}}

\newcommand\globCSCocart{\text{CSCocart}^{\text{glob}}(\simplexcat^\op)}
\DeclareMathOperator{\LaxSect}{LocRARI}
\DeclareMathOperator{\Env}{Env}

\DeclareMathOperator{\Map}{Map}

\DeclareMathOperator{\Fun}{Fun}

\DeclareMathOperator{\Lax}{Lax}

\newcommand\twoCat{\text{2-Cat}}

\renewcommand\bar{\overline}
\renewcommand\tilde{\reallywidetilde}

\newcommand\laxto{\leadsto}
\newcommand\xlaxto[1]{\overset{#1}{\leadsto}}

\definecolor{diagmarker}{HTML}{FF0000}
\definecolor{diaggray}{HTML}{595959}
\definecolor{nlab}{HTML}{226622}
\definecolor{pastellred}{HTML}{FFB8B1}
\definecolor{pastellgreen}{HTML}{B5EAD6}
\definecolor{otherpastellgreen}{HTML}{00ab41}
\definecolor{pastellorange}{HTML}{FFD6A5}
\definecolor{pastellblue}{HTML}{A0C4FF}
\definecolor{pastelldarkblue}{HTML}{9BAADD}
\definecolor{pastelllila}{HTML}{BDB2FF}
\definecolor{pastellyellow}{HTML}{FDFFB6}
\definecolor{pastellviolett}{HTML}{5C5CD6}

\newcommand{\arrow}{%
\begin{tikzpicture}[diagram]%
    \matrix[objects,wide]{%
        |(a)| \bullet \& |(b)| \bullet \\%
    };%
    \path[maps,->]%
        (a) edge (b)%
    ;%
\end{tikzpicture}%
}

\numberwithin{equation}{subsection}

\bibliography{references}

\usepackage[unicode=true,colorlinks=true,hypertexnames=false]{hyperref}
\hypersetup{allcolors={otherpastellgreen}}

\renewcommand\define[1]{{\color{otherpastellgreen}\underline{#1}}}

\title{A Model Independent Universal Property for the Lax $2$-Functor Classifier}
\author{Johannes Gloßner}
\date{\today}

\begin{document}

\maketitle

\begin{abstract}
    In this article we provide a model-independent definition of the concept of lax $2$-functors
    from $(\infty,2)$-category theory and show that it agrees with the existing and widely 
    used combinatorial model for those in terms of inert-cocartesian functors, 
    which is utilized for example in the foundational work of Gaitsgory and Rozenblyum 
    on Derived Algebraic Geometry to talk about the lax Gray tensor product.
\end{abstract}

\tableofcontents




  \section{Introduction}



  %
  %
  %

    \subsection*{Motivation}


      One of the strengths of (higher) category theory is the abstract formulation of universal properties 
      as well as the ability to work with and reason about them. 
      For a mathematical object to have a simple universal property often helps to reduce the 
      complexity of problems involving this object. In this article we give a model-independent definition of 
      the ubiquitous notion of lax $2$-functors from $(\infty,2)$-category theory 
      by providing such a simple universal property for their classifying objects, 
      not depending on a particular combinatorial model or the use of ordinary $2$-category theory.

      From the viewpoint of classical ordinary $2$-category theory, lax $2$-functors between $2$-categories 
      appear in a lot of places. 
      Monad objects inside a $2$-category $\mathcal{K}$ can be identified with lax functors 
      $\ast \laxto \mathcal{K}$ out of the one-object $2$-category, 
      lax monoidal functors $(\mathcal{V},\otimes) \laxto (\mathcal{W}, \otimes)$ are the same thing 
      as lax functors $B\mathcal{V} \laxto B\mathcal{W}$ between the monoidal categories viewed 
      as one-object $2$-categories $B\mathcal{V}$, respectively $B\mathcal{W}$ and even 
      $\mathcal{V}$-enriched categories with object set $X$ can be seen equivalent to 
      lax functors $\codisc(X) \laxto B\mathcal{V}$, out of the \textit{codiscrete}, sometimes also \textit{chaotic}, 
      $2$-category $\codisc{X}$ on the set $X$.
      On the other hand, lax $2$-functors served as the basis for a vast generalization of the $1$-categorical 
      Grothendieck construction, functors $E \to C$ between ordinary categories having small fibers can be 
      equivalently described as lax $2$-functors $C \laxto \Span(\Set)$ into a certain $2$-category 
      of spans in the ordinary category $\Set$ of small sets.
      Furthermore, the notion of lax functors extends even deeper into the theory of $2$-categories
      as one can identify $2$-functors $A \otimes_{\text{lax}} B \to C$ out of the lax Gray tensor product 
      with special, sometimes called \textit{cubical}, lax functors $A \times B \laxto C$ out 
      of the cartesian product of $2$-categories. 

      Homotopy-coherent implementations of these viewpoints using combinatorial models 
      have allowed us to access these structures in the realm of higher category theory. 
      Lax monoidal functors between monoidal $(\infty,1)$-categories for example are defined 
      through Lurie's model of $\infty$-operads and maps between them, so-called inert-cocartesian functors,
      in \cite[Definition 2.1.2.7]{ha}.
      Homotopy-coherent monads on $(\infty,1)$-categories were at first also modelled as 
      algebras in monoidal $(\infty,1)$-categories of endofunctors in \cite[Definition 4.7.0.1]{ha}.
      In \cite[Definition 2.2.17]{gepnerhaugseng} Gepner and Haugseng model categories 
      enriched over a monoidal $(\infty,1)$-category using a generalization of 
      inert-cocartesian functors between non-symmetric generalized $\infty$-operads.
      In \cite[Theorem 4.1]{blomstraighteningofeveryfunctor} Blom proves the generalised
      straightening-unstraightening equivalence mentioned above in the model of Segal objects in the
      $(\infty,1)$-category of $(\infty,1)$-categories and inert-cocartesian functors 
      between those.

      Prominently, the viewpoint on the lax Gray tensor product presented above is exploited in 
      Gaitsgory and Rozenblyum's work on Derived Algebraic Geometry \cite{GR} to serve as a definition of
      the lax Gray tensor product in a model of $(\infty,2)$-categories.
      Based on the lax Gray tensor product and now proven conjectures about its properties, see \cite{gaitsgoryrozenblyumconjectures} 
      for an overview of the work involved,
      they procede to prove many general $(\infty,2)$-categorical results about 
      constructing, extending and uniquely characterising homotopy-coherent six functor formalisms
      from less coherent input data. 

      A model-independent definition of lax $2$-functors could complement these combinatorial
      descriptions by providing a common ground to compare them and to reason about them 
      and the examples from an abstract point of view.


    \subsection*{Intuition}

    Borrowing intuition from the classical theory of ordinary $2$-categories, 
    lax $2$-functors are supposed to be
    a generalization of the usual notion of 
    $2$-functors between $2$-categories. Usual $2$-functors $F \colon A \to B$ are the 
    most immediate realization of the idea of structure-preserving morphism of $2$-categories. 
    They have for example for each object $a$ of $A$ an identification $\id_{F(a)} \equiv F(\id_a)$,
    as well as for each $1$-morphism composite $g \circ f$ in $A$ an identification
    $F(g) \circ F(f) \equiv F(g \circ f)$, \ie they preserve $1$-morphism identities and compositions 
    up to identification.
    A lax $2$-functor $F$ from A to B, denoted by $F \colon A \laxto B$, on the other hand, 
    relaxes these preservation properties to the structure of not-necessarily invertible $2$-morphisms
    $\delta \colon \id_{F(a)} \twoto F(\id_a)$ and $\gamma \colon F(g) \circ F(f) \twoto F(g \circ f)$,
    but keeps their naturality.
    In order to have these $2$-morphisms still behave in a unital and associative way
    one also needs to retain for example the following identifications of $2$-morphisms
    \[\begin{tikzcd}[cramped,column sep=small]
      {F(a)} & {F(\tilde{a})} & {F(\tilde{a})} & \equiv & {F(a)} & {F(\tilde{a})} & \equiv & {F(a)} & {F(a)} & {F(\tilde{a})}
      \arrow["{F(f)}", from=1-1, to=1-2]
      \arrow[""{name=0, anchor=center, inner sep=0}, "{F(f)}"', curve={height=30pt}, from=1-1, to=1-3]
      \arrow[""{name=1, anchor=center, inner sep=0}, "{F(\id_{\tilde{a}})}"', from=1-2, to=1-3]
      \arrow[""{name=2, anchor=center, inner sep=0}, shift left=3, curve={height=-18pt}, equals, from=1-2, to=1-3]
      \arrow[""{name=3, anchor=center, inner sep=0}, "{F(f)}", curve={height=-12pt}, from=1-5, to=1-6]
      \arrow[""{name=4, anchor=center, inner sep=0}, "{F(f)}"', curve={height=12pt}, from=1-5, to=1-6]
      \arrow[""{name=5, anchor=center, inner sep=0}, "{F(\id_{\tilde{a}})}"', from=1-8, to=1-9]
      \arrow[""{name=6, anchor=center, inner sep=0}, shift left=3, curve={height=-18pt}, equals, from=1-8, to=1-9]
      \arrow[""{name=7, anchor=center, inner sep=0}, "{F(f)}"', curve={height=30pt}, from=1-8, to=1-10]
      \arrow["{F(f)}"', from=1-9, to=1-10]
      \arrow["\gamma"', shorten >=3pt, Rightarrow, from=1-2, to=0]
      \arrow["\delta"', shorten <=3pt, shorten >=3pt, Rightarrow, from=2, to=1]
      \arrow[shorten <=3pt, shorten >=3pt, equals, from=3, to=4]
      \arrow["\delta"', shorten <=3pt, shorten >=3pt, Rightarrow, from=6, to=5]
      \arrow["\gamma"', shorten >=3pt, Rightarrow, from=1-9, to=7]
    \end{tikzcd}\]
    and
    \[\begin{tikzcd}[cramped,column sep=small]
      & {F(a_1)} & {F(a_2)} &&&& {F(a_1)} & {F(a_2)} \\
      {F(a_0)} &&& {F(a_3)} & \equiv & {F(a_0)} &&& {F(a_3)}
      \arrow["{F(g)}", from=1-2, to=1-3]
      \arrow["{F(h)}", from=1-3, to=2-4]
      \arrow["{F(g)}", from=1-7, to=1-8]
      \arrow[""{name=0, anchor=center, inner sep=0}, "{F(h \circ g)}"{description}, curve={height=12pt}, from=1-7, to=2-9]
      \arrow["{F(h)}", from=1-8, to=2-9]
      \arrow["{F(f)}", from=2-1, to=1-2]
      \arrow[""{name=1, anchor=center, inner sep=0}, "{F(g \circ f)}"{description}, curve={height=12pt}, from=2-1, to=1-3]
      \arrow[""{name=2, anchor=center, inner sep=0}, "{F(h \circ (g \circ f))}"', curve={height=24pt}, from=2-1, to=2-4]
      \arrow["{F(f)}", from=2-6, to=1-7]
      \arrow[""{name=3, anchor=center, inner sep=0}, "{F((h \circ g) \circ f)}"', curve={height=24pt}, from=2-6, to=2-9]
      \arrow["\gamma", shorten <=1pt, shorten >=4pt, Rightarrow, from=1-2, to=1]
      \arrow["\gamma", shorten <=14pt, shorten >=10pt, Rightarrow, from=1-3, to=2]
      \arrow["\gamma"', shorten <=14pt, shorten >=10pt, Rightarrow, from=1-7, to=3]
      \arrow["\gamma", shorten <=1pt, shorten >=4pt, Rightarrow, from=1-8, to=0]
    \end{tikzcd}\]
    which are also implicitely part of the ordinary notion of $2$-functor between $2$-categories.
    Let us remark here, that while for ordinary $2$-categories these axioms are enough, they are by 
    no means exhaustive in the realm of $(\infty,2)$-categories, where \eg also higher arity associativity 
    needs to be taken into account as structure.
    Actual $2$-functors of $(\infty,2)$-categories should still be examples of these more general kinds 
    of $2$-functors. 
    In fact, they are then the lax $2$-functors that are \define{strict}, in the sense that the 
    $2$-morphisms $\delta$ and $\gamma$ are invertible.

    In ordinary $2$-category theory, lax $2$-functors can furthermore be classified by certain $2$-functors in 
    the following sense. For every $2$-category $A$ there exists a $2$-category $\Lax(A)$, which we call the 
    \define{lax $2$-functor classifier of $A$}, together 
    with a lax $2$-functor $\iota \colon A \laxto \Lax(A)$ which is supposed to be the initial lax $2$-functor 
    out of $A$, \ie every lax $2$-functor $F \colon A \laxto B$ into another $2$-category $B$ can be 
    factored uniquely as 
    \begin{equation*}
      A \xlaxto{\iota} \Lax(A) \xto{\bar{F}} B
    \end{equation*}
    that is, a strict $2$-functor $\bar{F} \colon \Lax(A) \to B$, precomposed with $A \laxto \Lax(A)$.
    We say that the $2$-functor $\bar{F} \colon \Lax(A) \to B$ \define{classifies} the lax $2$-functor $F$.
    Thus, in order to define the notion of lax $2$-functors it suffices to define the lax $2$-functor
    classifiers and consider strict $2$-functors out of them. As usual $2$-functors should also be 
    considered as lax $2$-functors, they can be classified too. In the case of the identity $2$-functor
    $\id \colon A \to A$ we obtain the strict $2$-functor 
    $\lambda \coloneqq \overline{\id} \colon \Lax(A) \to A$.


    \subsection*{Definition}

    To generalise lax $2$-functors into the realm of $(\infty,2)$-category theory
    we now introduce a model-independent universal property for the $(\infty,2)$-category
    $\Lax(A)$ with respect to its strict $2$-functor $\lambda \colon \Lax(A) \to A$.
    More precisely, it will not depend on combinatorial descriptions or the use 
    of ordinary $2$-category theory, but solely rely on basic abstract features of 
    $(\infty,2)$-categories and $(\infty,1)$-category theory.

    \pagebreak

    \begin{definition}
      \label{definitionlaxfunctorclassifier}
      Let $A$ be a $(\infty,2)$-category. A $2$-functor $F \colon X \to A$ is said to have 
      \define{local right adjoint sections} if it satisfies the following two properties.
      \begin{enumerate}
        \item The induced functor $F^{\equiv} \colon X^{\equiv} \to A^{\equiv}$ on underlying 
          spaces of objects is an equivalence.
        \item For every pair of objects $x, \tilde{x}$ of $X$, the induced functor
          \begin{equation*}
            F_{x, \tilde{x}} \colon X(x, \tilde{x}) \to A(F(x), F(\tilde{x}))
          \end{equation*}
          on Hom-$(\infty,1)$-categories admits a fully-faithful right adjoint.
      \end{enumerate}
      For a $2$-functor 
      \begin{equation}
            \begin{tikzpicture}[diagram]
                \matrix[objects] {%
                  |(a)| X \& \& |(b)| Y \\
                  \& |(c)| A \\
                };
                \path[maps,->]
                (a) edge node[above]  {$H$} (b)
                (a) edge node[below left]   {$F$} (c)
                (b) edge node[below right]  {$G$} (c)
                ;
            \end{tikzpicture}
        \end{equation}
        over $A$ between two such $2$-functors $F$ and $G$ having local right adjoint sections we say that it 
        \define{commutes with the local right adjoint sections} if the induced 
        square
        \begin{equation}
            \begin{tikzpicture}[diagram]
                \matrix[objects] {%
                  |(a)| X(x,\tilde{x}) \& |(b)| Y(Hx,H\tilde{x}) \\
                  |(c)| A(Fx,F\tilde{x}) \& |(d)| A(GHx,GH\tilde{x}) \\
                };
                \path[maps,->]
                (a) edge node[above]  {$H_{x,\tilde{x}}$} (b)
                (c) edge node[below]  {$\equiv$} (d)
                (a) edge node[left]   {$F_{x,\tilde{x}}$} (c)
                (b) edge node[right]  {$G_{Hx,H \tilde{x}}$} (d)
                ;
            \end{tikzpicture}
        \end{equation}
        is vertically right adjointable, 
        or in this case equivalently, the unit of the adjunction involving $G_{Hx, H \tilde{x}}$ 
        becomes invertible after precomposition with $H_{x,\tilde{x}}$ followed by 
        precomposition with the right adjoint of $F_{x,\tilde{x}}$. 
        These two conditions define a non-full subcategory $\LaxSect(A)$ 
        of the strict slice $(\infty,1)$-category $\twoCat_{/ A}$, 
        \ie the $(\infty,1)$-categorical slice over $A$ of the $(\infty,1)$-category $\twoCat$ of $(\infty,2)$-categories.
        A \define{lax $2$-functor classifier of $A$} is an initial object $\lambda \colon \Lax(A) \to A$
        in $\LaxSect(A)$.


    \end{definition}

    To illustrate this definition with an example, let us denote the Hom-local right adjoints of 
    $\lambda_{a,\tilde{a}}$ by $\iota_{a,\tilde{a}}$.
    Then the units of these Hom-adjunctions will provide us with the 
    desired identity comparison $2$-cells
    $\id_a \twoto \iota_{a, a} (\lambda_{a, a} (\id_a)) \equiv \iota_{a,a}(\id_a)$
    and composition comparison $2$-cells 
    \begin{equation*}
      \iota_{a_1, a_2} (g) \circ \iota_{a_0,a_1}(f) \twoto \iota_{a_0,a_2}(\lambda_{a_0,a_2}(\iota_{a_1, a_2} (g) \circ \iota_{a_0,a_1}(f))) \equiv \iota_{a_0,a_2}(g \circ f) .
    \end{equation*}

    \subsection*{Comparison to a combinatorial model}

    The main goal of this article is to review an existing model-dependent 
    definition of lax functors, discussed for example in \cite[Chapter 10.3]{GR}, 
    as well as the construction of their classifiers, and to show that 
    these classifiers satisfy our universal property.
    More precisely, the model is obtained by starting with the so-called complete $2$-fold Segal space model 
    for $(\infty,2)$-categories from \cite{barwickschommerpries}, 
    and then unstraightening these 
    complete Segal objects $X$
    in the $(\infty,1)$-category of small $(\infty,1)$-categories
    with $X_0$ a space to 
    cocartesian fibrations over $\simplexcat^\op$.
    In this model, the cocartesian functors play the role of strict $2$-functors.
    By restricting the class of morphisms in $\simplexcat^\op$ to the so-called \define{inert} simplex maps,
    one obtains a definition of lax functors between globular complete Segal cocartesian fibrations by 
    considering functors over $\simplexcat^\op$ that are only required to preserve cocartesian lifts
    along the inert simplex maps. 
    Inert-cocartesian functors out of an $(\infty,2)$-category $A$, modelled this way, also admit an explicitely 
    constructed classifier, the \define{envelope} $\Env(A)$ of $A$, together with a distinguished 
    strict $2$-functor $\lambda \colon \Env(A) \to A$ classifying the identity inert-cocartesian functor on $A$.
    This also already appeared in the literature for example as \cite{AMGR}.
    Our main goal is then made precise by \autoref{theoremenvelopeislaxfunctorclassifier}, 
    or, to put it shortly, by

    \begin{theorem*}
      The $2$-functor $\lambda \colon \Env(A) \to A$ is a lax functor classifier in the sense of 
      \autoref{definitionlaxfunctorclassifier}.
    \end{theorem*}


    \subsection*{Foundations}
    
    As the heart of this article is the model-independent formulation of a universal property of an 
    $(\infty,2)$-categorical notion it is desirable to make precise on which model-independent aspects 
    of $(\infty,2)$-category theory it relies on.
    Furthermore, this also entails the need to be model-independet in our use of $(\infty,1)$-category theory. 
    To this end, we kept the interpretation of the aforementioned model of $(\infty,2)$-categories and 
    lax $2$-functors between them, as well as the proofs involved in attaining 
    \autoref{theoremenvelopeislaxfunctorclassifier} general and model-independent enough 
    so that they will be intepretable, \ie make sense, in any \textit{synthetic} $(\infty,1)$-category theory, 
    such as the ones that are being developed by Cisinski, Cnossen, Nguyen and Walde in their book project 
    \cite{formalizationbookproject}, or by Riehl and Shulman in \cite{riehlshulman}, 
    and by Buchholtz and Weinberger in \cite{buchholtzweinberger}.
    To this end, for $1$-categories, we want to avoid specific set-theoretic implementations, 
    such as model structures, and combinatorial descriptions of $1$-categories, 
    like complete Segal spaces, and the use of strict $1$-categories, 
    except for the $n$-simplices $\Delta^n$, \ie the finite non-empty totally ordered sets.
    Synthetic $(\infty,1)$-category theory however means even more than this. 
    It entails that our use of $(\infty,1)$-category theory will be general enough 
    to be also interpretable in the theory of internal higher categories in any $(\infty,1)$-topos, 
    as for example developed by Martini and Wolf in for example \cite{Martini2021}, \cite{MW2021}.
    We refer the reader to 
    \autoref{sectioninertcocartfunsandtheenvelopeconstr} for 
    more details on the relevant model-independent aspects of 
    $(\infty,2)$-category theory
    and to the \autoref{appendixbackgroundononecategorytheory}
    for a more in-depth and thorough development of the necessary $(\infty,1)$-categorical statements
    involved in the proof of \autoref{theoremenvelopeislaxfunctorclassifier}.

  \subsection*{Organisation of the article}



    In \autoref{sectioninertcocartfunsandtheenvelopeconstr} 
    we start by recalling the aforementioned model of $(\infty,2)$-categories 
    then make precise in what sense 
    the above introduced universal property of the lax $2$-functor classifier is actually
    model-independent and on what aspects of $(\infty,2)$-category theory it relies on.
    Afterwards, we introduce the tractable model of lax $2$-functors therein, given by the inert-cocartesian functors.
    Then we proceed to give the explicit classifier construction that exists for these
    inert-cocartesian functors and discuss some of its relevant properties.

    In \autoref{sectiontheenvelopesatisfiestheUPofthelaxfunclassifier} 
    we start by relating  the model-independent notions of \autoref{definitionlaxfunctorclassifier}, \ie
    of $2$-functors having local right adjoint sections and $2$-functors commuting with them, 
    to more model-dependent counterparts
    and then finally prove \autoref{theoremenvelopeislaxfunctorclassifier}.

    In the \autoref{appendixbackgroundononecategorytheory} 
    we prove the relevant statements from $(\infty,1)$-category theory that the other sections
    rely on by developing the necessary background material on adjunctions and cocartesian fibrations in more detail,
    starting with the definitions.

\section*{Acknowledgements}

  The author gratefully acknowledges support from the SFB 1085 Higher Invariants in Regensburg, funded by the DFG.
  The author wishes to thank his advisor Denis-Charles Cisinski for his ongoing support, patience
  and the opportunity to explore where curiosity leads.
  The author wishes to thank Benjamin Dünzinger and Bastiaan Cnossen for their ongoing help, patience and 
  willingness to listen to even the strangest ideas.

  \section*{Conventions and Notation}
    \label{sectionconventionsandfoundations}

    We will refer to $(\infty,1)$-categories as just \textit{$1$-categories}, or even just \textit{categories}.
    Furthermore we will refer to $(\infty,2)$-categories as just \textit{$2$-categories}.




    \begin{notation}
      We write $\Fun(A,B)$ for the category of functors from $A$ to $B$. 
      We will denote natural transformations $\alpha$ from a functor $F$ to a functor $G$ by 
      $\alpha \colon F \twoto G$ to distinguish them from ordinary morphisms.
      For a functor $F \colon A \to B$ we write $F_\ast \colon \Fun(T,A) \to \Fun(T,B)$ for the 
      postcomposition with $F$ and ${}_\ast F \colon \Fun(B,T) \to \Fun(A,T)$
      for the precomposition with $F$.
      We write $\Delta_A \colon A \to \Fun(\Delta^1,A)$ for the precomposition functor ${}_\ast (\Delta^1 \to \Delta^0)$,
      and $\ev_0 , \ev_1 \colon \Fun(\Delta^1,A) \to A$ for the domain and the codomain functor, respectively.
      We denote the ($\infty$-)category of small $1$-categories by $\Cat$, its full subcategory of small spaces, 
      also called $\infty$-groupoids, by $\Spaces$ 
      and its full subcategory of simplices, \ie non-empty finite totally ordered sets, by $\simplexcat$.
      We write $(-)^\equiv \colon \Cat \to \Spaces$ for the right adjoint to the inclusion functor 
      $\Spaces \mono \Cat$, so that for a category $A$ we get its underlying space of objects 
      denoted by $A^\equiv$.
      We denote by $\Map(A,B) \coloneqq \Fun(A,B)^\equiv$ the space of functors from $A$ to $B$.
      For a cospan of functors $f \colon A \to C \longleftarrow B \colon g$ we sometimes write its pullback as 
      $A \times^{(f,g)}_C B$ to emphasise the dependence on $f$ and $g$, 
      in contrast to the notation $A \times_C B$.
      When a functor $F \colon A \to B$ is left adjoint to a functor $U \colon B \to A$ 
      we denote this by $F \adj U$.
      Functors $X$ to $B$ that are cocartesian or cartesian fibrations will be denoted by $X \fibration B$.
      For a category $A$ we write $\comp \colon \Fun(\Delta^1,A) \times_A \Fun(\Delta^1,A) \to \Fun(\Delta^1,A)$ 
      for the composition functor of its morphisms.
      For a category $A$ and on object $a$ of $A$ we write $A_{/ a}$ for the slice category over $a$ 
      and ${}^{a /} \hspace{-0.3em} A$ for the slice category under $a$.
      For two functors $p \colon A \to I$ and $q \colon B \to I$, seen as objects in $\Cat_{/ I}$,
      we define the relative functor category over $I$ from $(A,p)$ to $(B,q)$ via the following
      pullback square.
      \[\begin{tikzcd}[cramped]
        {\Fun_{/ I}((A,p),(B,q))} & {\Fun(A,B)} \\
        {\Delta^0} & {\Fun(A,I)}
        \arrow[from=1-1, to=1-2]
        \arrow[from=1-1, to=2-1]
        \arrow["\lrcorner"{anchor=center, pos=0.125}, draw=none, from=1-1, to=2-2]
        \arrow["{q_\ast}", from=1-2, to=2-2]
        \arrow["p"', from=2-1, to=2-2]
      \end{tikzcd}\]
    \end{notation}


  \section{Inert-Cocartesian Functors and the Envelope Construction}
    \label{sectioninertcocartfunsandtheenvelopeconstr}

    In this section we will start with a recollection of a particular model of $2$-categories, 
    one closely related to the notion of complete $2$-fold Segal space, which is also the one used 
    in \cite[Appendix on $(\infty,2)$-categories]{GR}.
    Afterwards, we will introduce a model of lax $2$-functors for this model of $2$-categories, 
    which is the one used in \cite[Chapter 10.3]{GR} to define and work with the lax Gray tensor product.
    For this model there exists an explicit construction of its classifying object, 
    for which we will later prove the model-independent universal property.
    The construction of the classifying objects in this generality already appeared for example in 
    \cite{gepnerhaugseng}.
    Lastly we will show that this classifier exhibits several properties which are very similar to 
    the ones used in \autoref{definitionlaxfunctorclassifier} 
    of a $2$-functor having local right adjoint sections.
    All these statements can also be found in \cite[Chapter 11.A]{GR}.

    \subsection{Globular Complete Segal Cocartesian Fibrations over $\simplexcat^\op$ and Inert-Cocartesian Functors}
      \label{recollectiontwocatsasglobularcompletesegalfibrations}
      It is a classical result,
      for example discussed in \cite{barwickschommerpries},
      that one can model $2$-categories as complete $2$-fold Segal spaces, 
      defined as certain bisimplicial objects in spaces. For our purposes it is actually more convenient to use the equivalent 
      formulation as simplicial objects $X \colon \simplexcat^\op \to \Cat$ in the 
      category $\Cat$ of (small) $1$-categories, which are
      Segal, complete and \define{globular}, \ie satisfy the extra 
      assumption that $X_0$ is a space. The corresponding notion of 
      $2$-functors in this model is simply given by natural transformations.

      Using the Straightening-Unstraightening equivalence 
      \begin{equation*}
        \Fun(\simplexcat^\op, \Cat) \equiv \CoCartFib(\simplexcat^\op)
      \end{equation*}
      between functors to $\Cat$ and cocartesian fibrations we can alternatively
      describe these special simplicial objects in $\Cat$ as cocartesian fibrations 
      $X \fibration \simplexcat^\op$ over the category $\simplexcat^\op$ 
      which satisfy equivalently expressed Segal axioms, completeness axioms
      and satisfy that the fiber $X_0$ over $\Delta^0$ is a space.
      Let us denote this category by $\globCSCocart$ for later reference.
      Here, the corresponding role of $2$-functors is played by the cocartesian functors 
      between such globular complete Segal cocartesian fibrations over $\simplexcat^\op$.

      To give some intuition, let us take a globular complete Segal object $X \colon \simplexcat^\op \to \Cat$,
      or equivalently unstraightened as a cocartesian fibration $p \colon X \fibration \simplexcat^\op$.
      We interpret the space $X_0$ as the space of objects $X^\equiv$ of the $2$-category $X$.
      We think of $X_1$ as the category of $1$-morphisms and $2$-cells between those,
      and the two evaluations functors $X_0 \leftarrow X_1 \rightarrow X_0$ as the source and target 
      functors
      \ie the fibers of $(s,t) \colon X_1 \to X_0 \times X_0$ over a pair of objects $(a,b)$ should exactly model the Hom-$1$-categories
      $X(a,b)$ varying appropriately in both variables over the space of objects $X^\equiv \coloneqq X_0$.
      The functor $X_0 \to X_1$ maps objects to their identity $1$-morphisms.
      By the Segal axiom we can identify $X_2 \equiv X_1 \times_{X_0} X_1$.
      In the fibrational viewpoint the composition of two $1$-morphisms is then given by 
      cocartesian pushforward along $(0,2) \colon \Delta^1 \mono \Delta^2$ in $\simplexcat^\op$.

    We now continue by reviewing what we actually need to know about $2$-categories and 
    $2$-category theory in order to model-independently define our universal property 
    \autoref{definitionlaxfunctorclassifier}.
    Essentially, we require the existence of the $1$-category $\twoCat$ of $2$-categories,
    and our universal property will be stated entirely in terms of basic properties of $\twoCat$, that do 
    not depend on how it was defined.

    \begin{remark}
      We will only use the following basic ingredients from $2$-category theory.
      Note that this list is by no means exhaustive, \ie it does not intend to capture 
      all of $2$-category theory, but it should be a minimal subset with respect to 
      \autoref{definitionlaxfunctorclassifier} that is interpretable 
      in any model of $2$-category theory, not just the above introduced complete $2$-fold Segal spaces 
      or the globular complete Segal cocartesian fibrations over $\simplexcat^\op$.

      \begin{enumerate}

        \item We have a $1$-category $\twoCat$ of all (small) $2$-categories, (strict) $2$-functors between them,
        $2$-natural isomorphisms between those and so on.

        \item Every $2$-category $A$ has an \define{underlying space of objects $A^\equiv$}. More precisely,
          this assignment should be functorial and 
          $(-)^\equiv \colon \twoCat \to \Spaces$ is supposed to be the right adjoint to the 
          fully-faithful inclusion $\Spaces \mono \twoCat$ of spaces into
          $2$-categories.
          In particular we get that every (strict) $2$-functor $F \colon A \to B$ induces a
          \define{functor on underlying spaces of objects} $F^\equiv \colon A^\equiv \to B^\equiv$.

        \item Every $2$-category $A$ has \define{Hom-$1$-categories}, parametrized by its space of 
          objects 
          \begin{equation*}
            A(- , - ) \colon (A^\equiv)^\op \times A^\equiv \to \Cat .
          \end{equation*}
          For our purposes, we do not need further functoriality in non-invertible $1$- or $2$-morphisms.


        \item Every (strict) $2$-functor $F \colon A \to B$ induces 
          \define{$1$-functors on the Hom-$1$-categories}, natural in the space of objects of $A$, 
          $F_{- , - } \colon A(- , - ) \twoto B(F^\equiv (- ), F^\equiv (- ))$. We will usually abuse notation and just write $B(F(-),F(-))$ for the codomain of these functors.

        \item For every triangle of $2$-functors 
          \begin{equation}
            \begin{tikzpicture}[diagram]
                \matrix[objects] {%
                  |(a)| A \& \& |(b)| B \\
                  \& |(c)| C \\
                };
                \path[maps,->]
                (a) edge node[above]  {$H$} (b)
                (a) edge node[below left]   {$F$} (c)
                (b) edge node[below right]  {$G$} (c)
                ;
            \end{tikzpicture}
        \end{equation}
        commuting by a $2$-natural isomorphism we already know that we get a commutative triangle of functors on 
        object spaces
        \begin{equation}
            \begin{tikzpicture}[diagram]
                \matrix[objects] {%
                  |(a)| A^\equiv \& \& |(b)| B^\equiv \\
                  \& |(c)| C^\equiv \\
                };
                \path[maps,->]
                (a) edge node[above]  {$H$} (b)
                (a) edge node[below left]   {$F$} (c)
                (b) edge node[below right]  {$G$} (c)
                ;
            \end{tikzpicture}
        \end{equation}
        and we want this to also extend to commutative squares of actions on Hom-categories
        \begin{equation}
            \begin{tikzpicture}[diagram]
                \matrix[objects] {%
                  |(a)| A(a,\tilde{a}) \& |(b)| B(Ha,H\tilde{a}) \\
                  |(c)| C(Fa,F\tilde{a}) \& |(d)| C(GHa,GH\tilde{a}) \\
                };
                \path[maps,->]
                (a) edge node[above]  {$H$} (b)
                (c) edge node[below]  {$\equiv$} (d)
                (a) edge node[left]   {$F$} (c)
                (b) edge node[right]  {$G$} (d)
                ;
            \end{tikzpicture}
        \end{equation}
        natural in the space of pairs $(a,\tilde{a})$ of objects of $A$, 
        more precisely in space $(A^\equiv)^\op \times A^\equiv$.

      \item Additionally, the last three points can also be made coherent by asking for a section 
        of the large categorical presheaf 
        $\Fun(((-)^\equiv)^\op \times (-)^\equiv, \Cat) \colon \twoCat^\op \to \CAT$,
        although we will not need this further coherence.

      \end{enumerate}

      Note in particular, that these ingredients are agnostic to whether one implements $2$-categories
      as so-called $\Cat$-enriched categories or vertically trivial and complete internal category objects in $\Cat$, 
      which, in a combinatorial way we do here.
      So the model-independent universal property of the lax $2$-functor classifier of 
      \autoref{definitionlaxfunctorclassifier} 
      will also be agnostic to these choices.
    \end{remark}

    Since our main theorem is the comparison of the model-independent universal property
    \autoref{definitionlaxfunctorclassifier}
    and the aforementioned model-dependent definition of lax $2$-functors 
    between unstraightened globular complete Segal objects in $\Cat$,
    we will need to instantiate the above list of $2$-categorical ingredients in this model. 
    Nevertheless for the proof of our main theorem we still avoid specific set-theoretic implementations, 
    such as model structures, combinatorial descriptions of $1$-categories, 
    like complete Segal spaces, and the use of ordinary $1$-categories.

    \begin{proof}
      \begin{enumerate}
        \item As outlined in \autoref{recollectiontwocatsasglobularcompletesegalfibrations}, 
          we obtain a $1$-category of our combinatorial model of $2$-categories by 
          starting with the functor $1$-category $\Fun(\simplexcat^\op, \Cat)$ and take its full 
          subcategory of complete Segal objects $X$ such that $X_0$ is a space.
          By unstraightening this is equivalent to $\globCSCocart$.

        \item In this model the adjunction for underlying spaces of objects is obtained by 
          restricting the adjunction 
          \begin{equation}
            \begin{tikzpicture}[diagram]
                \matrix[objects] {%
                |(b)| \Cat \& |(a)| \Fun(\simplexcat^\op,\Cat) \\
                };
                \path[maps,->]
                (a) edge[bend left] node[below] (F) {$\ev_0$} (b)
                ;
                \path[maps,{Hooks[right]}->]
                (b) edge[bend left] node[above] (U) {$\disc$} (a)
                ;
                \node[rotate=-90] at ($ (U) ! 0.5 ! (F) $) {$\adj$};
            \end{tikzpicture}
          \end{equation}
          to the left hand side to the full subcategory on spaces and the right hand side to the 
          full subcategory on globular complete Segal objects. This means we define $X^\equiv \coloneqq X_0$.

        \item For the Hom-$1$-categories and their functoriality we start by looking at 
          the cospan $\Delta^0 \xto{0} \Delta^1 \xleftarrow{1} \Delta^0$ in $\simplexcat$
          and by applying $\Fun((-)^\op, \Cat)$ get 
          \begin{equation*}
            \Fun(\simplexcat^\op,\Cat) \times^{(\ev_0,\incl)}_{\Cat} \Spaces \xto{\ev_0 \leftarrow \ev_1 \rightarrow \ev_0} \Spaces \times_{\Cat} \Fun(\{ 2 \leftarrow 0 \rightarrow 1 \} , \Cat) \times_{\Cat} \Spaces
          \end{equation*}
          which takes a simplicial object $X \colon \simplexcat^\op \to \Cat$ with $X_0$ a space and 
          associates to it its source-target span $X^\equiv \coloneqq X_0 \xleftarrow{s} X_1 \xto{t} X_0 \eqqcolon X^\equiv$.
          To turn this into the shape of an actual Hom-functor parametrised by $(X^\equiv)^\op \times X^\equiv$
          we proceed in the following way.
          The codomain of this functor, \ie the category of spans $Y \leftarrow C \rightarrow Z$ with $Y$ and $Z$ spaces, 
          is itself equivalent to the total category of the unstraightening of 
          the functor 
          \begin{equation*}
            \Cat_{/ (-)_1 \times (-)_2 } \colon \Spaces \times \Spaces \xto{\times} \Spaces \mono \Cat \xto{\Cat_{/ (-)}} \CAT
          \end{equation*}
          To this we first apply the cartesian straightening equivalence $\Str^{\cart}$, in the special case over spaces $Y$, and then apply the cocartesian straightening , 
          also in the special case over spaces $Z$, \ie 
          \begin{equation*}
            \Str^{\cart}_Y \colon \Cat_{/ Y} \equiv \Fun(Y^\op, \Cat) \qquad \text{and} \qquad \Str^{\cocart}_Z \colon \Cat_{/ Z} \equiv \Fun(Z, \Cat)
          \end{equation*}
          to get the chain of equivalences
          \[\begin{tikzcd}[cramped, column sep=large]
            {\Cat_{/ (-)_1 \times (-)_2 }} & {\Fun((-)_{1}^{\op} , \Cat_{/ (-)_2 })} \\
            {(\Cat_{/ (-)_1 })_{/ (\pr_1 \colon (-)_1 \times (-)_2 \to (-)_1)}} & {\Fun((-)_{1}^{\op} , \Fun((-)_2 , \Cat))} \\
            {\Fun((-)_1^{\op} , \Cat )_{/ \const_{(-)_2} }} & {\Fun((-)_{1}^{\op} \times (-)_2 , \Cat))}
            \arrow["\equiv"', from=1-1, to=2-1]
            \arrow["{( \Str^{\cocart} )_{\ast}}", from=1-2, to=2-2]
            \arrow["{(\Str^{\cart} )_{/ \pr_1}}"', from=2-1, to=3-1]
            \arrow["\equiv", from=2-2, to=3-2]
            \arrow["\equiv"{description}, from=3-1, to=1-2]
          \end{tikzcd}\]
          that realize the desired straightening to functors into $\Cat$ with the correct variances.

      \end{enumerate}
    \end{proof}

    Next we review a model of lax functors between $2$-categories modelled as globular complete Segal
    cocartesian fibrations over $\simplexcat^\op$. This is done by relaxing the cocartesian lift preservation 
    of functors over $\simplexcat^\op$ between such fibrations. 
    For this we will now define two special subclasses of morphisms in $\simplexcat$.

    \begin{definition}
      Let $\alpha \colon \Delta^n \to \Delta^m$ be a simplex morphism.
      \begin{enumerate}
        \item $\alpha$ is called 
          \define{active} if $\alpha(0) = 0$ and $\alpha(n) = m$.
        \item $\alpha$ is called 
          \define{inert} if $\forall i \in \{0,\dots,n-1 \} \colon \alpha(i + 1) = \alpha(i) + 1$.
      \end{enumerate}
    \end{definition}

    \begin{remark}
      At least under sufficient knowledge on how to characterise functors into simplices $\Delta^n$,
      these two classes of morphisms constitute a factorisation system on $\simplexcat$.
      For example if we know that $\Map(C,\Delta^n)$ is actually equivalent to the subspace of $\Map(C^\equiv, \{0,\dots,n \})$
      on those functors $F \colon C^\equiv \to \{0,\dots, n \}$ that satisfy for all morphisms 
      $f \colon c \to \tilde{c}$ in $C$ that $F(c) \leq F(\tilde{c})$ with respect to the poset structure 
      $\{0 < \dots < n \}$ on $\{0,\dots, n \}$, one can deduce the factorisation system.
      Explicitely, for a simplex morphism $\alpha \colon \Delta^n \to \Delta^m$ we 
      can then factorise it as 
      \begin{equation*}
        \Delta^n \xto{\text{active}} \Delta^{\{\alpha(0),\dots,\alpha(n)\}} \xmono{\text{inert}} \Delta^m
      \end{equation*}
      where we denoted by $\Delta^{\{\alpha(0),\dots,\alpha(n)\}}$ the convex subsimplex of $\Delta^m$ 
      spanned by the vertices of $\Delta^m$ between $\alpha(0)$ and $\alpha(n)$.
    \end{remark}

    A model of lax functors in this situation is now given by looking at functors over $\simplexcat^\op$ 
    that are only required to preserve cocartesian lifts along (opposites of) inert morphisms.

    \begin{recollection}
      As expressed in \cite[Chapter 10.3]{GR}, this unstraightened model of $2$-categories now supports a very tractable
      model of lax $2$-functors, the \define{inert-cocartesian functors}, 
      \ie the functors between cocartesian fibrations over $\simplexcat^\op$ 
      \begin{equation}
          \begin{tikzpicture}[diagram]
              \matrix[objects] {%
                |(a)| A \& \& |(b)| B \\
              \& |(c)| \simplexcat^\op \\
              };
              \path[maps,->]
              (a) edge node[above]  {$$} (b)
              (a) edge[->>] node[left]   {$$} (c)
              (b) edge[->>] node[right]  {$$} (c)
              ;
          \end{tikzpicture}
      \end{equation}
      which preserve only the cocartesianness of cocartesian 
      lifts of inert simplex morphisms, or more compactly, which are cocartesian functors after pulling
      back along the non-full subcategory inclusion $\simplexcat_{\text{inert}}^\op \mono \simplexcat^\op$ on inert
      simplex morphisms.
      For this model of lax $2$-functors there is also an explicit classifier, generalizing the 
      monoidal envelope construction as for example in \cite[Section 2.2.4]{ha}.
    \end{recollection}

    \begin{example}
      The simplest examples of simplex morphisms which are not inert are $\Delta^1 \to \Delta^0$ and 
      the inclusion $(0,2) \colon \Delta^1 \mono \Delta^2$.
      So for an inert-cocartesian functor 
      \begin{equation}
          \begin{tikzpicture}[diagram]
              \matrix[objects] {%
                |(a)| A \& \& |(b)| B \\
              \& |(c)| \simplexcat^\op \\
              };
              \path[maps,->]
              (a) edge node[above]  {$F$} (b)
              (a) edge[->>] node[left]   {$$} (c)
              (b) edge[->>] node[right]  {$$} (c)
              ;
          \end{tikzpicture}
      \end{equation}
      between $2$-categories the fact that $F$ does not necessarily preserve cocartesian lifts along these morphisms 
      gives us in the case of $\Delta^1 \to \Delta^0$ a not necessarily invertible comparison morphisms
      \begin{equation*}
        \id_{F(x)} = (\Delta^1 \rightarrow \Delta^0)_! F(x) \to F((\Delta^1 \rightarrow \Delta^0)_! x) = F(\id_x)
      \end{equation*}
      and in the case for $(0,2) \colon \Delta^1 \mono \Delta^2$
      not necessarily invertible comparison morphisms
      \begin{equation*}
        F(g) \circ F(f) = (0,2)_! F(f,g) = (0,2)_! (F(f),F(g)) \to F((0,2)_! (f,g)) = F(g \circ f)
      \end{equation*}
      in the category $X_1$, which we interpret as $2$-cells between $1$-morphisms.
    \end{example}

    \subsection{The Envelope Construction and its Fundamental Properties}
    We now turn our attention to an explicit construction of classifying objects for this model of 
    lax functors.
    As not only the definition of inert-cocartesian functors but also their classifier also work in 
    the greater generality of just Segal cocartesian fibrations over $\simplexcat^\op$ (which in turn are
    models for double $\infty$-categories). 
    We will also adopt this generality for the following construction and statements when applicable.

    \begin{definition}
      Let us denote by $\Fun^{\text{active}}(\Delta^1,\simplexcat^\op)$ the full subcategory of 
      $\Fun(\Delta^1,\simplexcat^\op)$ on all the functors which correspond to active simplex morphisms.
      Let $p \colon A \fibration \simplexcat^\op$ be a Segal cocartesian fibration.
      Then its \define{envelope construction} is defined to be functor
      \begin{equation*}
        \ev_1 \colon \Env(A) \coloneqq  A \times_{\simplexcat^\op} \Fun^{\text{active}}(\Delta^1,\simplexcat^\op) \to \simplexcat^\op .
      \end{equation*}
    \end{definition}

    \begin{proposition}
      \label{envelopeissegalcocartfib}
      Let $p \colon A \fibration \simplexcat^\op$ be a Segal cocartesian fibration.
      Then its envelope $\ev_1 \colon \Env(A) \to A$
      is again a cocartesian fibration and Segal.
    \end{proposition}
    \begin{reference}

      This already appeared as \cite[Proposition A.1.2]{gepnerhaugseng}.
    \end{reference}

    The following proposition will introduce the model of the, soon to be proven initial, lax functor 
    $\iota \colon A \laxto \Lax(A)$ out of $A$, as well as establish completeness of the 
    envelope construction, even in the generality of starting with a  
    double $\infty$-category, which is modelled here by a Segal cocartesian fibration.

    \begin{proposition}
      Let $p \colon A \fibration \simplexcat^\op$ be a Segal cocartesian fibration.
      Then the envelope construction 
      \begin{equation*}
        \ev_1 \colon A \times_{\simplexcat^\op} \Fun^{\text{active}}(\Delta^1,\simplexcat^\op) \to \simplexcat^\op
      \end{equation*}
      is always also complete.
      Furthermore, the  functor 
      $\iota$ defined as
      \begin{equation*}
        A \equiv A \times_{\simplexcat^\op} \simplexcat^\op \xto{\id_A \times_{\id} \Delta_{(\simplexcat^\op)}} A \times_{\simplexcat^\op} \Fun^{\text{active}}(\Delta^1,\simplexcat^\op) 
      \end{equation*}
      is an inert-cocartesian functor
      over $\simplexcat^\op$ and an equivalence on fibers over $\Delta^0$. 
      Hence if $A_0$ is a space, then so is $(A \times_{\simplexcat^\op} \Fun^{\text{active}}(\Delta^1,\simplexcat^\op))_0$
      In particular, if $p \colon A \fibration \simplexcat^\op$ models a $2$-category in 
      the sense that it is Segal, complete  and globular, \ie 
      has as fiber over $\Delta^0$ a space, then so does its envelope
      $\ev_1 \colon A \times_{\simplexcat^\op} \Fun^{\text{active}}(\Delta^1,\simplexcat^\op) \to \simplexcat^\op$.
    \end{proposition}
    \begin{proof}
      \begin{enumerate}
        \item To check completeness we first need to know how cocartesian lifting in 
          \begin{equation*}
            \ev_1 \colon A \times_{\simplexcat^\op} \Fun^{\text{active}}(\Delta^1,\simplexcat^\op) \to \simplexcat^\op
          \end{equation*}
          works.
          We extract the necessary parts for this from the aforementioned \cite[Proposition A.1.2]{gepnerhaugseng}
          to make it explicit now. Another excellent reference for the cocartesian lifting of these kinds of fibrations is 
          \cite[Proposition A.0.1]{AMGR}. Let us start with an arbitrary object in 
          $A \times_{\simplexcat^\op} \Fun^{\text{active}}(\Delta^1,\simplexcat^\op)$
          in the fiber over an arbitrary $\Delta^n$, \ie a tuple consisting of an active 
          simplex morphism $\alpha \colon \Delta^n \to \Delta^m$ in $\simplexcat$, viewed as a 
          morphism $\Delta^m \to \Delta^n$ in $\simplexcat^\op$, and an object $\sigma$ in the 
          fiber over $\Delta^m$. Let us furthermore be given another arbitrary simplex
          morphism $\gamma \colon \Delta^k \to \Delta^n$, viewed as a morphism
          $\Delta^n \to \Delta^k$ in $\simplexcat^\op$, along which we want to 
          cocartesian pushforward. From now on we will only talk about simplex morphisms,
          \ie morphisms in $\simplexcat$, but will implicitely use them as morphisms in 
          $\simplexcat^\op$. The first step is to factorize the composite $\alpha \circ \gamma$ 
          into an active followed by an inert simplex morphism, 
          \begin{equation}
              \begin{tikzpicture}[diagram]
                  \matrix[objects] {%
                  |(a)| \Delta^k \& |(b)| \Delta^l \\
                  |(c)| \Delta^n \& |(d)| \Delta^m \\
                  };
                  \path[maps,->]
                  (a) edge node[above]  {$\beta$} (b)
                  (a) edge node[below]  {$\text{active}$} (b)
                  (c) edge node[below]  {$\alpha$} (d)
                  (c) edge node[above]  {$\text{active}$} (d)
                  (a) edge node[left]   {$\gamma$} (c)
                  (b) edge node[right]  {$\theta$} (d)
                  (b) edge node[left]  {$\text{inert}$} (d)
                  ;
              \end{tikzpicture}
          \end{equation}
          using the corresponding unique factorization system.
          Then the $\ev_1$-cocartesian pushforward of $(\alpha, \sigma)$ is exactly
          $(\beta, \theta_! \sigma)$, where $\theta_! \sigma$ is the 
          $p$-cocartesian pushforward of the object $\sigma$ along the morphism $\theta$.
          For our discussion of completeness let us take an arbitrary
          $2$-simplex $(\alpha, \sigma)$ in 
          $A \times_{\simplexcat^\op} \Fun^{\text{active}}(\Delta^1,\simplexcat^\op)$
          such that its cocartesian pushforward along $(0,2) \colon \Delta^1 \to \Delta^2$ is 
          degenerated from a $0$-simplex $(\id_{\Delta^0},a)$. Note that the 
          identity of $\Delta^0$ is the only active map out of $\Delta^0$.

          This means explicitely that for the unique active-inert factorization
          \begin{equation}
              \begin{tikzpicture}[diagram]
                  \matrix[objects] {%
                  |(a)| \Delta^1 \& |(b)| \Delta^m \\
                  |(c)| \Delta^2 \& |(d)| \Delta^m \\
                  };
                  \path[maps,->]
                  (a) edge node[above]  {$(\alpha(0), \alpha(2))$} (b)
                  (a) edge node[below]  {$\text{active}$} (b)
                  (c) edge node[below]  {$\alpha$} (d)
                  (c) edge node[above]  {$\text{active}$} (d)
                  (a) edge node[left]   {$(0,2)$} (c)
                  (b) edge node[right]  {$\id$} (d)
                  (b) edge node[left]  {$\text{inert}$} (d)
                  ;
              \end{tikzpicture}
          \end{equation}
          and the $p$-cocartesian pushforward $\id_! \sigma = \sigma$,
          we have that the cocartesian pushforwarded tuple 
          $((\alpha(0),\alpha(2)), \sigma)$ is
          equivalent to the degeneration of $(\id_{\Delta^0},a)$, \ie
          via the factorization
          \begin{equation}
              \begin{tikzpicture}[diagram]
                  \matrix[objects] {%
                  |(a)| \Delta^1 \& |(b)| \Delta^0 \\
                  |(c)| \Delta^0 \& |(d)| \Delta^0 \\
                  };
                  \path[maps,->]
                  (a) edge node[above]  {$$} (b)
                  (a) edge node[below]  {$\text{active}$} (b)
                  (c) edge node[below]  {$\id$} (d)
                  (c) edge node[above]  {$\text{active}$} (d)
                  (a) edge node[left]   {$$} (c)
                  (b) edge node[right]  {$\id$} (d)
                  (b) edge node[left]  {$\text{inert}$} (d)
                  ;
              \end{tikzpicture}
          \end{equation}
          so the tuple $(\Delta^1 \to \Delta^0, a)$.
          But this means on the one hand side that 
          \begin{equation}
              \begin{tikzpicture}[diagram]
                  \matrix[objects] {%
                  |(a)| \Delta^1 \& |(b)| \Delta^m \\
                  |(c)| \& |(d)| \Delta^0 \\
                  };
                  \path[maps,->]
                  (a) edge node[above]  {$(\alpha(0), \alpha(2))$} (b)
                  (a) edge node[left]   {$$} (d)
                  (b) edge node[right]  {$\equiv$} (d)
                  ;
              \end{tikzpicture}
          \end{equation}
          \ie that $m = 0$, and that on the other hand side $\sigma \equiv a$.
          But put together this says that the $2$-simplex $(\alpha, \sigma)$ 
          itself is a degeneration of the $0$-simplex $(\id_{\Delta^0},a)$.
          Hence we just proved that every $2$-simplex with degenerate 
          $0 \to 2$ edge is already degenerate itself. Hence so will be every invertible $1$-simplex.
          In particular the pushforward along $\Delta^1 \to \Delta^0$ functor 
          \begin{equation*}
            (A \times_{\simplexcat^\op} \Fun^{\text{active}}(\Delta^1,\simplexcat^\op))_0 \to (A \times_{\simplexcat^\op} \Fun^{\text{active}}(\Delta^1,\simplexcat^\op))^{\text{invert}}_1
          \end{equation*}
          whose codomain is restricted to the full subspace on the invertible morphisms, is essentially
          surjective.
          We have also 
          \begin{equation*}
            (A \times_{\simplexcat^\op} \Fun^{\text{active}}(\Delta^1,\simplexcat^\op))_0 \equiv A \times_{\simplexcat^\op} \{ \Delta^0 \xto{\id} \Delta^0 \}
          \end{equation*}
          As seen above, the pushforward along $\Delta^1 \to \Delta^0$ functor factorizes as 
          \[\begin{tikzcd}[cramped]
            {A \times_{\simplexcat^\op} \{ \Delta^0 \xto{\id} \Delta^0 \}} & {A \times_{\simplexcat^\op} \{ \Delta^1 \to \Delta^0 \}} & {(A \times_{\simplexcat^\op} \Fun^{\text{active}}(\Delta^1, \simplexcat^\op))_1} \\
            & {\{ \Delta^1 \to \Delta^0 \}} & {(\Fun^{\text{active}}(\Delta^1, \simplexcat^\op))_1}
            \arrow["\equiv"', from=1-1, to=1-2]
            \arrow["{(\Delta^1 \to \Delta^0)_!}", curve={height=-24pt}, from=1-1, to=1-3]
            \arrow[hook, from=1-2, to=1-3]
            \arrow["{\pr_1}"', from=1-2, to=2-2]
            \arrow["{(\pr_1)_1}", from=1-3, to=2-3]
            \arrow[""{name=0, anchor=center, inner sep=0}, hook, from=2-2, to=2-3]
            \arrow["\lrcorner"{anchor=center, pos=0.125}, draw=none, from=1-2, to=0]
          \end{tikzcd}\]
          where the square is a pullback by definition. The functor including the object
          \begin{equation*}
            \{ \Delta^1 \to \Delta^0 \} \mono (\Fun^{\text{active}}(\Delta^1, \simplexcat^\op))_1
          \end{equation*}
          is in fact fully-faithful as $\Delta^0$ has no non-trivial endofunctors, hence so 
          is its pullback in the above diagram.
          Putting these observations together we proved that the pushforward along
          $\Delta^1 \to \Delta^0$ functor is actually fully-faithful, hence
          we proved that
          \begin{equation*}
            (A \times_{\simplexcat^\op} \Fun^{\text{active}}(\Delta^1,\simplexcat^\op))_0 \to (A \times_{\simplexcat^\op} \Fun^{\text{active}}(\Delta^1,\simplexcat^\op))^{\text{invert}}_1
          \end{equation*}
          is an equivalence.

        \item The functor $\iota$ is defined as
          \begin{equation*}
            A \equiv A \times_{\simplexcat^\op} \simplexcat^\op \xto{\id_A \times_{\id} \Delta_{(\simplexcat^\op)}} A \times_{\simplexcat^\op} \Fun^{\text{active}}(\Delta^1,\simplexcat^\op) 
          \end{equation*}
          By the the explicit formula for cocartesian lifting in the envelope construction
          and the fact that for any inert simplex morphism $i \colon \Delta^n \mono \Delta^m$ the 
          active-inert factorization of the composite $\Delta^n \xmono{i} \Delta^m \xto{\id} \Delta^m$
          is exactly $\Delta^n \xto{\id} \Delta^n \xmono{i} \Delta^m$, we can deduce that 
          $\iota$ is inert-cocartesian.
          Lastly, the functor $\simplexcat^\op \xto{\Delta_{\simplexcat^\op}} \Fun^{\text{active}}(\Delta^1,\simplexcat^\op) $
          is in the fiber over $\Delta^0$ just the functor
          $\{ \Delta^0 \} \to \{ \Delta^0 \xto{\id} \Delta^0 \}$, which is an equivalence.
          Hence so is $\iota$ as a pullback.

      \end{enumerate}
    \end{proof}

    Next we prove that the envelope construction is in fact a classifier for inert-cocartesian functors
    out of $A \fibration \simplexcat^\op$.
    This already appeared for example as \cite[Proposition A.1.3]{gepnerhaugseng}.
    Another proof strategy appears in \cite[Chapter 11, Theorem A.1.5]{GR} and builds on the fact
    that one can relatively (over $\simplexcat^\op$) left Kan extend inert-cocartesian functors 
    from $A \fibration \simplexcat^\op$ to some cocartesian fibration along $\iota$.
    In this general situation, the relative left Kan extensions in $\Cat_{/ \simplexcat^\op}$ are 
    provided by the fact that $\iota$, as a functor between the total categories, is a 
    fully-faithful left adjoint and by the cocartesianness of the target in the extension problem.

    One way to make this precise is for example the following general statement, that left adjoints have 
    a directed kind of left lifting property with respect to cocartesian fibrations and its 
    accompanying specialization, 
    which states that for fully-faithful left adjoints this becomes 
    an actual left lifting property.
    Proofs of these statements can be found in \autoref{subsectiondirectedllpofladjagainstcocartfibs}.

    \begin{restatable}{corollary}{laxliftingpropertyofladjwrtcocartfib}
      \label{laxliftingpropertyofleftadjointsagainstcocartesianfibrations}
      Let $F \colon A \to B$ a left adjoint with right adjoint $U$ and
      $p \colon X \fibration Y$ a cocartesian fibration.
      Then the functor
      \begin{equation*}
        \Fun(\Delta^1,\Fun(A,X)) \hspace{-0.6em} \pullback{\Fun(A,X)} \hspace{-0.6em} \Fun(B,X) \to \Fun(A,X) \hspace{-0.6em} \pullback{\Fun(A,Y)} \hspace{-0.6em} \Fun(\Delta^1,\Fun(A,Y)) \hspace{-0.6em} \pullback{\Fun(A,Y)} \hspace{-0.6em} \Fun(B,Y)
      \end{equation*}
      defined as in \autoref{weirdadjointablecocartfiblemma}
      from the adjointable square
      \begin{equation}
            \begin{tikzpicture}[diagram]
                \matrix[objects] {%
                |(a)| \Fun(B,X) \& |(b)| \Fun(A,X) \\
                |(c)| \Fun(B,Y) \& |(d)| \Fun(A,Y) \\
                };
                \path[maps,->]
                (a) edge node[above]  {$\Fun(F,X)$} (b)
                (c) edge node[below]  {$\Fun(F,Y)$} (d)
                (a) edge node[left]   {$\Fun(B,p)$} (c)
                (b) edge node[right]  {$\Fun(A,p)$} (d)
                ;
            \end{tikzpicture}
      \end{equation}
      has a fully-faithful left adjoint.
    \end{restatable}

    \begin{restatable}{corollary}{liftingpropertyofffladjwrtcocartfib}
      \label{liftsoffullandfaithfulleftadjointsagainstcocartesianfibrationsincommutativesquares}
      Let $F \colon A \to B$ a left adjoint and
      $p \colon X \fibration Y$ a cocartesian fibration.
      If we furthermore assume the left adjoint $F$ to be fully-faithful,
      then the adjunction from 
      \autoref{laxliftingpropertyofleftadjointsagainstcocartesianfibrations}
      restricts on both sides to an adjunction
      \begin{equation}
            \begin{tikzpicture}[diagram]
                \matrix[objects] {%
                |(a)| \Fun(B,X) \& \& |(b)| \Fun(A,X) \times^{(p_\ast,{}_\ast F)}_{\Fun(A,Y)} \Fun(B,Y) \\
                \& |(c)| \Fun(B,Y)  \\
                };
                \path[maps,->]
                (a) edge node[above] (U) {$({}_\ast F, p_\ast)$} (b)
                (a) edge node[below left]   {$p_\ast$} (c)
                (b) edge node[below right]  {$\pr_1$} (c)
                ;
                \path[maps,{Hooks[left]}->]
                (b) edge[bend right] node[above] (F) {$$} (a)
                ;
                \node[rotate=-90] at ($ (U) ! 0.5 ! (F) $) {$\adj$};
            \end{tikzpicture}
      \end{equation}
      which even lives over $\Fun(B,Y)$ via the indicated functors in the 
      above diagram.

      In this particular situation, by further unpacking the proof of 
      \autoref{laxliftingpropertyofleftadjointsagainstcocartesianfibrations},
      we can characterize the essential image of the fully-faithful left adjoint 
      of this adjunction via \autoref{characterizationofessimoffullandfaithfulradj} 
      as exactly those functors $L \colon B \to X$ whose whiskering $L \epsilon$, 
      where $\epsilon$ is the counit of the adjunction $F \adj U$, 
      is a $p$-cocartesian lift.

    \end{restatable}

    The fact that $\iota$ is a fully-faithful left adjoint will be deduced from the stability of 
    right adjoints with fully-faithful left adjoints under pullback.
    This can be found proven in \autoref{subsectionadjunctionsadjointabilityandcocartfibs}.

    \begin{restatable}{lemma}{pullbackoflalis}
        \label{pullbackofadjunctions}
        Let
        \begin{equation*}
            \begin{tikzpicture}[diagram]
                \matrix[objects]{%
                    |(a)| X \& |(b)| A \\
                    |(c)| Y \& |(d)| B \\
                };
                \path[maps,->]
                    (a) edge node[above] {$H$} (b)
                    (a) edge node[left] {$V$} (c)
                    (b) edge node[right] {$U$} (d)
                    (c) edge node[below] {$K$} (d)
                ;
                \node at (barycentric cs:a=0.8,b=0.3,c=0.3) (phi) {\mbox{\LARGE{$\lrcorner$}}};
            \end{tikzpicture}
        \end{equation*}
        be a pullback square such that $U$ has a fully-faithful left adjoint $F$.
        Then its pullback $V$ also has a fully-faithful left adjoint $G$, , 
        and the pullback square is furthermore horizontally adjointable.
        Additionally, one can compute the fully-faithful left adjoint $G$ as the pullback of $F$ along $H$,
        and the (co)unit of the adjunction $G \adj V$ 
        as the pullback of the (co)unit of $F \adj U$ along $K$, respectively $H$.
    \end{restatable}

    \begin{proposition}
      \label{operadicenvelopeconstrisinertcocartfunclassifier}
      Let $p \colon A \fibration \simplexcat^\op$ be a Segal cocartesian fibration.
      Then the envelope construction 
      \begin{equation*}
        \ev_1 \colon A \times_{\simplexcat^\op} \Fun^{\text{active}}(\Delta^1,\simplexcat^\op) \to \simplexcat^\op
      \end{equation*}
      is an \define{inert-cocartesian functor classifier} in the sense that for any 
      other Segal cocartesian fibration $q \colon B \fibration \simplexcat^\op$ we have that 
      precomposition with the inert-cocartesian functor 
      \begin{equation*}
        \iota \colon A \to A \times_{\simplexcat^\op} \Fun^{\text{active}}(\Delta^1,\simplexcat^\op)
      \end{equation*}
      over $\simplexcat^\op$ gives an equivalence of restricted relative functor
      categories over $\simplexcat^\op$
      \begin{equation*}
        \Fun^{\text{cocart}}_{/ \simplexcat^\op}(A \times_{\simplexcat^\op} \Fun^{\text{active}}(\Delta^1,\simplexcat^\op), B) \xto{{}_\ast \iota \quad \equiv} \Fun^{\text{inert-cocart}}_{/ \simplexcat^\op}(A,B)
      \end{equation*}
      where on the left hand side we chose the full subcategory on cocartesian 
      functors and on the right hand side the full subcategory on inert-cocartesian
      functors in $\Fun_{/ \simplexcat^\op}(-,-)$.
      This tells us that $\iota$ is the initial inert-cocartesian functor 
      over $\simplexcat^\op$ out of $p \colon A \fibration \simplexcat^\op$.
    \end{proposition}
    \begin{proof}
      We will proceed as layed out in \cite[Chapter 11, Theorem A.1.5]{GR}.
      First, we observe that the canonical adjunction
      \begin{equation*}
        \begin{tikzpicture}[diagram]
          \matrix[objects] {
            |(a)| \Fun(\Delta^1, \simplexcat^\op) \& |(b)| \simplexcat^\op \\
          };
          \path[maps,->] 
            (a) edge[bend left] node[above] (f) {$\Delta_{\simplexcat^\op}$} (b)
            (b) edge[bend left] node[below] (g) {$\ev_0$}  (a)
          ;
          \node[rotate=-90] at ($ (f) ! 0.5 ! (g) $) (psi) {$\adj$};
        \end{tikzpicture}
      \end{equation*}
      restricts from $\Fun(\Delta^1, \simplexcat^\op)$ down to the full subcategory 
      $\Fun^{\text{active}}(\Delta^1, \simplexcat^\op)$
      as identity morphism are in particular active.
      By applying \autoref{pullbackofadjunctions} 
      now to the defining pullback square 
      \[\begin{tikzcd}[cramped]
        {A \times_{\simplexcat^\op} \Fun^{\text{active}}(\Delta^1, \simplexcat^\op)} & A \\
        {\Fun^{\text{active}}(\Delta^1, \simplexcat^\op)} & {\simplexcat^\op}
        \arrow["{\pr_0}", from=1-1, to=1-2]
        \arrow["{\pr_1}", from=1-1, to=2-1]
        \arrow["p", from=1-2, to=2-2]
        \arrow[""{name=0, anchor=center, inner sep=0}, "{\ev_0}"', from=2-1, to=2-2]
        \arrow["\lrcorner"{anchor=center, pos=0.125}, draw=none, from=1-1, to=0]
      \end{tikzcd}\]
      we deduce that the fully-faithful functor $\iota$ which is given by
      \begin{equation*}
        A \equiv A \times_{\simplexcat^\op} \simplexcat^\op \xto{\id_A \times_{\id} \Delta_{(\simplexcat^\op)}} A \times_{\simplexcat^\op} \Fun^{\text{active}}(\Delta^1,\simplexcat^\op) 
      \end{equation*}
      actually is left adjoint to $\rho \coloneqq \pr_0 \colon A \times_{\simplexcat^\op} \Fun^{\text{active}}(\Delta^1,\simplexcat^\op) \to  A$ 
      on total categories with invertible unit.
      Note however, that this adjunction does not live over $\simplexcat^\op$ via $\ev_1$ and $p$, 
      \ie its counit will not become invertible after applying $\ev_1$.
      We know from \autoref{laxliftingpropertyofleftadjointsagainstcocartesianfibrations}
      and
      \autoref{liftsoffullandfaithfulleftadjointsagainstcocartesianfibrationsincommutativesquares}
      that for the adjunction $\iota \adj \rho$ and the cocartesian fibration $q$ we have the 
      restricted lifting adjunction

      \begin{adjustbox}{scale=0.9}
        \begin{tikzcd}[cramped]
          {\Fun(A \times_{\simplexcat^\op} \Fun^{\text{active}}(\Delta^1,\simplexcat^\op),B)} & {\Fun(A,B) \times^{(q_\ast,{}_\ast \iota)}_{\Fun(A,\simplexcat^\op)} \Fun(A \times_{\simplexcat^\op} \Fun^{\text{active}}(\Delta^1,\simplexcat^\op), \simplexcat^\op)} \\
          & {\Fun(A \times_{\simplexcat^\op} \Fun^{\text{active}}(\Delta^1,\simplexcat^\op), \simplexcat^\op)}
          \arrow[""{name=0, anchor=center, inner sep=0}, "{({}_\ast \iota, q_\ast)}", from=1-1, to=1-2]
          \arrow["{q_\ast}"', curve={height=18pt}, from=1-1, to=2-2]
          \arrow[""{name=1, anchor=center, inner sep=0}, curve={height=-18pt}, hook, from=1-2, to=1-1]
          \arrow["{\pr_1}", from=1-2, to=2-2]
          \arrow["\dashv"{anchor=center, rotate=153}, draw=none, from=1, to=0]
        \end{tikzcd}
      \end{adjustbox}
      \newline
      which even lives over 
      $\Fun(A \times_{\simplexcat^\op} \Fun^{\text{active}}(\Delta^1,\simplexcat^\op), \simplexcat^\op)$.
      This means we can pull it back to its fiber over the functor $\ev_1$
      and obtain an adjunction
        \[\begin{tikzcd}[cramped]
          {\Fun_{/ \simplexcat^\op}(A \times_{\simplexcat^\op} \Fun^{\text{active}}(\Delta^1,\simplexcat^\op),B)} & {\Fun_{/ \simplexcat^\op}(A,B)}
          \arrow[""{name=0, anchor=center, inner sep=0}, "{{}_\ast \iota}", from=1-1, to=1-2]
          \arrow[""{name=1, anchor=center, inner sep=0}, curve={height=-18pt}, hook, from=1-2, to=1-1]
          \arrow["\dashv"{anchor=center, rotate=26}, draw=none, from=1, to=0]
        \end{tikzcd}\]
      between the relative functor categories over $\simplexcat^\op$.
      Our goal now is to restrict this adjunction to the desired equivalence.
      In general, one always can restrict such an adjunction with invertible unit to an equivalence 
      if one restricts on the left hand side to the essential image of the fully-faithful left adjoint.
      Now as stated in 
      \autoref{liftsoffullandfaithfulleftadjointsagainstcocartesianfibrationsincommutativesquares}
      we can characterize the essential image of the fully-faithful left 
      adjoint as exactly those functors 
      $L \colon A \times_{\simplexcat^\op} \Fun^{\text{active}}(\Delta^1,\simplexcat^\op) \to B$
      over $\simplexcat^\op$,
      whose whiskering $L \epsilon$, where $\epsilon$ denotes the counit of the 
      adjunction $\iota \adj \rho$, is a $q$-cocartesian lift.
      Next we want to restrict the right hand side of our restricted
      relative functor category equivalence to inert-cocartesian functors.
      To prove the desired equivalence it now remains to show that for a functor 
      $L \colon A \times_{\simplexcat^\op} \Fun^{\text{active}}(\Delta^1,\simplexcat^\op) \to B$
      to be cocartesian is actually equivalent to asking $L \epsilon$ to be a $q$-cocartesian lift 
      together with asking that $L \circ \iota$ is inert-cocartesian.
      In order to prove this claim we first need an explicit description of the 
      counit of the adjunction $\iota \adj \rho$. 
      As the adjunction $\iota \adj \rho$ is the pullback of the 
      adjunction $\Delta_{(\simplexcat^\op)} \adj \ev_0$ its component at an 
      object $(\alpha \colon \Delta^n \xto{\text{active}} \Delta^m, \sigma)$ of 
      $A \times_{\simplexcat^\op} \Fun^{\text{active}}(\Delta^1,\simplexcat^\op)$
      is precisely the square
      \begin{equation}
          \begin{tikzpicture}[diagram]
              \matrix[objects] {%
              |(a)| \Delta^n \& |(b)| \Delta^m \\
              |(c)| \Delta^m \& |(d)| \Delta^m \\
              };
              \path[maps,->]
              (a) edge node[above]  {$\alpha$} (b)
              (a) edge node[below]  {$\text{active}$} (b)
              (c) edge node[below]  {$\id$} (d)
              (c) edge node[above]  {$\text{active}$} (d)
              (a) edge node[left]   {$\alpha$} (c)
              (b) edge node[right]  {$\id$} (d)
              ;
          \end{tikzpicture}
      \end{equation}
      seen as a morphism $\id_{\Delta^m} \to \alpha$ in 
      $\Fun^{\text{active}}(\Delta^1,\simplexcat^\op)$
      and the morphism $\id_\sigma$ over $\id_{\Delta^m}$ in $A$.
      Note that this is in fact a cocartesian lift in
      $A \times_{\simplexcat^\op} \Fun^{\text{active}}(\Delta^1,\simplexcat^\op)$.
      Also, images of inert-cocartesian lifts in $A$ under $\iota$ 
      are cocartesian in
      $A \times_{\simplexcat^\op} \Fun^{\text{active}}(\Delta^1,\simplexcat^\op)$
      as $\iota$ is inert-cocartesian.
      Now, to prove our remaining claim, let us take an arbitrary
      cocartesian lift of an arbitrary object 
      $(\alpha \colon \Delta^n \xto{\text{active}} \Delta^m, \sigma)$ in
      $A \times_{\simplexcat^\op} \Fun^{\text{active}}(\Delta^1,\simplexcat^\op)$,
      along an arbitrary simplex morphism $\gamma \colon \Delta^k \to \Delta^n$,
      \ie 
      \begin{equation}
          \begin{tikzpicture}[diagram]
              \matrix[objects] {%
              |(a)| \Delta^k \& |(b)| \Delta^l \\
              |(c)| \Delta^n \& |(d)| \Delta^m \\
              };
              \path[maps,->]
              (a) edge node[above]  {$\beta$} (b)
              (a) edge node[below]  {$\text{active}$} (b)
              (c) edge node[below]  {$\alpha$} (d)
              (c) edge node[above]  {$\text{active}$} (d)
              (a) edge node[left]   {$\gamma$} (c)
              (b) edge node[right]  {$\theta$} (d)
              (b) edge node[left]  {$\text{inert}$} (d)
              ;
          \end{tikzpicture}
      \end{equation}
      seen as a morphism $\alpha \to \beta$ in 
      $\Fun^{\text{active}}(\Delta^1,\simplexcat^\op)$
      and the $p$-cocartesian lift $\sigma \xto{\text{cocart}} \theta_! \sigma$
      of $\sigma$ along $\theta$.
      Let us depict this as in the following way, 
      where objects in $A$ are drawn over their fibers in $\simplexcat^\op$.
      \[\begin{tikzcd}[cramped]
        && {\theta_! \sigma} \\
        {\Delta^k} && {\Delta^l} & \sigma \\
        & {\Delta^n} && {\Delta^m}
        \arrow[no head, from=1-3, to=2-3]
        \arrow["\beta", from=2-1, to=2-3]
        \arrow["{\text{active}}"', from=2-1, to=2-3]
        \arrow["\gamma"', from=2-1, to=3-2]
        \arrow["\theta", hook, from=2-3, to=3-4]
        \arrow["{\text{inert}}"', hook, from=2-3, to=3-4]
        \arrow["{\text{cocart}}"', from=2-4, to=1-3]
        \arrow[no head, from=2-4, to=3-4]
        \arrow["\alpha", from=3-2, to=3-4]
        \arrow["{\text{active}}"', from=3-2, to=3-4]
      \end{tikzcd}\]
      Let us postcompose this with the component to the counit $\epsilon$ 
      at $(\alpha, \sigma)$ to get the following morphism.
      \[\begin{tikzcd}[cramped]
        && {\theta_! \sigma} \\
        {\Delta^k} && {\Delta^l} & \sigma \\
        & {\Delta^n} && {\Delta^m} & \sigma \\
        && {\Delta^m} && {\Delta^m}
        \arrow[no head, from=1-3, to=2-3]
        \arrow["\beta", from=2-1, to=2-3]
        \arrow["{\text{active}}"', from=2-1, to=2-3]
        \arrow["\gamma"', from=2-1, to=3-2]
        \arrow["\theta", hook, from=2-3, to=3-4]
        \arrow["{\text{inert}}"', hook, from=2-3, to=3-4]
        \arrow["{\text{cocart}}"', from=2-4, to=1-3]
        \arrow[no head, from=2-4, to=3-4]
        \arrow["\alpha", from=3-2, to=3-4]
        \arrow["{\text{active}}"', from=3-2, to=3-4]
        \arrow["\alpha"', from=3-2, to=4-3]
        \arrow[Rightarrow, no head, from=3-4, to=4-5]
        \arrow[Rightarrow, no head, from=3-5, to=2-4]
        \arrow[no head, from=3-5, to=4-5]
        \arrow[Rightarrow, no head, from=4-3, to=4-5]
      \end{tikzcd}\]
      This can in fact be factorized also in a different way as
      \[\begin{tikzcd}[cramped]
        && {\theta_! \sigma} \\
        {\Delta^k} && {\Delta^l} & {\theta_! \sigma} \\
        & {\Delta^l} && {\Delta^l} & \sigma \\
        && {\Delta^m} && {\Delta^m}
        \arrow[no head, from=1-3, to=2-3]
        \arrow["\beta", from=2-1, to=2-3]
        \arrow["{\text{active}}"', from=2-1, to=2-3]
        \arrow["\beta"', from=2-1, to=3-2]
        \arrow[Rightarrow, no head, from=2-3, to=3-4]
        \arrow[Rightarrow, no head, from=2-4, to=1-3]
        \arrow[no head, from=2-4, to=3-4]
        \arrow[Rightarrow, no head, from=3-2, to=3-4]
        \arrow["\theta"', hook, from=3-2, to=4-3]
        \arrow["{\text{inert}}", hook, from=3-2, to=4-3]
        \arrow["{\text{inert}}"', hook, from=3-4, to=4-5]
        \arrow["\theta", hook, from=3-4, to=4-5]
        \arrow["{\text{cocart}}"', from=3-5, to=2-4]
        \arrow[no head, from=3-5, to=4-5]
        \arrow[Rightarrow, no head, from=4-3, to=4-5]
      \end{tikzcd}\]
      Note that we have here another component of the counit $\epsilon$, 
      but at $(\beta, \theta_! \sigma)$, as well as the image 
      of an inert-cocartesian lift 
      $\sigma \xto{\text{cocart}} \theta_! \sigma$ under $\iota$.
      Now if our functor $L$ sends $\epsilon$ to a cocartesian
      lift as well as images of inert-cocartesian lifts under $\iota$,
      then it in particular sends the last depicted composite to a 
      cocartesian lift. But as this agrees with the second to last 
      depicted morphism, thus it also gets sent to a cocartesian lift.
      In this second to last composite, the first morphism is also sent to a 
      cocartesian lift by $L$, so by the left cancellation property of cocartesian lifts,
      for example proven as \autoref{cocartesianliftscanbediagrammaticallyleftcancelled}
      in the Appendix,
      we deduce that also our arbitrary chosen cocartesian lift 
      gets sent to something cocartesian.
      Thus the whole functor $L$ is cocartesian.
    \end{proof}

    Lastly, we will exhibit the existence of the strict $2$-functor $\lambda \colon \Env(A) \to A$ for the
    envelope construction. But in fact we will prove more: $\lambda$ will be the left adjoint to 
    the fully-faithful inert-cocartesian functor $\iota$ over $\simplexcat^\op$.
    The main ingredient for the proof will be the following $1$-categorical statement, a proof of 
    which can also be found in the Appendix in \autoref{subsectionmoreadvancedstabilityproperties}.

    \begin{restatable}{lemma}{superpullbackradjalongcocartfib}
      \label{pullbackofrightadjointalongcocartesianfibrationisrightadjoint}
      Let
      \begin{equation*}
            \begin{tikzpicture}[diagram]
                \matrix[objects]{%
                    |(a)| Y \& |(b)| X \\
                    |(c)| A \& |(d)| B \\
                };
                \path[maps,->]
                    (a) edge node[above] {$V$} (b)
                    (a) edge[->>] node[left] {$q$} (c)
                    (b) edge[->>] node[right] {$p$} (d)
                    (c) edge node[below] {$U$} (d)
                ;
                \node at (barycentric cs:a=0.8,b=0.3,c=0.3) (phi) {\mbox{\LARGE{$\lrcorner$}}};
            \end{tikzpicture}
      \end{equation*}
      be a pullback square with $p$ a cocartesian fibration and $U$ has a left adjoint $F$.
      Then its pullback $V$ also has a left adjoint, denoted by $G$, and the square is adjointable.
      Additionally, 
      $V$ is cocartesian over $U$
      and 
      $G$ is also a cocartesian functor over $F$.
      Furthermore, if the counit of $F \adj U$ is invertible, then the counit of $G \adj V$ will be 
      invertible too.

    \end{restatable}

    \begin{lemma}
      \label{leftadjtoiotaforoperadicenvconstr}
      Let $p \colon A \fibration \simplexcat^\op$ be a Segal cocartesian fibration.
      Then the inert-cocartesian functor 
      $$\iota \colon A \to A \times_{\simplexcat^\op} \Fun^{\text{active}}(\Delta^1,\simplexcat^\op)$$
      has a left adjoint $\lambda$ over $\simplexcat^\op$, such that the 
      counit of this adjunction is invertible, and which is also 
      a cocartesian functor.
      Furthermore, it corresponds to the identity functor
      on $A$ under the equivalence 
      \begin{equation*}
        \Fun^{\text{cocart}}_{/ \simplexcat^\op}(A \times_{\simplexcat^\op} \Fun^{\text{active}}(\Delta^1,\simplexcat^\op), A) \xto{{}_\ast \iota \quad \equiv} \Fun^{\text{inert-cocart}}_{/ \simplexcat^\op}(A,A)
      \end{equation*}
      of 
      \autoref{operadicenvelopeconstrisinertcocartfunclassifier}.
    \end{lemma}
    \begin{proof}
      The left adjoint $\lambda$ is constructed via 
      \autoref{pullbackofrightadjointalongcocartesianfibrationisrightadjoint}
      applied to the defining pullback square for $\iota$, \ie
      the left hand pullback square in the pullback factorisation
      \begin{equation*}
            \begin{tikzpicture}[diagram]
                \matrix[objects]{%
                  |(a)| A \& |(b)| A \times_{\simplexcat^\op} \Fun^{\text{active}}(\Delta^1,\simplexcat^\op) \& |(x)| A \\
                  |(c)| \simplexcat^\op \& |(d)| \Fun^{\text{active}}(\Delta^1,\simplexcat^\op) \& |(y)| \simplexcat^\op \\
                };
                \path[maps,->]
                    (a) edge node[above] {$\iota$} (b)
                    (b) edge node[above] {$\pr_0$} (x)
                    (a) edge[->>] node[left] {$$} (c)
                    (x) edge[->>] node[left] {$$} (y)
                    (b) edge[->>] node[right] {$\pr_1$} (d)
                    (c) edge node[below] {$\Delta_{(\simplexcat^\op)}$} (d)
                    (d) edge node[below] {$\ev_0$} (y)
                    (a) edge[-,double distance=0.2em,out=20,in=160] (x)
                    (c) edge[-,double distance=0.2em,out=-20,in=-160] (y)
                ;
                \node at (barycentric cs:a=0.8,b=0.3,c=0.3) (phi) {\mbox{\LARGE{$\lrcorner$}}};
                \node at (barycentric cs:b=0.8,x=0.3,d=0.3) (phi) {\mbox{\LARGE{$\lrcorner$}}};
            \end{tikzpicture}
      \end{equation*}
      where we use that the other canonical adjunction $\ev_1 \adj \Delta_{(\simplexcat^\op)}$
      also restricts along the full subcategory inclusion 
      $\Fun^{\text{active}}(\Delta^1,\simplexcat^\op) \mono \Fun(\Delta^1,\simplexcat^\op)$,
      and that its counit is invertible.
      The adjointability statement of
      \autoref{pullbackofrightadjointalongcocartesianfibrationisrightadjoint}
      automatically gives us that 
      \begin{equation*}
            \begin{tikzpicture}[diagram]
                \matrix[objects]{%
                  |(a)| A \& |(b)| A \times_{\simplexcat^\op} \Fun^{\text{active}}(\Delta^1,\simplexcat^\op) \\
                  |(c)| \simplexcat^\op \& |(d)| \Fun^{\text{active}}(\Delta^1,\simplexcat^\op) \\
                };
                \path[maps,->]
                    (b) edge node[above] {$\lambda$} (a)
                    (a) edge[->>] node[left] {$$} (c)
                    (b) edge[->>] node[right] {$\pr_1$} (d)
                    (d) edge node[below] {$\ev_1$} (c)
                ;
            \end{tikzpicture}
      \end{equation*}
      commutes, \ie $\lambda$ defines a functor over $\simplexcat^\op$, which is even cocartesian 
      over $\ev_1$.
      But we also have that $\ev_1$
      \begin{equation*}
            \begin{tikzpicture}[diagram]
                \matrix[objects]{%
                  |(a)| \simplexcat^\op \& |(b)| \Fun^{\text{active}}(\Delta^1,\simplexcat^\op) \\
                  |(c)| \simplexcat^\op \& |(d)| \simplexcat^\op \\
                };
                \path[maps,->]
                    (b) edge node[above] {$\ev_1$} (a)
                    (a) edge[->>] node[left] {$\id$} (c)
                    (b) edge[->>] node[right] {$\ev_1$} (d)
                    (d) edge node[below] {$\id$} (c)
                ;
            \end{tikzpicture}
      \end{equation*}
      is itself a cocartesian functor from itself to 
      the identity. Putting these two observations together we obtain, that
      $\lambda$ is a cocartesian functor
      \begin{equation*}
            \begin{tikzpicture}[diagram]
                \matrix[objects]{%
                  |(a)| A \& |(b)| A \times_{\simplexcat^\op} \Fun^{\text{active}}(\Delta^1,\simplexcat^\op) \\
                  |(c)| \simplexcat^\op \& |(d)| \Fun^{\text{active}}(\Delta^1,\simplexcat^\op) \\
                  |(x)| \simplexcat^\op \& |(y)| \simplexcat^\op \\
                };
                \path[maps,->]
                    (b) edge node[above] {$\lambda$} (a)
                    (a) edge[->>] node[left] {$$} (c)
                    (b) edge[->>] node[right] {$\pr_1$} (d)
                    (d) edge node[below] {$\ev_1$} (c)
                    (c) edge[->>] node[left] {$\id$} (x)
                    (d) edge[->>] node[right] {$\ev_1$} (y)
                    (y) edge node[below] {$\id$} (x)
                ;
            \end{tikzpicture}
      \end{equation*}
      essentially because the needed vertical adjointability boils down to having a whiskered counit 
      \begin{equation*}
        (\epsilon^{\cocart_p \adj (\ev_0,p_\ast)}) \circ \lambda_\ast \circ \cocart_{\pr_1} \circ \cocart_{\ev_1}
      \end{equation*}
      invertible.
      Here $(\epsilon^{\cocart_p \adj (\ev_0,p_\ast)})$ denotes the counit of the adjunction 
      $\cocart_p \adj (\ev_0,p_\ast)$ which witnesses the cocartesianness of $p$ as in 
      \autoref{definitioncocartesianfibrationandfunctor}.
      This in turn follows from the cocartesianness of $\lambda$ over $\ev_1$, 
      \ie the fact that the whiskered counit 
      $\epsilon^{\cocart_p \adj (\ev_0,p_\ast)} \circ \lambda_\ast \circ \cocart_{\pr_1}$
      is already invertible.
      The fact that $\lambda$ is cocartesian and $\lambda \circ \iota \equiv \id$ also gives us the last statement.
    \end{proof}

  \section{The Envelope satisfies the Universal Property of the Lax 2-Functor Classifier}
    \label{sectiontheenvelopesatisfiestheUPofthelaxfunclassifier}

    In this section we will show that the envelope construction $\Env(A)$ together with its 
    canonical strict $2$-functor $\lambda \colon \Env(A) \to A$ satisfies the model-independent 
    universal property of the lax $2$-functor classifier given in \autoref{definitionlaxfunctorclassifier}.

    As a preliminary step we prove that for a $2$-functor
    \begin{equation*}
      \begin{tikzpicture}[diagram]
          \matrix[objects]{%
            |(a)| X \& \& |(b)| A \\
            \& |(c)| \simplexcat^\op \\
          };
          \path[maps,->]
              (a) edge node[above] {$F$} (b)
              (a) edge[->>] node[left] {$$} (c)
              (b) edge[->>] node[right] {$$} (c)
          ;
      \end{tikzpicture}
    \end{equation*}
    in our model of $2$-categories, having local right adjoint sections 
    is actually equivalent to having a fully-faithful inert-cocartesian right adjoint 
    on total categories, just as observed for 
    $\lambda \colon \Env(A) \to A$ in \autoref{leftadjtoiotaforoperadicenvconstr}.
    Equipped with this statement and its corresponding statement for $2$-functors commuting with 
    local right adjoint sections, we can then reformulate the model-independently defined 
    $1$-category $\LaxSect(A)$, which hosts our universal property, in a way that is 
    closely tailored to the model of globular complete Segal cocartesian fibrations 
    over $\simplexcat^\op$.
    After having achieved this goal we then proceed to step-by-step enhance the 
    inert-cocartesian functor classification property of $\Env(A)$ to finally ressemble initiality
    in our aformentioned model-dependent reformulation of $\LaxSect(A)$.

    The first two main ingredients for the preliminary step are about glueing fiberwise adjoints
    for a cocartesian functor between cocartesian fibrations together to an actual adjoint on the total categories, 
    which can already be found in \cite[Proposition 7.3.2.6]{ha},
    and what we need to know about the fiberwise adjoints in order to deduce further cocartesianness 
    of this global adjoint.
    These are proven in \autoref{subsectionrelativeadjoints}.

    \begin{restatable}{proposition}{reladjcanbegivenobjwinbase}
      \label{relativeadjointsexistiftheyexistfiberwise}
        Let $p \colon A \fibration I$, $q \colon B \fibration I$ be two cocartesian
        fibrations over some category $I$ with $\alpha \colon p \equiv q F$ and
        $F \colon A \to B$ a cocartesian functor over $I$. Let us denote for every object
        $i$ of $I$ by $F_i \colon A_i \to B_i$ the restriction of $F$, $A$ and $B$ to the
        fiber over $i$, \ie the pullbacks along $\Delta^0 \xto{i} I$.
        Then the following two statements are equivalent.
        \begin{enumerate}
            \item For all objects $i$ of $I$ the functor $F_i$ has a right adjoint $U_i$.
            \item the functor $F \colon A \to B$ has a right adjoint $U$, which is also a functor over
                $I$, \ie we have that the canonical mate 
                $q\epsilon \circ \alpha U \colon p U \twoto q F U \twoto q$ of $\alpha$ is invertible, 
                or equivalently the counit $\epsilon \colon FU \twoto \id$ of the adjunction is a 
                natural transformation over $I$, \ie the whiskering $q \epsilon$ is invertible.
        \end{enumerate}
    \end{restatable}

    %
    \begin{restatable}{corollary}{cocartnessofreladj}
      \label{cocartesiannessofradjisequivtoadjointabilityofcocartpushforwardsquareforladj}
      Let $F$ be a cocartesian functor
      \begin{equation}
            \begin{tikzpicture}[diagram]
                \matrix[objects] {%
                  |(a)| A \& \& |(b)| B \\
                  \& |(c)| I \\
                };
                \path[maps,->]
                (a) edge node[above]  {$F$} (b)
                (a) edge[->>] node[below left]   {$p$} (c)
                (b) edge[->>] node[below right]  {$q$} (c)
                ;
            \end{tikzpicture}
      \end{equation}
      between cocartesian fibrations $p$ and $q$.
      Let us futhermore assume that $F$ has a right adjoint $U \colon B \to A$ over the base $I$.
      Then the following statements are equivalent.
      \begin{enumerate}
        \item $U$ is a cocartesian functor.



        \item For every morphism $k \colon i \to j$ in $I$ the square
          \begin{equation}
                \begin{tikzpicture}[diagram]
                    \matrix[objects]{%
                        |(a)| A_i \& |(b)| B_i \\
                        |(c)| A_j \& |(d)| B_j \\
                    };
                    \path[maps,->]
                        (a) edge node[above] {$F_i$} (b)
                        (a) edge node[left] {$k_!$} (c)
                        (b) edge node[right] {$k_!$} (d)
                        (c) edge node[below] {$F_j$} (d)
                    ;
                \end{tikzpicture}
          \end{equation}
          is horizontally adjointable.

      \end{enumerate}
    \end{restatable}

    The second main ingredient for the preliminary step is about pullbacks of left adjoints 
    being left adjoints again, a proof can be found in \autoref{subsectionmoreadvancedstabilityproperties}.

    \begin{restatable}{lemma}{cubepullbackofladj}
      \label{cubepullbackofladjwithadjfacesisladj}
        Consider the commutative cube of categories such that the top and bottom faces are
        pullbacks.
        \begin{equation}
            \begin{tikzpicture}[diagram]
                \matrix[objects,narrow] {%
                    |(a)| A_2 \& \& |(b)| A_1 \\
                    \& |(c)| A_0 \& \& |(d)| A \\
                    |(e)| B_2 \& \& |(f)| B_1 \\
                    \& |(g)| B_0 \& \& |(h)| B \\
                };
                \path[maps,->]
                (b) edge node[above]  {$$} (a)
                (f) edge node[above]  {$$} (e)

                (d) edge node[above]  {$$} (c)
                (h) edge node[above]  {$$} (g)

                (c) edge node[left]   {$$} (a)
                (g) edge node[left]   {$$} (e)

                (d) edge node[left]   {$$} (b)
                (h) edge node[left]   {$$} (f)

                (a) edge node[left]   {$F_2$} (e)
                (c) edge node[below left]   {$F_0$} (g)
                (b) edge node[below left]   {$F_1$} (f)
                (d) edge node[left]   {$F$} (h)
                ;
            \end{tikzpicture}
        \end{equation}
        If the $F_i$ are left adjoints and the back face as well as
        the left face are adjointable, then the pullback induced functor $F$ is also a left
        adjoint and the right and front faces are also adjointable.
        Additionally, if all the left adjoints $F_i$, or all the right
        adjoints $U_i$
        respectively, are fully-faithful, then so is the pulled-back one.
    \end{restatable}

    The strategy to relate having local right adjoint sections and having a fully-faithful inert-cocartesian 
    right adjoint over $\simplexcat^\op$ now will be roughly the following.
    \begin{itemize}
      \item First we will pass from the Hom-functor properties of having local right adjoint sections to having 
        right adjoint sections for the functors between the fibers over $\Delta^1$ in $\simplexcat^\op$, as the 
        Hom-categories $X(x,y)$ are exactly the fibers of the source-target functor $(s,t) \colon X_1 \to X_0 \times X_0$ 
        in our model.
      \item Then we will glue these right adjoint sections using \autoref{cubepullbackofladjwithadjfacesisladj} along the inverse equivalences of $F^\equiv$
        to get right adjoint sections on all fibers over arbitrary $\Delta^n$ via the Segal axiom. 
      \item At last, we use \autoref{relativeadjointsexistiftheyexistfiberwise} to glue the fiberwise adjoints together to a global adjoint on total categories.
    \end{itemize}

    \begin{proposition}
      \label{laxadjointsectionmodelindependentlyandinCSOCocartfiboverDeltaOP}

        Let $F$ be a cocartesian functor
        \begin{equation}
              \begin{tikzpicture}[diagram]
                  \matrix[objects] {%
                    |(a)| X \& \& |(b)| A \\
                    \& |(c)| \simplexcat^\op \\
                  };
                  \path[maps,->]
                  (a) edge node[above]  {$F$} (b)
                  (a) edge node[below left]   {$p$} (c)
                  (b) edge node[below right]  {$q$} (c)
                  ;
              \end{tikzpicture}
        \end{equation}
        between Segal cocartesian fibrations $p$ and $q$. 
        Then the following are equivalent.
        \begin{enumerate}
          \item $F$ has a (fully-faithful) right adjoint $U$ over $\simplexcat^\op$, \ie the counit 
            is a natural transformation over $\simplexcat^\op$, and $U$ is
            inert-cocartesian.
          \item For all $n \in \N$ the functor $F_n$ in the fiber over $\Delta^n$ 
            has a (fully-faithful) right adjoint $U_n$ and for every
            inert simplex map $\alpha \colon \Delta^n \mono \Delta^m$ the induced square
            \begin{equation}
              \label{pushforwardadjointablility}
                  \begin{tikzpicture}[diagram]
                      \matrix[objects]{%
                          |(a)| X_n \& |(b)| X_m \\
                          |(c)| A_n \& |(d)| A_m \\
                      };
                      \path[maps,->]
                          (a) edge node[above] {$\alpha_!$} (b)
                          (a) edge node[left] {$F_n$} (c)
                          (b) edge node[right] {$F_m$} (d)
                          (c) edge node[below] {$\alpha_!$} (d)
                      ;
                  \end{tikzpicture}
            \end{equation}
            whitnessing cocartesian pushforward preservation of $F$,
            is vertically adjointable.
          \item $F_0$ and $F_1$ have (fully-faithful) right adjoints $U_0$, respectively $U_1$,
            and for both $s,t \colon \Delta^0 \to \Delta^1$ the vertical adjointablility of 
            \autoref{pushforwardadjointablility} holds.

        \end{enumerate}
        If we furthermore assume that the fibers of $p$ and $q$ over $\Delta^0$ are spaces, then we have 
        even more equivalent statements, extending the list of equivalent statements above.
        \begin{enumerate}
          \setcounter{enumi}{3}

          \item $F_0$ is an equivalence and $F_1$ has a (fully-faithful) right adjoint $U_1$
          \item $F_0$ is an equivalence and for all objects $x, y$ of $X$ the action on hom-categories functors
            \begin{equation}
              X(x,y) \xto{F_{x,y}} A(Fx,Fy)
            \end{equation}
            of $F$ have (fully-faithful) right adjoints.

        \end{enumerate}
    \end{proposition}
    \begin{proof}
      We will show that each assertion is equivalent to the next one.
      \begin{enumerate}
        \item $\iff$ 2. The first part about the right adjoints follows from 
          \autoref{relativeadjointsexistiftheyexistfiberwise}.
          By definition $U$ being inert-cocartesian means that its pullback along 
          $\simplexcat^\op_{\text{inert}} \mono \simplexcat^\op$, let us denote this
          by $\tilde{U} \colon \tilde{A} \to \tilde{X}$, is a cocartesian functor.
          Let us also denote the pullback of $F$ along $\simplexcat^\op_{\text{inert}} \mono \simplexcat^\op$
          by $\tilde{F} \colon \tilde{X} \to \tilde{A}$. 
          By stability of relative adjunctions under base change, \ie pullback, 
          proved for example in \autoref{pullbackpreservesadjunctionsoverbase},
          this is still a relative adjunction over $\simplexcat^\op_{\text{inert}}$.
          Now applying
          \autoref{cocartesiannessofradjisequivtoadjointabilityofcocartpushforwardsquareforladj}
          gives us the second part of the equivalence.

        \item $\iff$ 3. For the non-trivial implication, we iteratively apply 
          \autoref{cubepullbackofladjwithadjfacesisladj} to commutatives cubes like
          \[\begin{tikzcd}[cramped]
            {X_1 \times_{X_0} X_1} && {X_1} \\
            & {X_1} && {X_0} \\
            {A_1 \times_{A_0} A_1} && {A_1} \\
            & {A_1} && {A_0}
            \arrow[from=1-1, to=1-3]
            \arrow[from=1-1, to=2-2]
            \arrow["{F_1 \times_{F_0} F_1}"', from=1-1, to=3-1]
            \arrow["t", from=1-3, to=2-4]
            \arrow["{F_1}"{pos=0.7}, from=1-3, to=3-3]
            \arrow["s"{pos=0.3}, from=2-2, to=2-4]
            \arrow["{F_1}"{pos=0.7}, from=2-2, to=4-2]
            \arrow["{F_0}"', from=2-4, to=4-4]
            \arrow[from=3-1, to=3-3]
            \arrow[from=3-1, to=4-2]
            \arrow["t", from=3-3, to=4-4]
            \arrow["s"', from=4-2, to=4-4]
          \end{tikzcd}\]
          with top and bottom square pullbacks, 
          and then use the Segal property.
          To see that we actually get all vertical adjointability of 
          \autoref{pushforwardadjointablility} for all inert
          simplex morphisms from just the ones of $s$ and $t$, one just needs to 
          observe that every inert simplex morphism can be iteratively decomposed 
          into the simpler inert simplex morphisms
          \begin{equation*}
            \Delta^n \equiv \Delta^n \pushout{\Delta^0} \Delta^0 \xto{\id \pushout{\id} s} \Delta^n \pushout{\Delta^0} \Delta^1 \equiv \Delta^{n+1}
          \end{equation*}
          and 
          \begin{equation*}
            \Delta^n \equiv \Delta^0 \pushout{\Delta^0} \Delta^n \xto{t \pushout{\id} \id} \Delta^1 \pushout{\Delta^0} \Delta^n \equiv \Delta^{n+1}
          \end{equation*}
            The vertical adjointability of \autoref{pushforwardadjointablility}
            for these two types of simplex morphisms follows directly from the ones
            of $s$ and $t$ and the Segal property.

      \end{enumerate}

      If we furthermore assume that the fibers of $p$ and $q$ over $\Delta^0$ are spaces, then we can prove
      also the following equivalences.

      \begin{enumerate}
      \setcounter{enumi}{2}
        \item $\iff$ 4. First note, that any adjunction between spaces automatically also is an 
          adjoint equivalence. Vice versa every equivalence can be upgraded to an adjoint equivalence.
          To see that we get the vertical adjointability for $s$ and $t$ for 
          free with our assumptions one just needs to observe that the vertical mate
          \[\begin{tikzcd}[cramped]
            {A_1} & {X_1} & {X_0} \\
            & {A_1} & {A_0} & {X_0}
            \arrow["{U_1}", from=1-1, to=1-2]
            \arrow[""{name=0, anchor=center, inner sep=0}, Rightarrow, no head, from=1-1, to=2-2]
            \arrow["s", from=1-2, to=1-3]
            \arrow["{F_1}", from=1-2, to=2-2]
            \arrow["{F_0}"', from=1-3, to=2-3]
            \arrow[""{name=1, anchor=center, inner sep=0}, Rightarrow, no head, from=1-3, to=2-4]
            \arrow["s"', from=2-2, to=2-3]
            \arrow["{U_0}"', from=2-3, to=2-4]
            \arrow["{\epsilon_1}"', shorten >=2pt, Rightarrow, from=1-2, to=0]
            \arrow["{\eta_0}", shorten <=2pt, Rightarrow, from=1, to=2-3]
          \end{tikzcd}\]
          will always be invertible as $X_0$ was by assumption a space.

        \item $\iff$ 5. This equivalence follows from \autoref{relativeadjointsexistiftheyexistfiberwise}
          applied to 
          \begin{equation}
                  \begin{tikzpicture}[diagram]
                      \matrix[objects]{%
                          |(a)| X_1 \& |(b)| A_1 \\
                          |(c)| X_0 \times X_0 \& |(d)| A_0 \times A_0 \\
                      };
                      \path[maps,->]
                          (a) edge node[above] {$F_1$} (b)
                          (a) edge[->>] node[left] {$(\ev_0,\ev_1)$} (c)
                          (b) edge[->>] node[right] {$(\ev_0,\ev_1)$} (d)
                          (c) edge node[above] {$\equiv$} (d)
                          (c) edge node[below] {$F_0 \times F_0$} (d)
                      ;
                  \end{tikzpicture}
            \end{equation}
            and the facts that $(\ev_0,\ev_1)$ are trivially cocartesian fibrations, as their
            codomains are spaces, and similarly that $F_1$ is trivially a cocartesian functor.

      \end{enumerate}
    \end{proof}

    Now that we have seen how to model-dependently reformulate the property of having
    local right adjoint sections, we also need to know how to reformulate what it means 
    for $2$-functors over $A$ to preserve these, \ie commute with these.
    \begin{proposition}
      \label{moroflaxadjointsectionmodelindependentlyandinCSOCocartfiboverDeltaOP}
      Let
      \begin{equation}
        \begin{tikzpicture}[diagram]
            \matrix[objects] {%
              |(a)| X \& \& |(b)| Y \\
              \& |(c)| A \\ 
              \& |(d)| \simplexcat^\op \\
            };
            \path[maps,->]
            (b) edge[bend left] (d)
            (c) edge (d)
            (a) edge[bend right] (d)

            (a) edge node[below left]  {$F$} (c)
            (a) edge node[above]   {$H$} (b)
            (b) edge node[below right]  {$G$} (c)
            ;
        \end{tikzpicture}
      \end{equation}
      be a commutative triangle of cocartesian functors between cocartesian fibrations 
      over $\simplexcat^\op$, such that both $F$ and $G$ satisfy both the assumptions and the 
      equivalent conditions of the first part of \autoref{laxadjointsectionmodelindependentlyandinCSOCocartfiboverDeltaOP}.
      Then the following are equivalent.
      \begin{enumerate}
        \item The commutative triangle of total categories
          \begin{equation}
            \begin{tikzpicture}[diagram]
                \matrix[objects] {%
                  |(a)| X \& \& |(b)| Y \\
                  \& |(c)| A  \\
                };
                \path[maps,->]
                (a) edge node[below left]  {$F$} (c)
                (a) edge node[above]   {$H$} (b)
                (b) edge node[below right]  {$G$} (c)
                ;
            \end{tikzpicture}
          \end{equation}
          is vertically adjointable, \ie as a square after adding in the identity functor on $A$, 
          with respect to the right adjoint $U$ of $F$ over
          $\simplexcat^\op$ and the right adjoint $V$ of $G$ over $\simplexcat^\op$.
        \item For all $n \in \N$ the commutative triangle of functors between the fibers 
          over $\Delta^n$
          \begin{equation}
            \begin{tikzpicture}[diagram]
                \matrix[objects] {%
                  |(a)| X_n \& \& |(b)| Y_n \\
                  \& |(c)| A_n  \\
                };
                \path[maps,->]
                (a) edge node[below left]  {$F_n$} (c)
                (a) edge node[above]   {$H_n$} (b)
                (b) edge node[below right]  {$G_n$} (c)
                ;
            \end{tikzpicture}
          \end{equation}
          is vertically adjointable with respect to the fiberwise right adjoints
          $U_n$ of $F_n$ and $V_n$ of $G_n$.
        \item For $i \in \{0,1\}$ the commutative triangle of functors between the fibers 
          over $\Delta^i$
          \begin{equation}
            \begin{tikzpicture}[diagram]
                \matrix[objects] {%
                  |(a)| X_i \& \& |(b)| Y_i \\
                  \& |(c)| A_i  \\
                };
                \path[maps,->]
                (a) edge node[below left]  {$F_i$} (c)
                (a) edge node[above]   {$H_i$} (b)
                (b) edge node[below right]  {$G_i$} (c)
                ;
            \end{tikzpicture}
          \end{equation}
          is vertically adjointable with respect to the fiberwise right adjoints
          $U_i$ of $F_i$ and $V_i$ of $G_i$.
      \end{enumerate}

      If we furthermore assume, as in the second part of 
      \autoref{laxadjointsectionmodelindependentlyandinCSOCocartfiboverDeltaOP},
      that $X_0$, $Y_0$ and $A_0$ are spaces, then we can extend the above list of 
      equivalent statements by the following.

      \begin{enumerate}
        \setcounter{enumi}{3}
        \item 
          For all objects $x, \tilde{x}$ of $X$ the commutative triangle of Hom-functors
          \begin{equation}
            \begin{tikzpicture}[diagram]
                \matrix[objects] {%
                  |(a)| X(x,\tilde{x}) \& \& |(b)| Y(H(x),H(\tilde{x})) \\
                  \& |(c)| A(F(x),F(\tilde{x}))  \\
                };
                \path[maps,->]
                (a) edge node[below left]  {$F$} (c)
                (a) edge node[above]   {$H$} (b)
                (b) edge node[below right]  {$G$} (c)
                ;
            \end{tikzpicture}
          \end{equation}
          is vertically adjointable with respect to the right adjoints on Hom-categories of 
          $F$ and $G$.
      \end{enumerate}
    \end{proposition}
    \begin{proof}
      We will show that each assertion is equivalent to the next one.
      \begin{enumerate}
        \item $\iff$ 2. A natural transformation over a base category is, as 
          every natural transformation, invertible if and only if it is so 
          objectwise. But every object lives in some fiber over the base, hence 
          a natural transformation over a base category is invertible if and only 
          if it is fiberwise invertible.

          But by the proof of \autoref{relativeadjointsexistiftheyexistfiberwise} 
          we know that all the data of the global adjunction agrees with the data of 
          the fiber adjunctions, when restricted to the fibers. Hence the formation of 
          the mate for the global commutative triangle and the global adjunction 
          on total categories agrees with the fiberwise formed mates of the fiber 
          commutative triangles and the fiber adjunctions, when restricted to the 
          fibers.
          Thus our observation about the invertibility of natural transformations over 
          some base tells us that the global mate is invertible if and only if all the 
          fiber mates are invertible.

        \item $\iff$ 3. We have the commutative prism
          \[\begin{tikzcd}
            {X_n} && {Y_n} \\
            & {A_n} \\
            {X_1 \times_{X_0} \dots \times_{X_0} X_1} && {Y_1 \times_{Y_0} \dots \times_{Y_0} Y_1} \\
            & {A_1 \times_{A_0} \dots \times_{A_0} A_1}
            \arrow["{H_n}", from=1-1, to=1-3]
            \arrow["{F_n}"', from=1-1, to=2-2]
            \arrow["\equiv"', from=1-1, to=3-1]
            \arrow["{G_n}", from=1-3, to=2-2]
            \arrow["\equiv", from=1-3, to=3-3]
            \arrow["\equiv"{pos=0.2}, from=2-2, to=4-2]
            \arrow["{H_1 \times_{H_0} \dots \times_{H_0} H_1}"'{pos=0.25}, from=3-1, to=3-3]
            \arrow["{F_1 \times_{F_0} \dots \times_{F_0} F_1}"', from=3-1, to=4-2]
            \arrow["{G_1 \times_{G_0} \dots \times_{G_0} G_1}", from=3-3, to=4-2]
          \end{tikzcd}\]
          Furthermore we know from the proof of 
          \autoref{laxadjointsectionmodelindependentlyandinCSOCocartfiboverDeltaOP},
          \ie the application of \autoref{pullbackofadjunctions},
          that the two front vertical faces of this prism are in fact horizontally adjointable 
          and that the adjunction data for the two bottom diagonal functors between the 
          iterated fiber products are in fact factorwise.
          This means in particular that forming the mate upstairs is equivalent to downstairs 
          forming the mate, which in itself is done fiber product factorwise.
          Now we observe that such a factorwise natural transformation between 
          factorwise given functors 
          \[\begin{tikzcd}[cramped]
            {A_1 \times_{A_0} \dots \times_{A_0} A_1} & {\Fun(\Delta^1, Y_1 \times_{Y_0} \dots \times_{Y_0} Y_1)} \\
            & {\Fun(\Delta^1, Y_1) \times_{\Fun(\Delta^1, Y_0)} \dots \times_{\Fun(\Delta^1, Y_0)} \Fun(\Delta^1, Y_1)}
            \arrow["{\text{mate}_1 \times_{\text{mate}_0} \dots \times_{\text{mate}_0} \text{mate}_1}"', from=1-1, to=2-2]
            \arrow["\equiv"', from=2-2, to=1-2]
          \end{tikzcd}\]
          is invertible if and only if all its factors are invertible, which follows from the fact 
          that a morphism in an iterated fiber product is invertible if and only if 
          its projections to all its factors is invertible.

      \end{enumerate}

      If we furthermore assume, as in the second part of 
      \autoref{laxadjointsectionmodelindependentlyandinCSOCocartfiboverDeltaOP},
      that $X_0$, $Y_0$ and $A_0$ are spaces, then we can also prove the following.

      \begin{enumerate}
        \setcounter{enumi}{2}
        \item $\iff$ 4. First we observe that because $F_0$ and $G_0$ are equivalences $H_0$ 
          also has to be an equivalence. Now our situation can be pictured as 
          \[\begin{tikzcd}
            {X_1} && {Y_1} \\
            & {A_1} \\
            {X_0 \times X_0} && {Y_0 \times Y_0} \\
            & {A_0 \times A_0}
            \arrow["{H_1}", from=1-1, to=1-3]
            \arrow["{F_1}", from=1-1, to=2-2]
            \arrow["{(s,t)}"', from=1-1, to=3-1]
            \arrow["{G_0}"', from=1-3, to=2-2]
            \arrow["{(s,t)}"', from=1-3, to=3-3]
            \arrow["{H_0 \times H_0}"{pos=0.2}, from=3-1, to=3-3]
            \arrow["\equiv"'{pos=0.4}, from=3-1, to=3-3]
            \arrow["{F_0 \times F_0}"', from=3-1, to=4-2]
            \arrow["\equiv", from=3-1, to=4-2]
            \arrow["{G_0 \times G_0}", from=3-3, to=4-2]
            \arrow["\equiv"', from=3-3, to=4-2]
            \arrow["{(s,t)}"{pos=0.1}, from=2-2, to=4-2,crossing over]
          \end{tikzcd}\]
          In particular we can now apply the same reasoning as in 
          1. $\iff$ 2. to get the equivalence of being globally adjointable
          in fiber over $\Delta^1$ and of being Hom-wise adjointable.

      \end{enumerate}
    \end{proof}

    Equipped with these two characterisations we are now able to prove the 
    universal property \autoref{definitionlaxfunctorclassifier} 
    of the lax functor classifier for the envelope construction $\lambda \colon \Env(A) \to A$.
    The final step in the proof will again rely on the stability of 
    left adjoints with invertible counit under pullbacks, \ie \autoref{pullbackofadjunctions}.


    \begin{theorem}
      \label{theoremenvelopeislaxfunctorclassifier}
      In the model of $2$-categories as globular complete Segal cocartesian fibrations
      over $\simplexcat^\op$, the envelope construction
      \begin{equation*}
        \ev_1 \colon \Env(A) \coloneqq A \times_{\simplexcat^\op} \Fun^{\text{active}}(\Delta^1,\simplexcat^\op) \to \simplexcat^\op
      \end{equation*}
      for a $2$-category $A \fibration \simplexcat^\op$ together with its 
      canonical $2$-functor $\lambda \colon \Env(A) \to A$ from \autoref{leftadjtoiotaforoperadicenvconstr}
      satisfies the universal property \autoref{definitionlaxfunctorclassifier} of the 
      lax functor classifier of $A \fibration \simplexcat^\op$.
    \end{theorem}
    \begin{proof}
      The universal property of the lax functor classifier
      is stated as initiality in the non-full subcategory $\LaxSect(A)$ of the 
      strict slice-over $1$-category $\twoCat_{/ A}$, defined in 
      \autoref{definitionlaxfunctorclassifier}.
      As described in \autoref{recollectiontwocatsasglobularcompletesegalfibrations}
      $\twoCat_{/ A}$ can be modelled by $\globCSCocart_{/ A}$.
      By \autoref{laxadjointsectionmodelindependentlyandinCSOCocartfiboverDeltaOP}
      and \autoref{moroflaxadjointsectionmodelindependentlyandinCSOCocartfiboverDeltaOP}
      the non-full subcategory $\LaxSect(A)$ can be modelled by the 
      non-full subcategory of $\globCSCocart_{/ A}$ on those cocartesian functors 
      $F \colon X \to A$ over $\simplexcat^\op$ such that the functor $F$ on 
      total categories has a fully-faithful right adjoint $U$ over
      $\simplexcat^\op$ that is also inert-cocartesian, and morphisms
      \begin{equation}
        \begin{tikzpicture}[diagram]
            \matrix[objects] {%
              |(a)| X \& \& |(b)| Y \\
              \& |(c)| A \\ 
              \& |(d)| \simplexcat^\op \\
            };
            \path[maps,->]
            (b) edge[bend left] (d)
            (c) edge (d)
            (a) edge[bend right] (d)

            (a) edge node[below left]  {$F$} (c)
            (a) edge node[above]   {$H$} (b)
            (b) edge node[below right]  {$G$} (c)
            ;
        \end{tikzpicture}
      \end{equation}
      between such cocartesian functors over $A \fibration \simplexcat^\op$
      such that the commutative triangle on total categories is vertically
      adjointable.
      Thus we have rephrased the category $\LaxSect(A)$ where we ask for initiality
      completely into the language of our model $\globCSCocart$.
      The next step is to exhibit our candidate for the lax functor classifier
      of the $2$-category $A \fibration \simplexcat^\op$ in this model.

      The lax functor classifier will be given by the envelope construction 
      \begin{equation*}
        \ev_1 \colon A \times_{\simplexcat^\op} \Fun^{\text{active}}(\Delta^1,\simplexcat^\op) \to \simplexcat^\op .
      \end{equation*}
      and its $2$-functor to $A$ with respect to which we 
      will prove the universal property, \ie the initiality in 
      $\LaxSect(A)$, will be exactly the cocartesian functor $\lambda$ over 
      $\simplexcat^\op$ from 
      \autoref{leftadjtoiotaforoperadicenvconstr}
      which corresponds to the identity functor
      on $A$ under the equivalence 
      \begin{equation*}
        \Fun^{\text{cocart}}_{/ \simplexcat^\op}(A \times_{\simplexcat^\op} \Fun^{\text{active}}(\Delta^1,\simplexcat^\op), A) \xto{{}_\ast \iota \quad \equiv} \Fun^{\text{inert-cocart}}_{/ \simplexcat^\op}(A,A)
      \end{equation*}
      of 
      \autoref{operadicenvelopeconstrisinertcocartfunclassifier},
      but is also the left adjoint to the inert-cocartesian
      functor $\iota$.

      Now let us start proving the initiality statement.
      For this we take an arbitrary object of our model-dependent 
      version of $\LaxSect(A)$, \ie a cocartesian functor 
      $F \colon X \to A$ between globular complete Segal cocartesian fibrations 
      over $\simplexcat^\op$ such that the functor $F$ on 
      total categories has a fully-faithful right adjoint $U$ over
      $\simplexcat^\op$ that is also inert-cocartesian.
      The first step is to slice the equivalence
      \begin{equation*}
        \Fun^{\text{cocart}}_{/ \simplexcat^\op}(A \times_{\simplexcat^\op} \Fun^{\text{active}}(\Delta^1,\simplexcat^\op), X) \xto{{}_\ast \iota \quad \equiv} \Fun^{\text{inert-cocart}}_{/ \simplexcat^\op}(A,X)
      \end{equation*}
      from
      \autoref{operadicenvelopeconstrisinertcocartfunclassifier}.
      over $A \fibration \simplexcat^\op$ itself, \ie to transform it 
      into a statement in 
      \begin{equation*}
        (\Cat_{/ \simplexcat^\op})_{/ (A \fibration \simplexcat^\op)} \equiv \Cat_{/ A}
      \end{equation*}
      by the cube
      \[\begin{tikzcd}[cramped,column sep={8em,between origins}]
        & {\Fun^{\text{cocart}}_{/ \simplexcat^\op}(\Env(A), X)} && {\Fun^{\text{inert-cocart}}_{/ \simplexcat^\op}(A,X)} \\
        {\Fun^{\text{cocart}}_{/ A}((\Env(A),\lambda), (X,F))} && {\Fun^{\text{inert-cocart}}_{/ A}((A,\id_A),(X,F))} \\
        & {\Fun^{\text{cocart}}_{/ \simplexcat^\op}(\Env(A), A)} && {\Fun^{\text{inert-cocart}}_{/ \simplexcat^\op}(A,A)} \\
        {\Delta^0} && {\Delta^0}
        \arrow["{{}_\ast \iota}", from=1-2, to=1-4]
        \arrow["\equiv"', from=1-2, to=1-4]
        \arrow["{F_\ast}"{pos=0.3}, from=1-2, to=3-2]
        \arrow["{F_\ast}", from=1-4, to=3-4]
        \arrow[from=2-1, to=1-2]
        \arrow["\equiv"'{pos=0.3}, from=2-1, to=2-3]
        \arrow[from=2-1, to=4-1]
        \arrow[from=2-3, to=1-4]
        \arrow[from=2-3, to=4-3]
        \arrow["{{}_\ast \iota}"{pos=0.3}, from=3-2, to=3-4]
        \arrow["\equiv"'{pos=0.3}, from=3-2, to=3-4]
        \arrow["\lambda"', from=4-1, to=3-2]
        \arrow[Rightarrow, no head, from=4-1, to=4-3]
        \arrow["{\id_A}"', from=4-3, to=3-4]
        \arrow["{{}_\ast \iota}"{pos=0.3}, from=2-1, to=2-3]
      \end{tikzcd}\]
      where both the left and the right face are 
      pullbacks by definition.
      Note here that the back face of the cube commutes as $\iota$ 
      is in particular a functor $(A, \id_A) \to (\Env(A),\ev_1)$ over $A$
      as the counit of $\lambda \adj \iota$ is invertible.
      We obtain an equivalence
      \begin{equation}
        \label{inertcocartclassifierequivslicedoverA}
        {}_\ast \iota \colon \Fun^{\text{cocart}}_{/ A}((\Env(A),\lambda), (X,F)) \xto{\equiv} \Fun^{\text{inert-cocart}}_{/ A}((A,\id_A),(X,F)) 
      \end{equation}
      which is also given by precomposition of $\iota$ as a functor
      $(A, \id_A) \to (\Env(A),\ev_1)$ over $A$.
      Now we restrict the left hand side of this equivalence to exactly 
      those functors
      \begin{equation}
        \begin{tikzpicture}[diagram]
            \matrix[objects] {%
              |(a)| A \times_{\simplexcat^\op} \Fun^{\text{active}}(\Delta^1,\simplexcat^\op) \& \& |(b)| X \\
              \& |(c)| A \\ 
              \& |(d)| \simplexcat^\op \\
            };
            \path[maps,->]
            (b) edge[bend left] (d)
            (c) edge (d)
            (a) edge[bend right] (d)

            (a) edge node[below left]  {$\lambda$} (c)
            (a) edge node[above]   {$H$} (b)
            (b) edge node[below right]  {$F$} (c)
            ;
        \end{tikzpicture}
      \end{equation}
      such that the commutative triangle of total categories is 
      vertically adjointable. Let us denote this restriction by 
      $\Fun^{\text{cocart}, \text{vadj}}_{/ A}((\Env(A),\lambda), (X,F))$.
      Now we observe that 
      \begin{equation}
        \begin{tikzpicture}[diagram]
            \matrix[objects] {%
              |(a)| A \times_{\simplexcat^\op} \Fun^{\text{active}}(\Delta^1,\simplexcat^\op) \& \& |(b)| X \\
              \& |(c)| A \\ 
            };
            \path[maps,->]
            (a) edge node[below left]  {$\lambda$} (c)
            (a) edge node[above]   {$H$} (b)
            (b) edge node[below right]  {$F$} (c)
            ;
        \end{tikzpicture}
      \end{equation}
      is vertically adjointable if and only if its precomposition 
      with $\iota$ and the invertible counit $\lambda \circ \iota \equiv \id$
      \begin{equation}
        \begin{tikzpicture}[diagram]
            \matrix[objects] {%
              |(x)| A \& |(a)| A \times_{\simplexcat^\op} \Fun^{\text{active}}(\Delta^1,\simplexcat^\op) \& |(b)| X \\
              \& |(c)| A \\ 
            };
            \path[maps,->]
            (x) edge node[above]   {$\iota$} (a)
            (a) edge node[left]  {$\lambda$} (c)
            (a) edge node[above]   {$H$} (b)
            (b) edge node[below right]  {$F$} (c)
            ;
            \path[maps,-]
            (x) edge[double distance=0.2em] (c)
            ;
        \end{tikzpicture}
      \end{equation}
      is as a composed triangle vertically adjointable.
      Hence we may also restrict the right hand side of the equivalence
      \autoref{inertcocartclassifierequivslicedoverA}
      to the vertically adjointable triangles
      \begin{equation}
        \begin{tikzpicture}[diagram]
            \matrix[objects] {%
              |(x)| A \& |(a)|  \& |(b)| X \\
              \& |(c)| A \\ 
            };
            \path[maps,->]
            (x) edge node[above]   {$K$} (b)
            (b) edge node[below right]  {$F$} (c)
            ;
            \path[maps,-]
            (x) edge[double distance=0.2em] (c)
            ;
        \end{tikzpicture}
      \end{equation}
      We will denote this restriction by 
      $\Fun^{\text{inert-cocart}, \text{vadj}}_{/ A}((A,\id_A),(X,F))$
      and get the restricted equivalence
      \begin{equation}
        \Fun^{\text{cocart}, \text{vadj}}_{/ A}((\Env(A),\lambda), (X,F)) \xto{{}_\ast \iota \quad \equiv} \Fun^{\text{inert-cocart}, \text{vadj}}_{/ A}((A,\id_A),(X,F)) 
      \end{equation}
      But as $F$ has the fully-faithful right adjoint $U$ over 
      $\simplexcat^\op$ we also have a postcomposition adjunction
      \begin{equation*}
        \begin{tikzpicture}[diagram]
          \matrix[objects] {
            |(a)| \Fun_{/ \simplexcat^\op}(A,X) \& |(b)| \Fun_{/ \simplexcat^\op}(A,A) \\
          };
          \path[maps,->] 
            (a) edge[bend left] node[above] (f) {$F_\ast$} (b)
            (b) edge[bend left] node[below] (g) {$U_\ast$}  (a)
          ;
          \node[rotate=-90] at ($ (f) ! 0.5 ! (g) $) (psi) {$\adj$};
        \end{tikzpicture}
      \end{equation*}
      with invertible counit.
      And because $F$ and $U$ are both inert-cocartesian we can further 
      restrict this to full subcategories on inert-cocartesian functors
      \begin{equation*}
        \begin{tikzpicture}[diagram]
          \matrix[objects] {
            |(a)| \Fun^{\text{inert-cocart}}_{/ \simplexcat^\op}(A,X) \& |(b)| \Fun^{\text{inert-cocart}}_{/ \simplexcat^\op}(A,A) \\
          };
          \path[maps,->] 
            (a) edge[bend left] node[above] (f) {$F_\ast$} (b)
            (b) edge[bend left] node[below] (g) {$U_\ast$}  (a)
          ;
          \node[rotate=-90] at ($ (f) ! 0.5 ! (g) $) (psi) {$\adj$};
        \end{tikzpicture}
      \end{equation*}
      and get again an adjunction with invertible counit.
      In particular, in the pullback
      \begin{equation}
        \label{finalterminalitypullback}
            \begin{tikzpicture}[diagram]
                \matrix[objects]{%
                    |(a)| \Fun^{\text{inert-cocart}}_{/ A}((A,\id_A),(X,F)) \& |(b)| \Fun^{\text{inert-cocart}}_{/ \simplexcat^\op}(A,X) \\
                    |(c)| \Delta^0 \& |(d)| \Fun^{\text{inert-cocart}}_{/ \simplexcat^\op}(A,A) \\
                };
                \path[maps,->]
                    (a) edge node[above] {$$} (b)
                    (a) edge node[left] {$$} (c)
                    (b) edge node[right] {$F_\ast$} (d)
                    (c) edge node[below] {$\id_A$} (d)
                ;
                \node at (barycentric cs:a=0.8,b=0.3,c=0.3) (phi) {\mbox{\LARGE{$\lrcorner$}}};
            \end{tikzpicture}
      \end{equation}
      from before we now know that the right hand vertical functor has a 
      fully-faithful right adjoint. So we can deduce by 
      \autoref{pullbackofadjunctions} that also the 
      left hand vertical functor to the point has a fully-faithful right adjoint. 
      This means explicitely that the category $\Fun^{\text{inert-cocart}}_{/ A}((A,\id_A),(X,F))$
      has a terminal object, which is exactly
      \begin{equation}
        \begin{tikzpicture}[diagram]
            \matrix[objects] {%
              |(x)| A \& |(a)|  \& |(b)| X \\
              \& |(c)| A \\ 
            };
            \path[maps,->]
            (x) edge node[above]  (f) {$U$} (b)
            (b) edge node[below right] (g) {$F$} (c)
            ;
            \path[maps,-]
            (x) edge[double distance=0.2em] (c)
            ;
            \node[rotate=45] at ($ (f) ! 0.5 ! (c) $) (psi) {$\equiv$};
            \node[maps,below right] at (psi) {$\epsilon$};
        \end{tikzpicture}
      \end{equation}
      for $\epsilon$ the invertible counit of $F \adj U$.
      On the other hand the essential image of the fully-faithful 
      right adjoint of the left hand functor of 
      \autoref{finalterminalitypullback}
      can be described as all objects for which the unit of the adjunction 
      becomes invertible, but unpacking \autoref{pullbackofadjunctions} 
      in this situation tells us that the component of this pulled 
      back unit for an object 
      \begin{equation}
        \label{objectinFunoverAinertcocartfromidAtoF}
        \begin{tikzpicture}[diagram]
            \matrix[objects] {%
              |(x)| A \& |(a)|  \& |(b)| X \\
              \& |(c)| A \\ 
            };
            \path[maps,->]
            (x) edge node[above]   {$K$} (b)
            (b) edge node[below right]  {$F$} (c)
            ;
            \path[maps,-]
            (x) edge[double distance=0.2em] (c)
            ;
        \end{tikzpicture}
      \end{equation}
      of $\Fun^{\text{inert-cocart}}_{/ A}((A,\id_A),(X,F))$
      is exactly its composition with the unit $\eta$ of the adjunction 
      $F \adj U$, \ie
      \begin{equation}
        \begin{tikzpicture}[diagram]
            \matrix[objects] {%
              |(x)| A \& |(a)|  \& |(b)| X \\
              \& |(c)| A  \& \& |(d)| X \\ 
            };
            \path[maps,->]
            (x) edge node[above]   {$K$} (b)
            (b) edge node[above left] (g)  {$F$} (c)
            (c) edge node[below] (f) {$U$} (d)
            ;
            \path[maps,-]
            (x) edge[double distance=0.2em] (c)
            (b) edge[double distance=0.2em] (d)
            ;
            \node[rotate=-135] at ($ (f) ! 0.5 ! (b) $) (psi) {$\twoto$};
            \node[maps,below right] at (psi) {$\eta$};
        \end{tikzpicture}
      \end{equation}
      Note that this composition is in fact exactly the mate transformation 
      of the triangle \autoref{objectinFunoverAinertcocartfromidAtoF}.
      Hence the essential image of the fully-faithful right adjoint 
      to the left hand vertical functor in our pullback square \autoref{finalterminalitypullback} 
      consists exactly of the vertically adjointable triangles.
      Restricting this pulled back adjunction now to its essential image we obtain 
      the second equivalence in
      \begin{equation}
        \Fun^{\text{cocart}, \text{vadj}}_{/ A}((\Env(A),\lambda), (X,F)) \xto{{}_\ast \iota \quad \equiv} \Fun^{\text{inert-cocart}, \text{vadj}}_{/ A}((A,\id_A),(X,F)) \xto{\equiv} \Delta^0 
      \end{equation}
      In particular we deduce from this the equivalence on underlying 
      mapping spaces
      \begin{equation}
        \Map^{\text{cocart}, \text{vadj}}_{/ A}((\Env(A),\lambda), (X,F)) \xto{{}_\ast \iota \quad \equiv} \Map^{\text{inert-cocart}, \text{vadj}}_{/ A}((A,\id_A),(X,F)) \xto{\equiv} \Delta^0 
      \end{equation}
      which now proves initiality of $(\Env(A), \lambda)$ in the model-dependent 
      reformulation of $\LaxSect(A)$.
    \end{proof}

  \appendix
  \section{Background on Synthetic 1-Category Theory}
    \label{appendixbackgroundononecategorytheory}

    In this appendix we recall some fundamental $1$-categorical definitions and prove the 
    $1$-categorical statements that the previous sections proofs relied on, namely
    \autoref{laxliftingpropertyofleftadjointsagainstcocartesianfibrations},
    \autoref{liftsoffullandfaithfulleftadjointsagainstcocartesianfibrationsincommutativesquares},
    \autoref{pullbackofrightadjointalongcocartesianfibrationisrightadjoint},
    \autoref{relativeadjointsexistiftheyexistfiberwise},
    \autoref{cocartesiannessofradjisequivtoadjointabilityofcocartpushforwardsquareforladj},
    \autoref{cubepullbackofladjwithadjfacesisladj},
    and \autoref{pullbackofadjunctions}.

    In this exposition we try to keep the proofs as elementary and model-independent as possible, 
    based on the definitions and facts we chose to base ourselves on, which are recorded in the first subsection.
    This is partly, because we could not find elementary and model-independent proofs for the statements
    \autoref{liftsoffullandfaithfulleftadjointsagainstcocartesianfibrationsincommutativesquares},
    \autoref{pullbackofrightadjointalongcocartesianfibrationisrightadjoint},
    \autoref{cubepullbackofladjwithadjfacesisladj},
    and \autoref{pullbackofadjunctions}.
    We do not claim any originality for these statements or their proofs.
    The statement \autoref{relativeadjointsexistiftheyexistfiberwise}
    can already be deduced from the dual of \cite[Proposition 7.3.2.6]{ha}.
    For its extension
    \autoref{cocartesiannessofradjisequivtoadjointabilityofcocartpushforwardsquareforladj}
    and the directed lifting property of left adjoints against cocartesian fibrations
    \autoref{laxliftingpropertyofleftadjointsagainstcocartesianfibrations},
    we could not at all find a reference in the literature.

    \subsection{Definitions and basic properties}

    We start by reviewing some definitions and facts that we need for the $1$-categorical statements 
    in the previous sections, like adjunctions, (co)cartesian fibrations, (co)cartesian functors 
    between those and the adjointability of commutative squares. 

    \begin{definition}
      An \define{adjunction $F \adj U$} between functor $F \colon A \to B$, called the \define{left adjoint},
      and $U \colon B \to A$, called the \define{right adjoint}, consists of two natural transformations 
      $\eta \colon \id \twoto UF$ and $\epsilon \colon FU \twoto \id$ and two identifications 
      $\id_U \equiv (U \epsilon) \circ (\eta U)$ and $\id_F \equiv (\epsilon F) \circ (F \eta)$, called 
      \define{triangle identities}.
      When the unit $\eta$ is invertible we say that the \define{left adjoint $F$ is fully-faithful}.
      Dually, when the counit $\epsilon$ is invertible we say that the 
      \define{right adjoint $U$ is fully-faithful}.
    \end{definition}

    \begin{definition}
      \label{definitioncocartesianfibrationandfunctor}
      A functor $p \colon X \to B$ is a \define{cocartesian fibration}
      if the functor
      \begin{equation*}
        (\ev_0, p_\ast) \colon \Fun(\Delta^1,X) \to X \times_B \Fun(\Delta^1,B)
      \end{equation*}
      has a fully-faithful left adjoint.
      Dually, it is a \define{cartesian fibration} if the functor
      \begin{equation*}
        (p_\ast, \ev_1) \colon \Fun(\Delta^1,X) \to \Fun(\Delta^1,B) \times_B X
      \end{equation*}
      has a fully-faithful right adjoint.
    \end{definition}

    \begin{definition}
      Let
      \begin{equation*}
            \begin{tikzpicture}[diagram]
                \matrix[objects]{%
                    |(a)| X \& |(b)| A \\
                    |(c)| Y \& |(d)| B \\
                };
                \path[maps,->]
                    (a) edge node[above] {$H$} (b)
                    (a) edge node[left] {$V$} (c)
                    (b) edge node[right] {$U$} (d)
                    (c) edge node[below] {$K$} (d)
                ;
            \end{tikzpicture}
      \end{equation*}
      be a commutative square for which both the functors $U$ and $V$ are 
      right adjoints. From the data of these two adjunctions, more precisely the 
      counit $\epsilon'$ of the adjunction $G \adj V$ and the unit $\eta$ of the 
      adjunction $F \adj U$, we can form the composed natural transformation
      \[\begin{tikzcd}[cramped]
        Y & X & A \\
        & Y & B & A
        \arrow["G", from=1-1, to=1-2]
        \arrow[""{name=0, anchor=center, inner sep=0}, equals, from=1-1, to=2-2]
        \arrow["H", from=1-2, to=1-3]
        \arrow["V", from=1-2, to=2-2]
        \arrow["U"', from=1-3, to=2-3]
        \arrow[""{name=1, anchor=center, inner sep=0}, equals, from=1-3, to=2-4]
        \arrow["K"', from=2-2, to=2-3]
        \arrow["F"', from=2-3, to=2-4]
        \arrow["{\epsilon'}"', shorten >=2pt, Rightarrow, from=1-2, to=0]
        \arrow["\eta", shorten <=2pt, Rightarrow, from=1, to=2-3]
      \end{tikzcd}\]
      which is called the \define{vertical mate} of this commutative square.
      Such a commutative square in which both vertical functors are right adjoints
      is called \define{vertically adjointable} if its vertical mate transformation is 
      invertible.
    \end{definition}

    \begin{definition}
      Let
      \begin{equation*}
            \begin{tikzpicture}[diagram]
                \matrix[objects]{%
                    |(a)| X \& |(b)| Y \\
                    |(c)| A \& |(d)| B \\
                };
                \path[maps,->]
                    (a) edge node[above] {$F$} (b)
                    (a) edge node[left] {$p$} (c)
                    (b) edge node[right] {$q$} (d)
                    (c) edge node[below] {$G$} (d)
                ;
            \end{tikzpicture}
      \end{equation*}
      be a commutative square such that $p$ and $q$ are cocartesian fibrations.
      We call the functor $F$ a \define{cocartesian functor over $G$} if the induced
      commutative square
      \begin{equation*}
            \begin{tikzpicture}[diagram]
                \matrix[objects]{%
                    |(a)| \Fun(\Delta^1,X) \& |(b)| \Fun(\Delta^1,Y) \\
                    |(c)| X \times_A \Fun(\Delta^1,A) \& |(d)| Y \times_B \Fun(\Delta^1,B) \\
                };
                \path[maps,->]
                    (a) edge node[above] {$F_\ast$} (b)
                    (a) edge node[left] {$(\ev_0, p_\ast)$} (c)
                    (b) edge node[right] {$(\ev_0, q_\ast)$} (d)
                    (c) edge node[below] {$F \times_G G_\ast$} (d)
                ;
            \end{tikzpicture}
      \end{equation*}
      is vertically adjointable.
    \end{definition}

    We record the following facts about the arrow categories $\Fun(\Delta^1,A)$ 
    and their domain/codomain-fibrations $\ev_0, \ev_1 \colon \Fun(\Delta^1, A) \to A$ 
    for later reference.

    \begin{proposition}
      \label{evaluationfunctorsarefibrations}
      For every category $A$ we have 
      \begin{enumerate}
        \item $\ev_1 \colon \Fun(\Delta^1,A) \to A$ is a cocartesian fibration.

        \item $\ev_0 \colon \Fun(\Delta^1,A) \to A$ is a cartesian fibration.

        \item The fibers of $(\ev_0,\ev_1) \colon \Fun(\Delta^1,A) \to A \times A$ are spaces, 
          or equivalently $(\ev_0,\ev_1)$ is a conservative functor.

        \item The functor 
          $\Delta_A \coloneqq (\Delta^1 \to \Delta^0)_\ast \colon A \equiv \Fun(\Delta^0,A) \to \Fun(\Delta^1,A)$ 
          is both a section to $\ev_0$ and $\ev_1$ and the invertible natural transformations 
          $\ev_1 \circ \Delta_A \equiv \id$ and $\id \equiv \Delta_A \circ \ev_1$ 
          witnessing this are counit, respectively unit for adjunctions $\ev_1 \adj \Delta_A \adj \ev_0$.

        \item Furthermore, the unit $\id \twoto \Delta_A \circ \ev_1$ is the $\ev_0$-cartesian lift of $\Delta_A \circ \ev_1$ along the canonical natural transformation $\ev_0 \twoto \ev_1$.
          More precisely the composite
          \begin{equation}
            \Fun(\Delta^1,A) \xto{(\id, \Delta_A \ev_1)} \Fun(\Delta^1,A) \times_A \Fun(\Delta^1,A) \xmono{\cart_{\ev_0}} \Fun(\Delta^1,\Fun(\Delta^1,A))
          \end{equation}
          is equal to the unit.

        \item Dually, the counit $\Delta_A \circ \ev_0 \twoto \id$ is the $\ev_1$-cocartesian lift of 
          $\Delta_A \circ \ev_0$ along the canonical natural transformation $\ev_0 \twoto \ev_1$, \ie the composite
          \begin{equation}
            \Fun(\Delta^1,A) \xto{(\Delta_A \ev_0, \id)} \Fun(\Delta^1,A) \times_A \Fun(\Delta^1,A) \xmono{\cocart_{\ev_1}} \Fun(\Delta^1,\Fun(\Delta^1,A))
          \end{equation}
          is equal to the counit.

        \item The cocartesian pushforward $\cocart_{\ev_1} \colon \Fun(\Delta^1,A) \times_A \Fun(\Delta^1,A) \mono \Fun(\Delta^1,\Fun(\Delta^1,A))$
          of $\ev_1$ is $\ev_0$-vertical, \ie factors over $\Delta_A \colon A \mono \Fun(\Delta^1,A)$
          when postcomposed with $(\ev_0)_\ast$.

        \item The cartesian pullback functor $\cart_{\ev_0} \colon \Fun(\Delta^1,A) \times_A \Fun(\Delta^1,A) \mono \Fun(\Delta^1,\Fun(\Delta^1,A))$
          of $\ev_0$ is $\ev_1$-vertical, \ie factors over $\Delta_A \colon A \mono \Fun(\Delta^1,A)$
          when postcomposed with $(\ev_1)_\ast$.

      \end{enumerate}

    \end{proposition}
    \begin{proof}
      All the above statements can be justified by inserting certain adjunctions between 
      $\Delta^0$, $\Delta^1$, $\Delta^2$ and $\Delta^1 \times \Delta^1$ 
      into $\Fun(-,A)$ and using the canonical pushout identifications 
      $\Delta^2 \equiv \Delta^1 \cup_{\Delta^0} \Delta^1$ and 
      $\Delta^1 \times \Delta^1 \equiv \Delta^2 \cup_{\Delta^1} \Delta^2$, 
      \ie that the square can be written as two commutative triangles glued along their composite edges.
      The first statement follows from looking at the adjunction 
      \begin{equation}
            \begin{tikzpicture}[diagram]
                \matrix[objects] {%
                |(a)| \Delta^1 \& |(b)| \Delta^2 \\
                };
                \path[maps,->]
                (b) edge[bend left] node[below] (U) {$\sigma_0$} (a)
                ;
                \path[maps,{Hooks[left]}->]
                (a) edge[bend left] node[above] (F) {$(0,2)$} (b)
                ;
                \node[rotate=-90] at ($ (U) ! 0.5 ! (F) $) {$\adj$};
            \end{tikzpicture}
      \end{equation}
      where the right adjoint $\sigma_0$ contracts the edge $0 \to 1$, by applying pushout along $(0,2) \colon \Delta^1 \mono \Delta^2$ to get
      \begin{equation}
            \begin{tikzpicture}[diagram]
                \matrix[objects] {%
                  |(a)| \Delta^2 \equiv \Delta^2 \cup_{\Delta^1} \Delta^1 \& |(b)| \qquad \qquad \Delta^2 \cup_{\Delta^1} \Delta^2 \equiv \Delta^1 \times \Delta^1 \\
                };
                \path[maps,->]
                (b) edge[bend left] node[below] (U) {$\sigma_0$} (a)
                ;
                \path[maps,{Hooks[left]}->]
                (a) edge[bend left] node[above] (F) {$(0,2)$} (b)
                ;
                \node[rotate=-90] at ($ (U) ! 0.5 ! (F) $) {$\adj$};
            \end{tikzpicture}
      \end{equation}
      The fourth statement can be proven from the inital and terminal object adjunctions 
      \begin{equation}
            \begin{tikzpicture}[diagram]
                \matrix[objects] {%
                |(a)| \Delta^0 \& |(b)| \Delta^1 \\
                };
                \path[maps,->]
                (b) edge node (U) {$!$} (a)
                ;
                \path[maps,{Hooks[left]}->]
                (a) edge[bend left] node[above] (F) {$0$} (b)
                (a) edge[bend right] node[below] (V) {$1$} (b)
                ;
                \node[rotate=-90] at ($ (U) ! 0.5 ! (F) $) {$\adj$};
                \node[rotate=-90] at ($ (V) ! 0.5 ! (U) $) {$\adj$};
            \end{tikzpicture}
      \end{equation}
      for $\Delta^1$. The counit here for the adjunction $0 \adj \hspace{0.3em} !$ 
      is given by the functor $\Delta^1 \times \Delta^1 \to \Delta^1$ contracting the edges $(0,0) \to (0,1)$ and $(0,0) \to (1,0)$.
      In this sense one can then also deduce the fifth statement by observing that this counit agrees with the composite
      \begin{equation*}
        \Delta^1 \leftarrow \Delta^2 \equiv \Delta^2 \pushout{\Delta^1} \Delta^1 \leftarrow \Delta^2 \pushout{\Delta^1} \Delta^2 \equiv \Delta^1 \times \Delta^1 ,
      \end{equation*}
      where the leftmost functor contracts the edge $0 \to 1$.
      All other statements follow similary or by further inspecting these adjunctions.
    \end{proof}

    \begin{proposition}
      \label{functorsinducecocartesianfunctorsbetweencodomainfibrations}
      For every functor $F \colon A \to B$ we have that 
      \begin{enumerate}
        \item in the commutative square
          \[\begin{tikzcd}
            {\Fun(\Delta^1,A)} & {\Fun(\Delta^1,B)} \\
            A & B
            \arrow["{F_\ast}", from=1-1, to=1-2]
            \arrow["{\ev_1}"', from=1-1, to=2-1]
            \arrow["{\ev_1}", from=1-2, to=2-2]
            \arrow["F"', from=2-1, to=2-2]
          \end{tikzcd}\]
          the functor $F_\ast$ is cocartesian over $F$ and
        \item in the commutative square
          \[\begin{tikzcd}
            {\Fun(\Delta^1,A)} & {\Fun(\Delta^1,B)} \\
            A & B
            \arrow["{F_\ast}", from=1-1, to=1-2]
            \arrow["{\ev_0}"', from=1-1, to=2-1]
            \arrow["{\ev_0}", from=1-2, to=2-2]
            \arrow["F"', from=2-1, to=2-2]
          \end{tikzcd}\]
          the functor $F_\ast$ is cartesian over $F$.

      \end{enumerate}
    \end{proposition}
    \begin{proof}
      This follows from the fact that inserting the simplicial adjunction from the proof of 
      \autoref{evaluationfunctorsarefibrations} in the first variable of $\Fun(-,-)$ 
      and observing that this commutes with plugging in $F$ into the second variable of $\Fun(-,-)$.
    \end{proof}

    Furthermore we the following fact about commutative squares in categories, they can be decomposed into two commutative triangles that agree on their composite edges.

    \begin{proposition}
      \label{commutativesquarescanbedecomposedintotwocommtriangles}
      For any category $A$ the square
      \[\begin{tikzcd}
        {\Fun(\Delta^1,\Fun(\Delta^1,A))} && {\Fun(\Delta^1,A) \times_A \Fun(\Delta^1,A)} \\
        \\
        {\Fun(\Delta^1,A) \times_A \Fun(\Delta^1,A)} && {\Fun(\Delta^1,A)}
        \arrow["{((\ev_0)_\ast, \ev_1)}", from=1-1, to=1-3]
        \arrow["{(\ev_0, (\ev_1)_\ast)}"', from=1-1, to=3-1]
        \arrow["\comp", from=1-3, to=3-3]
        \arrow["\comp"', from=3-1, to=3-3]
      \end{tikzcd}\] 
      commutes and is a pullback.
    \end{proposition}
    \begin{proof}
      We use the following canonical pushout square.
      \[\begin{tikzcd}[cramped]
        {\Delta^1} & {\Delta^2} \\
        {\Delta^2} & {\Delta^1 \times \Delta^1}
        \arrow["02"', from=1-1, to=1-2]
        \arrow["02", from=1-1, to=2-1]
        \arrow[from=1-2, to=2-2]
        \arrow[from=2-1, to=2-2]
        \arrow["\lrcorner"{anchor=center, pos=0.125, rotate=180}, draw=none, from=2-2, to=1-1]
      \end{tikzcd}\]
    \end{proof}

  \subsection{Adjunctions, Adjointablility and Cocartesian Fibrations}
    \label{subsectionadjunctionsadjointabilityandcocartfibs}

    In this subsection we collect some general statements about adjunctions, cocartesian fibrations and some 
    elementary stability properties of them.

    The first statement is about the functor $\Fun(A,-)$ preserving adjoints.
    Furthermore this adjointability is in a sense preserved when varying $A$, \ie 
    we get adjointable squares.

    \begin{lemma}
      \label{FunApreservesadjunctionsandFunFgivesadjointablesquares}
      Let $A$ be a category and $F \colon X \to Y$ a (fully-faithful) left adjoint.
      Then $\Fun(A,X) \to \Fun(A,Y)$ is again a (fully-faithful) left adjoint. 
      Furthermore, for any functor $H \colon A \to B$, the commutative square
      \begin{equation}
            \begin{tikzpicture}[diagram]
                \matrix[objects] {%
                |(a)| \Fun(B,X) \& |(b)| \Fun(A,X) \\
                |(c)| \Fun(B,Y) \& |(d)| \Fun(A,Y) \\
                };
                \path[maps,->]
                (a) edge node[above]  {$\Fun(H,X)$} (b)
                (c) edge node[below]  {$\Fun(H,Y)$} (d)
                (a) edge node[left]   {$\Fun(B,F)$} (c)
                (b) edge node[right]  {$\Fun(A,F)$} (d)
                ;
            \end{tikzpicture}
      \end{equation}
      is vertically adjointable.
      The same is true for right adjoints.
    \end{lemma}
    \begin{proof}
      We have the post-whiskering functor
      \begin{equation*}
        \Fun(A,C) \xto{\alpha_\ast} \Fun(A,\Fun(\Delta^1,D)) \xto{\equiv} \Fun(\Delta^1,\Fun(A,D))
      \end{equation*}
      for any natural transformation $C \xto{\alpha} \Fun(\Delta^1,D)$, hence we can say
      that $\Fun(A,-)$ preserves natural transformations.
      In a similar way it can be seen that $\Fun(A,-)$ preserves compositions and identities of
      natural transformations.
      Hence it preserves the data of adjunctions.
      For adjointablility we now just need to observe that we have
      \begin{equation*}
        \Fun(F,D) \circ \Fun(B,\alpha) = \Fun(A,\alpha) \circ \Fun(F,C) 
      \end{equation*}
      which allows us to rewrite the mate of our square into $\Fun(F,U \epsilon \circ \eta U)$,
      which is equivalent to the the identity natural transformation because of the triangle identities.
    \end{proof}

    The following statement is about $\Fun(\Delta^1,-)$ preserving cocartesian fibrations and its 
    evaluations functors $\ev_0$ and $\ev_1$ constituting cocartesian functors.

    \begin{corollary}
      \label{evaluationfunctorsarecocartesian}
      Let $p \colon A \fibration I$ be a cocartesian fibration.
      Then $p_\ast \colon \Fun(\Delta^1,A) \to \Fun(\Delta^1,I)$
      is also a cocartesian fibration and both evaluation functors
      $\ev_0$ and $\ev_1$
      \begin{equation}
            \begin{tikzpicture}[diagram]
                \matrix[objects] {%
                |(a)| \Fun(\Delta^1,A) \& |(b)| A \\
                |(c)| \Fun(\Delta^1,I) \& |(d)| I \\
                };
                \path[maps,->]
                (a) edge node[above]  {$\ev_i$} (b)
                (c) edge node[below]  {$\ev_i$} (d)
                (a) edge node[left]   {$p_\ast$} (c)
                (b) edge node[right]  {$p$} (d)
                ;
            \end{tikzpicture}
      \end{equation}
      are cocartesian functors.
    \end{corollary}
    \begin{proof}
      As $\Fun(\Delta^1,-)$ preserves pullbacks we have the following commutative triangle.
      \[\begin{tikzcd}
        {\Fun(\Delta^1,\Fun(\Delta^1,A))} \\
        {\Fun(\Delta^1,A) \times_{\Fun(\Delta^1,I)} \Fun(\Delta^1,\Fun(\Delta^1,I))} & {\Fun(\Delta^1,A \times_I \Fun(\Delta^1,I))}
        \arrow["{(\ev_0, (p_\ast)_\ast)}"', from=1-1, to=2-1]
        \arrow["{(\ev_0,p_\ast)_\ast}", from=1-1, to=2-2]
        \arrow["\equiv", from=2-2, to=2-1]
      \end{tikzcd}\]
      Hence the left hand vertical functor has a fully-faithful left adjoint as the
      right hand diagonal functor has one, because $\Fun(\Delta^1,-)$ preserves
      having fully-faithful adjoints by \autoref{FunApreservesadjunctionsandFunFgivesadjointablesquares}.
      Using $i \colon \Delta^0 \mono \Delta^1$ as $H$ in the same lemma also gives us the needed
      adjointablility of the square witnessing the desired cocartesianness of the 
      evaluation functors $\ev_i$.
    \end{proof}

    We now construct for every adjunction defined using unit and counit its action on 
    the appropriate Hom-spaces.

    \begin{construction}
      \label{constructionofactiononHomspacesforadjunction}
      For an adjunction with left adjoint $F \colon B \to A$, right adjoint $U \colon A \to B$, 
      unit $\eta \colon \id UF$ and counit $\epsilon \colon FU \to \id$ we can construct the 
      following two functors.
      \begin{equation*}
        \phi \coloneqq B \times_{A} \Fun(\Delta^1,A) \xto{\eta \times_{U} (U_\ast, \ev_1)} \Fun(\Delta^1,B) \times_{B} \Fun(\Delta^1,B) \times_{B} A \xto{\comp \times_{\id} \id} \Fun(\Delta^1,B) \times_{B} A
      \end{equation*}
      \begin{equation*}
        \psi \coloneqq \Fun(\Delta^1,B) \times_{B} A  \xto{(\ev_0, F_\ast) \times_{F} \epsilon} B \times_{A} \Fun(\Delta^1,A) \times_{A} \Fun(\Delta^1,A) \xto{\id \times_{\id} \comp} B \times_{A} \Fun(\Delta^1,A)
      \end{equation*}
      These functors are called \define{the action of the adjunction on Hom-spaces}.
      By construction these two functors commute with the
      canonical projections 
      $(\ev_0 \circ \pr_0, \pr_1) \colon \Fun(\Delta^1,B) \times_B A \to B \times A$ and 
      $(\pr_0, \ev_1 \circ \pr_1) \colon B \times_A \Fun(\Delta^1,A) \to B \times A$.
    \end{construction}

    These actions on Hom-spaces are in fact inverse to each other, via the triangle identities of the 
    adjunction. We will later prove that having such a parametrized Hom-equivalence is actually equivalent to 
    constituting an adjunction.

    \begin{lemma}
      \label{firsthalfoftriangleidentityadjunctionsareequivalenttohomcatdef}
      For an adjunction $(F, U, \eta, \epsilon)$ the two functors $\phi$ and $\psi$ constructed in \autoref{constructionofactiononHomspacesforadjunction}
      are inverse to each other.
    \end{lemma}
    \begin{proof}
      For $\psi \circ \phi \equiv \id$ we start by looking at the composite commutative diagram 
      \[\begin{tikzcd}[cramped]
        {B \times_A \Fun(\Delta^1,A) \times_A \Fun(\Delta^1,A)} & {B \times_A \Fun(\Delta^1,A)} \\
        & {\Fun(\Delta^1,B) \times_B \Fun(\Delta^1,B) \times_B A} \\
        {B \times_A \Fun(\Delta^1,A) \times_A \Fun(\Delta^1,A) \times_A \Fun(\Delta^1,A)} & {\Fun(\Delta^1,B) \times_B A} \\
        & {B \times_A \Fun(\Delta^1,A) \times_A \Fun(\Delta^1,A)} \\
        {B \times_A \Fun(\Delta^1,A) \times_A \Fun(\Delta^1,A)} & {B \times_A \Fun(\Delta^1,A)}
        \arrow["{\id \times_{\id} \id \times ((FU)_\ast, \varepsilon \ev_1)}"', from=1-1, to=3-1]
        \arrow["{(\id, F \eta, \id)}"', from=1-2, to=1-1]
        \arrow["{\eta \times_U (U_\ast, \ev_1)}", from=1-2, to=2-2]
        \arrow["{(\ev_0, F_\ast) \times_F (F_\ast, \varepsilon)}"', from=2-2, to=3-1]
        \arrow["{\comp \times_{\id} \id}", from=2-2, to=3-2]
        \arrow["{\id \times_{\id} \comp \times_{\id} \id}"'{pos=0.3}, from=3-1, to=4-2]
        \arrow["{\id \times_{\id} \times_{\id} \comp}"{description}, from=3-1, to=5-1]
        \arrow["{(\ev_0, F_\ast) \times_F \varepsilon}", from=3-2, to=4-2]
        \arrow["{\id \times_{\id} \comp}", from=4-2, to=5-2]
        \arrow["{\id \times_{\id} \comp}"', from=5-1, to=5-2]
      \end{tikzcd}\]
      which on the right hand vertical compostition just unpacks the definitions of $\psi \circ \phi$ 
      and realize that the other composite going around the rectangle can be can reformulated like in the
      diagram
      \begin{adjustbox}{scale=0.9}
        \begin{tikzcd}[cramped]
          & {B \times_A \Fun(\Delta^1,A)} \\
          & {B \times_A \Fun(\Delta^1,A) \times_A \Fun(\Delta^1,A)} \\
          {B \times_A \Fun(\Delta^1,A) \times_A \Fun(\Delta^1,A) \times_A \Fun(\Delta^1,A)} \\
          & {B \times_A \Fun(\Delta^1,A) \times_A \Fun(\Delta^1,A) \times_A \Fun(\Delta^1,A)} \\
          {B \times_A \Fun(\Delta^1,A) \times_A \Fun(\Delta^1,A)} & {B \times_A \Fun(\Delta^1,A) \times_A \Fun(\Delta^1,A)} \\
          & {B \times_A \Fun(\Delta^1,A)}
          \arrow["{(\id, F \eta, \id)}"', from=1-2, to=2-2]
          \arrow["{\id \times_{\id} \id \times_{\id} (\varepsilon \ev_0, \id)}"', from=2-2, to=3-1]
          \arrow["{\id \times_{\id} \id \times ((FU)_\ast, \varepsilon \ev_1)}"', from=2-2, to=4-2]
          \arrow["{\id \times_{\id} \comp \times_{\id} \id}"', from=3-1, to=5-1]
          \arrow["{\id \times_{\id} \id \times_{\id} \comp}"', from=3-1, to=5-2]
          \arrow["{\id \times_{\id} \times_{\id} \comp}"{description}, from=4-2, to=5-2]
          \arrow["{\id \times_{\id} \comp}"{description}, curve={height=18pt}, from=5-1, to=6-2]
          \arrow["{\id \times_{\id} \comp}"', from=5-2, to=6-2]
        \end{tikzcd}
      \end{adjustbox}
      \vspace{1em}
      \newline
      using naturality of $\epsilon$ and associativity.
      Now we just need to observe that the left most vertical composite is equivalent to $(\id, \epsilon F \circ F \eta, \id)$ 
      which is by the triangle identities itself equivalent to the identity.
      The equivalence $\phi \circ \psi \equiv \id$ can be seen similarly.
    \end{proof}

    As a preparation to the functoriality of adjointing natural transformations between adjoints we 
    recall the so-called middle four interchange lemma, which states that the two ways one can compose 
    natural transformations between composable pairs of functors agree.

    \begin{lemma}
      \label{middlefourinterchangefornattrafos}
      For $F, G \colon A \to B$, $H, K \colon B \to C$ and $\alpha \colon F \twoto G$, 
      $\beta \colon H \twoto K$, the middle four interchange holds, \ie
      we have $(\beta G) ( H \alpha) \equiv (K \alpha) (\beta F)$.
    \end{lemma}
    \begin{proof}
      Let us denote the composite functor of
      \begin{equation*}
            \begin{tikzpicture}[diagram]
                \matrix[objects,narrow]{%
                    |(a)| \Fun(\Delta^1,\Fun(X,Y)) \times \Fun(\Delta^1,\Fun(Y,Z)) \\ 
                    |(b)| \Fun(\Delta^1 \times \Delta^1,\Fun(X,Y)) \times \Fun(\Delta^1 \times \Delta^1,\Fun(Y,Z)) \\
                    |(d)| \Fun(\Delta^1 \times \Delta^1,\Fun(X,Y) \times \Fun(Y,Z)) \\
                    |(c)| \Fun(\Delta^1, \Fun(\Delta^1,\Fun(X,Y) \times \Fun(Y,Z))) \\ 
                };
                \path[maps,->]
                    (a) edge node[left] {$(\pr_0)_\ast \times (\pr_1)_\ast$} (b)
                    (b) edge node[right] {$\equiv$} (d)
                    (d) edge node[right] {$\equiv$} (c)
                ;
            \end{tikzpicture}
      \end{equation*}
      by $(-) \times (-)$.
      Let us also abbreviate $C \coloneqq \Fun(X,Y) \times \Fun(Y,Z)$.
      Then the middle four interchange follows from the commutativity of 
      the following diagram.
      \[\begin{tikzcd}[cramped]
        {\Fun(\Delta^1,\Fun(X,Y)) \times \Fun(\Delta^1,\Fun(Y,Z))} \\
        {\Fun(\Delta^1, \Fun(\Delta^1,C))} & {\Fun(\Delta^1,C) \times_{C} \Fun(\Delta^1,C)} \\
        {\Fun(\Delta^1,C) \times_{C} \Fun(\Delta^1,C)} & {\Fun(\Delta^1,C)} \\
        & {\Fun(\Delta^1,\Fun(X,Z))}
        \arrow["{(-) \times (-)}", from=1-1, to=2-1]
        \arrow["{((\ev_0)_{\ast},\ev_1)}", from=2-1, to=2-2]
        \arrow["{(\ev_0, (\ev_1)_\ast)}"', from=2-1, to=3-1]
        \arrow["{\text{comp}}", from=2-2, to=3-2]
        \arrow["{\text{comp}}"', from=3-1, to=3-2]
        \arrow["{\Fun(\Delta^1,\circ)}", from=3-2, to=4-2]
      \end{tikzcd}\]
      where we use \autoref{commutativesquarescanbedecomposedintotwocommtriangles}.
    \end{proof}

    Using the middle four interchange we can prove the following preliminary result about 
    adjointable squares.

    \begin{lemma}
      \label{adjointablesquaregivesunitscommute}
      Let
      \begin{equation*}
            \begin{tikzpicture}[diagram]
                \matrix[objects]{%
                    |(a)| A \& |(b)| C \\
                    |(c)| B \& |(d)| D \\
                };
                \path[maps,->]
                    (a) edge node[above] {$H$} (b)
                    (a) edge node[left] {$F$} (c)
                    (b) edge node[right] {$G$} (d)
                    (c) edge node[below] {$K$} (d)
                ;
                \node at (barycentric cs:a=0.5,b=0.5,c=0.5,d=0.5) (phi) {$\equiv$};
                \node[maps,right] at (phi) {$\alpha$};
            \end{tikzpicture}
      \end{equation*}
      be a commutative square with adjunctions $F \adj U$ and $G \adj V$, such that the square is vertically 
      adjointable.
      Then the counits of the two adjunctions commute, \ie we can construct a commutative square 
      \begin{equation*}
            \begin{tikzpicture}[diagram]
                \matrix[objects]{%
                    |(a)| B \& |(b)| D \\
                    |(c)| A \times_B \Fun(\Delta^1,B) \& |(d)| C \times_D \Fun(\Delta^1,D) \\
                };
                \path[maps,->]
                    (a) edge node[above] {$K$} (b)
                    (a) edge node[left] {$(U, \epsilon)$} (c)
                    (b) edge node[right] {$(V, \tilde{\epsilon})$} (d)
                    (c) edge node[below] {$H \times_K K_\ast$} (d)
                ;
            \end{tikzpicture}
      \end{equation*}
      and the units commute, \ie we can construct a commutative square 
      \begin{equation*}
            \begin{tikzpicture}[diagram]
                \matrix[objects]{%
                    |(a)| A \& |(b)| C \\
                    |(c)| \Fun(\Delta^1,A) \times_A B \& |(d)| \Fun(\Delta^1,C) \times_C D \\
                };
                \path[maps,->]
                    (a) edge node[above] {$H$} (b)
                    (a) edge node[left] {$(\eta, F)$} (c)
                    (b) edge node[right] {$(\tilde{\eta},G)$} (d)
                    (c) edge node[below] {$H_\ast \times_H K$} (d)
                ;
            \end{tikzpicture}
      \end{equation*}
      where the bottom horizontal functor uses the invertible mate transformation $\mate_\alpha \colon HU \cong VK$.
    \end{lemma}
    \begin{proof}
      The adjointablility of the square means by definition that the mate transformation of $\alpha$,
      \ie the composite $\mate_\alpha \coloneqq (VK \epsilon) \circ (V \alpha U) \circ (\tilde{\eta} HU) \colon HU \twoto VK$
      is invertible.
      This tells us diagrammatically, that the following lift exists.
      \[\begin{tikzcd}
        B & D \\
        {A \times^{(F,\ev_0)}_B \Fun(\Delta^1,B)} \\
        {A \times^{(KF,\ev_0)}_D \Fun(\Delta^1,D)} & {\Fun(\Delta^1,C) \times^{(\ev_1,V)}_C D} \\
        {A \times^{(GH,\ev_0)}_D \Fun(\Delta^1,D)} & {C \times^{(G,\ev_0)}_D \Fun(\Delta^1,D)}
        \arrow[dashed, from=1-1, to=1-2]
        \arrow["{(U, \epsilon)}", from=1-1, to=2-1]
        \arrow["{\Delta_V}"', hook', from=1-2, to=3-2]
        \arrow["{\id \times_K K_\ast}", from=2-1, to=3-1]
        \arrow["{\alpha_{\ast}}", from=3-1, to=4-1]
        \arrow["{H \times_{\id} \id}", from=4-1, to=4-2]
        \arrow["\phi", from=4-2, to=3-2]
      \end{tikzcd}\]
      Postcomposing the diagram with $\pr_1 \colon \Fun(\Delta^1,C) \times^{(\ev_1,V)}_C D \to D$ 
      we can see that the lift must be equivalent to $K$.

      Now we can postcompose the inverse $\psi$ to $\phi$, by \autoref{firsthalfoftriangleidentityadjunctionsareequivalenttohomcatdef},
      and obtain
      \[\begin{tikzcd}
        B & D \\
        {A \times^{(F,\ev_0)}_B \Fun(\Delta^1,B)} && {C \times^{(G,\ev_0)}_D \Fun(\Delta^1,D)} \\
        {A \times^{(KF,\ev_0)}_D \Fun(\Delta^1,D)} & {\Fun(\Delta^1,C) \times^{(\ev_1,V)}_C D} \\
        {A \times^{(GH,\ev_0)}_D \Fun(\Delta^1,D)} & {C \times^{(G,\ev_0)}_D \Fun(\Delta^1,D)}
        \arrow["K", from=1-1, to=1-2]
        \arrow["{(U, \epsilon)}", from=1-1, to=2-1]
        \arrow["{(V, \tilde{\epsilon})}", from=1-2, to=2-3]
        \arrow["{\Delta_V}"', hook', from=1-2, to=3-2]
        \arrow["{\id \times_K K_\ast}", from=2-1, to=3-1]
        \arrow["{\alpha_{\ast}}", from=3-1, to=4-1]
        \arrow["\psi", from=3-2, to=2-3]
        \arrow["{H \times_{\id} \id}", from=4-1, to=4-2]
        \arrow[equals, from=4-2, to=2-3]
        \arrow["\phi", from=4-2, to=3-2]
      \end{tikzcd}\]
      so that after investing $\psi \circ \Delta_V \equiv (V,\tilde{\epsilon})$ we have our desired square.

      For the second commutative square we proceed in the following similiar way. 
      Using the middle four interchange \autoref{middlefourinterchangefornattrafos} 
      several times and then the triangle identities
      we can see that the mate of $\mate_\alpha$, \ie the transformation 
      $(\tilde{\epsilon} KF) \circ (G \mate_\alpha F) \circ (GH \eta)$, is 
      equivalent to $\alpha$ via the following equivalences, hence invertible.
      \begin{align*}
        (\tilde{\epsilon} KF) \circ (G \mate_\alpha F) \circ (GH \eta) &\equiv (\tilde{\epsilon} KF) \circ  (GVK \epsilon F) \circ (GV \alpha UF) \circ (G \tilde{\eta} HUF)  \circ (GH \eta) \\
        &\equiv (K \epsilon F) \circ (KF \eta) \circ \alpha \circ (\tilde{\epsilon} GH) \circ (G \tilde{\eta} H) \\
        &\equiv \alpha
      \end{align*}
      This means diagrammatically that we get the dashed lift in the following diagram.
      \[\begin{tikzcd}
        A & C \\
        {\Fun(\Delta^1,A) \times^{(\ev_1,U)}_A B} \\
        {\Fun(\Delta^1,C) \times^{(\ev_1,HU)}_C B} & {C \times^{(G,\ev_0)}_D \Fun(\Delta^1,D)} \\
        {\Fun(\Delta^1,C) \times^{(\ev_1,VK)}_C B} & {\Fun(\Delta^1,C) \times^{(\ev_1,V)}_C D}
        \arrow[dashed, from=1-1, to=1-2]
        \arrow["{(\eta, F)}", from=1-1, to=2-1]
        \arrow["{\Delta_G}"', hook', from=1-2, to=3-2]
        \arrow["{H_\ast \times_H \id}", from=2-1, to=3-1]
        \arrow["{(\mate_{\alpha})_\ast}", from=3-1, to=4-1]
        \arrow["{\id \times_{\id} K}", from=4-1, to=4-2]
        \arrow["\psi", from=4-2, to=3-2]
      \end{tikzcd}\]
      Postcomposing the diagram with the domain functor 
      $\pr_0 \colon C \times^{(G,\ev_0)}_D \Fun(\Delta^1,D) \to C$ 
      we can see that this lift must be equivalent to $H$.
      After postcomposing the diagram with the inverse equivalence 
      \begin{equation}
        \phi \colon C \times^{(G,\ev_0)}_D \Fun(\Delta^1,D) \to \Fun(\Delta^1,C) \times^{(\ev_1,V)}_C D
      \end{equation}
      of $\psi$ we obtain the commutative diagram
      \[\begin{tikzcd}
        A & C \\
        {\Fun(\Delta^1,A) \times^{(\ev_1,U)}_A B} && {\Fun(\Delta^1,C) \times^{(\ev_1,V)}_C D} \\
        {\Fun(\Delta^1,C) \times^{(\ev_1,HU)}_C B} & {C \times^{(G,\ev_0)}_D \Fun(\Delta^1,D)} \\
        {\Fun(\Delta^1,C) \times^{(\ev_1,VK)}_C B} & {\Fun(\Delta^1,C) \times^{(\ev_1,V)}_C D}
        \arrow["H", from=1-1, to=1-2]
        \arrow["{(\eta, F)}", from=1-1, to=2-1]
        \arrow["{(\tilde{\eta}, G)}", from=1-2, to=2-3]
        \arrow["{\Delta_G}"', hook', from=1-2, to=3-2]
        \arrow["{H_\ast \times_H \id}", from=2-1, to=3-1]
        \arrow["{(\mate_{\alpha})_\ast}", from=3-1, to=4-1]
        \arrow["\phi", from=3-2, to=2-3]
        \arrow["{\id \times_{\id} K}", from=4-1, to=4-2]
        \arrow[equals, from=4-2, to=2-3]
        \arrow["\psi", from=4-2, to=3-2]
      \end{tikzcd}\]
      and investing the fact that $(\tilde{\eta},G) \equiv \phi \circ \Delta_G$, 
      we obtain the desired commutative square.
    \end{proof}

    The following lemma is a preliminary version of the pullback stability of left adjoints in 
    $\Fun(\Delta^1,\Cat)$, where we restrict to the situation that the corresponding right adjoints 
    are additionally fully-faithful.

    \pagebreak

    \begin{lemma}
      \label{cubepullbackofladjwithffradjandadjfacesisladj}
        Consider a commutative cube of categories such that the top and bottom faces are
        pullbacks.
        \begin{equation}
            \begin{tikzpicture}[diagram]
                \matrix[objects,narrow] {%
                    |(a)| A_2 \& \& |(b)| A_1 \\
                    \& |(c)| A_0 \& \& |(d)| A \\
                    |(e)| B_2 \& \& |(f)| B_1 \\
                    \& |(g)| B_0 \& \& |(h)| B \\
                };
                \path[maps,->]
                (b) edge node[above]  {$$} (a)
                (f) edge node[above]  {$$} (e)

                (d) edge node[above]  {$$} (c)
                (h) edge node[above]  {$$} (g)

                (c) edge node[left]   {$$} (a)
                (g) edge node[left]   {$$} (e)

                (d) edge node[left]   {$$} (b)
                (h) edge node[left]   {$$} (f)

                (a) edge node[left]   {$F_2$} (e)
                (c) edge node[below left]   {$F_0$} (g)
                (b) edge node[below left]   {$F_1$} (f)
                (d) edge node[left]   {$F$} (h)
                ;
            \end{tikzpicture}
        \end{equation}
        If the $F_i$ have fully-faithful right adjoints and the back face as well as
        the left face are horizontally adjointable, then the pullback induced 
        functor $F$ also has a fully-faithful right adjoint and the right and front faces 
        are also horizontally adjointable.
    \end{lemma}
    \begin{proof}
      Let us denote the right adjoints by $U_i$ and their units by $\eta_i$.
      We produce the candidate for the right adjoint $U$ of $F$, 
      as well as a natural equivalence $\epsilon \colon FU \equiv \id$, via the horizontally levelwise
      pullback diagram
      \[\begin{tikzcd}[cramped]
        {B_2} && {B_1} \\
        & {B_0} && B \\
        {A_2} && {A_1} \\
        & {A_0} && A \\
        {B_2} && {B_1} \\
        & {B_0} && B
        \arrow["{U_2}", hook, from=1-1, to=3-1]
        \arrow[curve={height=12pt}, Rightarrow, no head, from=1-1, to=5-1]
        \arrow[from=1-3, to=1-1]
        \arrow["{U_1}"{pos=0.3}, hook, from=1-3, to=3-3]
        \arrow[curve={height=12pt}, Rightarrow, no head, from=1-3, to=5-3]
        \arrow[from=2-2, to=1-1]
        \arrow["{U_0}"{pos=0.3}, hook, from=2-2, to=4-2]
        \arrow[curve={height=12pt}, Rightarrow, no head, from=2-2, to=6-2]
        \arrow[from=2-4, to=1-3]
        \arrow[from=2-4, to=2-2]
        \arrow["{\exists ! U}", dashed, from=2-4, to=4-4]
        \arrow[curve={height=-30pt}, Rightarrow, no head, from=2-4, to=6-4]
        \arrow["{F_2}", from=3-1, to=5-1]
        \arrow[from=3-3, to=3-1]
        \arrow["{F_1}"{pos=0.3}, from=3-3, to=5-3]
        \arrow[from=4-2, to=3-1]
        \arrow["{F_0}"{pos=0.3}, from=4-2, to=6-2]
        \arrow[from=4-4, to=3-3]
        \arrow[from=4-4, to=4-2]
        \arrow["F", from=4-4, to=6-4]
        \arrow[from=5-3, to=5-1]
        \arrow[from=6-2, to=5-1]
        \arrow[from=6-4, to=5-3]
        \arrow[from=6-4, to=6-2]
      \end{tikzcd}\]
      whose composite cube is, as indicated, the vertical identity cube.
      To produce a candidate for the unit of this adjunction we look at the 
      following cube induced on the back and the left hand vertical faces by 
      \autoref{adjointablesquaregivesunitscommute} and whose top and bottom faces are pullbacks.
      \[\begin{tikzcd}[cramped]
        {A_2} && {A_1} \\
        & {A_0} && A \\
        {\Fun(\Delta^1,A_2)} && {\Fun(\Delta^1,A_1)} \\
        & {\Fun(\Delta^1,A_0)} && {\Fun(\Delta^1,A)}
        \arrow["{\eta_2}", from=1-1, to=3-1]
        \arrow[from=1-3, to=1-1]
        \arrow["{\eta_1}"{pos=0.3}, from=1-3, to=3-3]
        \arrow[from=2-2, to=1-1]
        \arrow["{\eta_0}"{pos=0.3}, from=2-2, to=4-2]
        \arrow[from=2-4, to=1-3]
        \arrow[from=2-4, to=2-2]
        \arrow["{\exists ! \eta}", dashed, from=2-4, to=4-4]
        \arrow[from=3-3, to=3-1]
        \arrow[from=4-2, to=3-1]
        \arrow[from=4-4, to=3-3]
        \arrow[from=4-4, to=4-2]
      \end{tikzcd}\]
      Note that, again by pullback, this natural transformation has the right domain and codomain.
      To deduce that $\eta U$ is invertible we use the following horizontally levelwise pullback
      diagram.

      
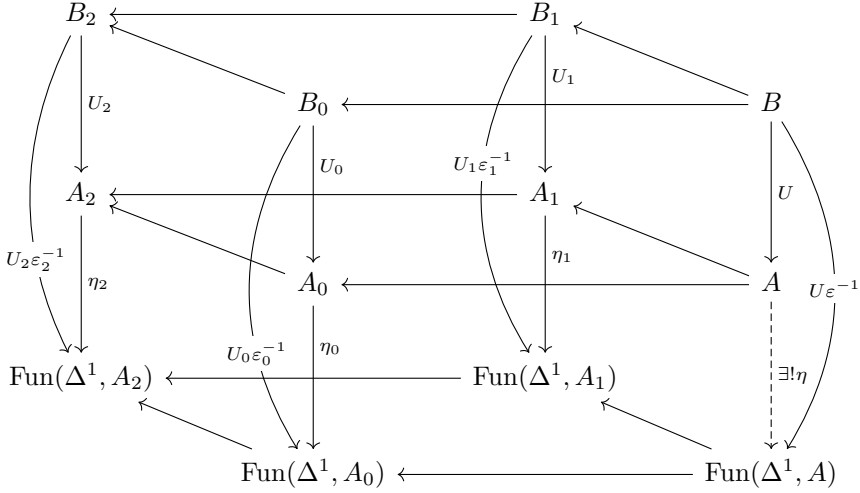
\[\begin{tikzcd}
        {B_2} && {B_1} \\
        & {B_0} && B \\
        {A_2} && {A_1} \\
        & {A_0} && A \\
        {\Fun(\Delta^1,A_2)} && {\Fun(\Delta^1,A_1)} \\
        & {\Fun(\Delta^1,A_0)} && {\Fun(\Delta^1,A)}
        \arrow["{U_2}", from=1-1, to=3-1]
        \arrow["{U_2 \epsilon^{-1}_2}"{description, pos=0.7}, curve={height=24pt}, from=1-1, to=5-1]
        \arrow[from=1-3, to=1-1]
        \arrow["{U_1}"{pos=0.3}, from=1-3, to=3-3]
        \arrow["{U_1 \epsilon^{-1}_1}"{description, pos=0.4}, curve={height=30pt}, from=1-3, to=5-3]
        \arrow[from=2-2, to=1-1]
        \arrow["{U_0}"{pos=0.3}, from=2-2, to=4-2]
        \arrow["{U_0 \epsilon^{-1}_0}"{description, pos=0.7}, curve={height=30pt}, from=2-2, to=6-2]
        \arrow[from=2-4, to=1-3]
        \arrow[from=2-4, to=2-2]
        \arrow["U", from=2-4, to=4-4]
        \arrow["{U \epsilon^{-1}}"{description}, curve={height=-30pt}, from=2-4, to=6-4]
        \arrow["{\eta_2}", from=3-1, to=5-1]
        \arrow[from=3-3, to=3-1]
        \arrow["{\eta_1}"{pos=0.3}, from=3-3, to=5-3]
        \arrow[from=4-2, to=3-1]
        \arrow["{\eta_0}"{pos=0.3}, from=4-2, to=6-2]
        \arrow[from=4-4, to=3-3]
        \arrow[from=4-4, to=4-2]
        \arrow["{\exists ! \eta}", dashed, from=4-4, to=6-4]
        \arrow[from=5-3, to=5-1]
        \arrow[from=6-2, to=5-1]
        \arrow[from=6-4, to=5-3]
        \arrow[from=6-4, to=6-2]
      \end{tikzcd}\]
      The invertibility of $F \eta$ can be proven similarly, via the 
      diagram
      
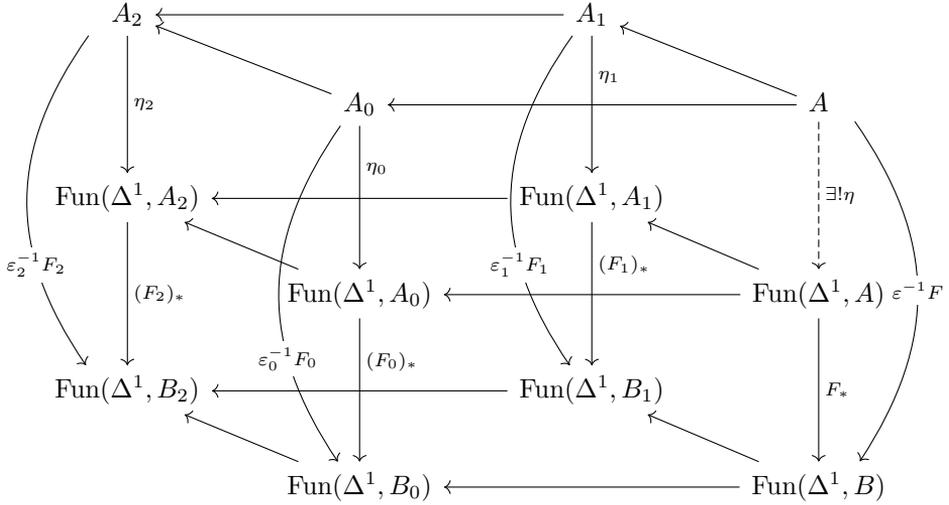
\[\begin{tikzcd}
        {A_2} && {A_1} \\
        & {A_0} && A \\
        {\Fun(\Delta^1,A_2)} && {\Fun(\Delta^1,A_1)} \\
        & {\Fun(\Delta^1,A_0)} && {\Fun(\Delta^1,A)} \\
        {\Fun(\Delta^1,B_2)} && {\Fun(\Delta^1,B_1)} \\
        & {\Fun(\Delta^1,B_0)} && {\Fun(\Delta^1,B)}
        \arrow["{\eta_2}", from=1-1, to=3-1]
        \arrow["{\epsilon^{-1}_2 F_2}"{description, pos=0.7}, shift right=4, curve={height=30pt}, from=1-1, to=5-1]
        \arrow[from=1-3, to=1-1]
        \arrow["{\eta_1}"{pos=0.3}, from=1-3, to=3-3]
        \arrow["{\epsilon^{-1}_1 F_1}"{description, pos=0.7}, shift right, curve={height=30pt}, from=1-3, to=5-3]
        \arrow[from=2-2, to=1-1]
        \arrow["{\eta_0}"{pos=0.3}, from=2-2, to=4-2]
        \arrow["{\epsilon^{-1}_0 F_0}"{description, pos=0.7}, shift right, curve={height=30pt}, from=2-2, to=6-2]
        \arrow[from=2-4, to=1-3]
        \arrow[from=2-4, to=2-2]
        \arrow["{\exists ! \eta}", dashed, from=2-4, to=4-4]
        \arrow["{\epsilon^{-1} F}"{description}, shift left=4, curve={height=-30pt}, from=2-4, to=6-4]
        \arrow["{(F_2)_\ast}", from=3-1, to=5-1]
        \arrow[from=3-3, to=3-1]
        \arrow["{(F_1)_\ast}"{pos=0.3}, from=3-3, to=5-3]
        \arrow[from=4-2, to=3-1]
        \arrow["{(F_0)_\ast}"{pos=0.3}, from=4-2, to=6-2]
        \arrow[from=4-4, to=3-3]
        \arrow[from=4-4, to=4-2]
        \arrow["{F_\ast}", from=4-4, to=6-4]
        \arrow[from=5-3, to=5-1]
        \arrow[from=6-2, to=5-1]
        \arrow[from=6-4, to=5-3]
        \arrow[from=6-4, to=6-2]
      \end{tikzcd}\]
      In particular, this proves that $\eta$ whiskered with the 
      functors $ A \to A_i \xto{F_i} B_i$ is invertible, which also 
      proves the adjointability claim as the counits of all these 
      adjunctions are invertible.
    \end{proof}

    As a first application we can deduce that some more functors and commutative squares give 
    cocartesian fibrations and functors.

    \begin{lemma}
      \label{freecocartesianfibrationsonfunctorsandsquaresoffunctorsinducecocartesianfunctorsbetweenthose}
      For any functor $F \colon A \to C$ the functor $\ev_1 \colon A \times_{C} \Fun(\Delta^1,C) \to C$
      is a cocartesian fibration.
      Furthermore, for any commutative square of functors
      \begin{equation}
            \begin{tikzpicture}[diagram]
                \matrix[objects] {%
                |(a)| A \& |(b)| B \\
                |(c)| C \& |(d)| D \\
                };
                \path[maps,->]
                (a) edge node[above]  {$H$} (b)
                (c) edge node[below]  {$K$} (d)
                (a) edge node[left]   {$F$} (c)
                (b) edge node[right]  {$G$} (d)
                ;
            \end{tikzpicture}
      \end{equation}
      the functor $H \times_{K} K_\ast$ in the commutative square
      \begin{equation}
            \begin{tikzpicture}[diagram]
                \matrix[objects] {%
                  |(a)| A \times_{B} \Fun(\Delta^1,B) \& |(b)| C \times_{D} \Fun(\Delta^1,D) \\
                |(c)| B \& |(d)| D \\
                };
                \path[maps,->]
                (a) edge node[above]  {$H \times_{K} K_\ast$} (b)
                (c) edge node[below]  {$K$} (d)
                (a) edge node[left]   {$\ev_1$} (c)
                (b) edge node[right]  {$\ev_1$} (d)
                ;
            \end{tikzpicture}
      \end{equation}
      is a cocartesian functor.
    \end{lemma}
    \begin{proof}
      We know that $\ev_1 \colon \Fun(\Delta^1, A) \fibration A$ is a cocartesian fibration, and the
      evaluation functors $\ev_0 \colon \Fun(\Delta^1,C) \to C$ always have fully-faithful 
      left adjoints $\Delta_C$. Thus we can deduce from the commutative cube
      \[\begin{tikzcd}[cramped,column sep=tiny]
        {\Fun(\Delta^1,B)} && {\Fun(\Delta^1,\Fun(\Delta^1,B))} \\
        & {\Fun(\Delta^1,A)} && {\Fun(\Delta^1,A \times_B \Fun(\Delta^1,B))} \\
        B && {\Fun(\Delta^1,B) \times_B \Fun(\Delta^1,B)} \\
        & A && {A \times_B \Fun(\Delta^1,B) \times_B \Fun(\Delta^1,B)}
        \arrow["{\ev_0}"', from=1-1, to=3-1]
        \arrow["{\ev_0}", from=1-3, to=1-1]
        \arrow["{(\ev_0, (\ev_1)_\ast)}"'{pos=0.7}, from=1-3, to=3-3]
        \arrow["{F_\ast}"', from=2-2, to=1-1]
        \arrow["{\ev_0}"'{pos=0.7}, from=2-2, to=4-2]
        \arrow[from=2-4, to=1-3]
        \arrow[from=2-4, to=2-2]
        \arrow["{(\ev_0, (\pr_1)_\ast)}"', from=2-4, to=4-4]
        \arrow[from=3-3, to=3-1]
        \arrow["F", from=4-2, to=3-1]
        \arrow[from=4-4, to=3-3]
        \arrow[from=4-4, to=4-2]
      \end{tikzcd}\]
      with top and bottom faces pullbacks and \autoref{cubepullbackofladjwithffradjandadjfacesisladj} 
      the first statement.
      For the second statement we proceed in two steps. As a first step 
      we factorize the given square into a triangle and a pullback as in
      \[\begin{tikzcd}[cramped]
        A & {C \times_D B} & B \\
        & C & D
        \arrow[dashed, from=1-1, to=1-2]
        \arrow["H", curve={height=-18pt}, from=1-1, to=1-3]
        \arrow["F", from=1-1, to=2-2]
        \arrow[from=1-2, to=1-3]
        \arrow[from=1-2, to=2-2]
        \arrow["G", from=1-3, to=2-3]
        \arrow[""{name=0, anchor=center, inner sep=0}, "K"', from=2-2, to=2-3]
        \arrow["\lrcorner"{anchor=center, pos=0.125}, draw=none, from=1-2, to=0]
      \end{tikzcd}\]
      Then for the pullback square we deduce the adjointability claim from 
      \autoref{cubepullbackofladjwithffradjandadjfacesisladj} applied to the 
      cube
      \begin{center}
        \begin{adjustbox}{scale=0.75}
          \begin{tikzcd}[column sep={11em,between origins}]
          {\Fun(\Delta^1, C \times_D B) \times_{\Fun(\Delta^1,C)} \Fun(\Delta^1, \Fun(\Delta^1,C))} 
              \arrow[rr] 
              \arrow[dr]
              \arrow[dd]
          && 
          {\Fun(\Delta^1, \Fun(\Delta^1,C))} 
              \arrow[dd] 
              \arrow[dr]
          \\
          & 
          {\Fun(\Delta^1, B) \times_{\Fun(\Delta^1,D)} \Fun(\Delta^1, \Fun(\Delta^1,D))} 
              \arrow[crossing over, rr] 
          && 
          {\Fun(\Delta^1, \Fun(\Delta^1,D))} 
              \arrow[dd]
          \\
          {(C \times_D B) \times_C \Fun(\Delta^1,C) \times_C \Fun(\Delta^1,C)} 
              \arrow[rr] 
              \arrow[dr]
          && 
          {\Fun(\Delta^1,C) \times_C \Fun(\Delta^1,C)} 
              \arrow[dr]
          \\
          & 
          {B \times_D \Fun(\Delta^1,D) \times_D \Fun(\Delta^1,D)} 
              \arrow[rr]
          && 
          {\Fun(\Delta^1,D) \times_D \Fun(\Delta^1,D)}
          \arrow[crossing over, from=2-2, to=4-2]
          \end{tikzcd}
        \end{adjustbox}
      \end{center}
      with top and bottom faces pullbacks and the fact that the claim holds when 
      $F = \id$ and $G = \id$ by
      \autoref{functorsinducecocartesianfunctorsbetweencodomainfibrations},
      \ie the front face is adjointable.
      For the remaining triangle we can deduce from the cube
      \begin{center}
        \begin{adjustbox}{scale=0.75}
          \begin{tikzcd}[column sep={12em,between origins}]
          {\Fun(\Delta^1, A) \times_{\Fun(\Delta^1,C)} \Fun(\Delta^1, \Fun(\Delta^1,C))} 
              \arrow[rr] 
              \arrow[dr]
              \arrow[dd] 
          && 
          {\Fun(\Delta^1, A)} 
              \arrow[dd] 
              \arrow[dr]
          \\
          & 
          {\Fun(\Delta^1, C \times_D B) \times_{\Fun(\Delta^1,C)} \Fun(\Delta^1, \Fun(\Delta^1,C))} 
              \arrow[crossing over, rr]
          && 
          {\Fun(\Delta^1,C \times_D B)} 
              \arrow[dd]
          \\
          {A \times_C \Fun(\Delta^1,C) \times_C \Fun(\Delta^1,C)} 
              \arrow[rr] 
              \arrow[dr]
          && 
          A 
              \arrow[dr]
          \\
          & 
          {(C \times_D B) \times_C \Fun(\Delta^1,C) \times_C \Fun(\Delta^1,C)} 
              \arrow[rr]
          && 
          {C \times_D B}
          \arrow[crossing over, from=2-2, to=4-2]
          \end{tikzcd}
        \end{adjustbox}
      \end{center}
      with top and bottom faces pullbacks in the same way that the back 
      face is adjointable.
    \end{proof}

    We continue this subsection with the most basic stability property of right adjoints with invertible counit 
    we will use and from which all the other stability properties we want to invoke can be 
    derived from.
    \pullbackoflalis*
    \begin{proof}
        First let us define the functor $G$ that is supposed to be left adjoint section to $V$
        and the invertible unit $\eta' \colon \id \equiv VG$
        via the following pullback
        \[\begin{tikzcd}[sep=large]
          Y & B \\
          X & A \\
          Y & B
          \arrow["K", from=1-1, to=1-2]
          \arrow["{\exists ! G}"', dashed, from=1-1, to=2-1]
          \arrow["\lrcorner"{anchor=center, pos=0.125}, draw=none, from=1-1, to=2-2]
          \arrow[""{name=0, anchor=center, inner sep=0}, shift right=3, curve={height=30pt}, equals, from=1-1, to=3-1]
          \arrow["F", from=1-2, to=2-2]
          \arrow[""{name=1, anchor=center, inner sep=0}, shift left=3, curve={height=-30pt}, equals, from=1-2, to=3-2]
          \arrow["H", from=2-1, to=2-2]
          \arrow["V"', from=2-1, to=3-1]
          \arrow["\lrcorner"{anchor=center, pos=0.125}, draw=none, from=2-1, to=3-2]
          \arrow["U", from=2-2, to=3-2]
          \arrow["K", from=3-1, to=3-2]
          \arrow["{\equiv \exists ! \eta'}"{description}, draw=none, from=0, to=2-1]
          \arrow["{\equiv \eta}"{description}, draw=none, from=1, to=2-2]
        \end{tikzcd}\]
        Next we define the counit $\epsilon'$ of the adjunction also via the following pullback.
        \[\begin{tikzcd}
          Y && B \\
          & X && A \\
          && {\Fun(\Delta^1,X)} && {\Fun(\Delta^1,A)} \\
          & Y && B \\
          && {\Fun(\Delta^1,Y)} && {\Fun(\Delta^1,B)}
          \arrow["K", from=1-1, to=1-3]
          \arrow["G", from=1-1, to=2-2]
          \arrow[equals, from=1-1, to=4-2]
          \arrow["F", from=1-3, to=2-4]
          \arrow[equals, from=1-3, to=4-4]
          \arrow["{\exists! \varepsilon'}"', dashed, from=2-2, to=3-3]
          \arrow["V", from=2-2, to=4-2]
          \arrow["\varepsilon", from=2-4, to=3-5]
          \arrow["U"{pos=0.8}, from=2-4, to=4-4]
          \arrow["\lrcorner"{anchor=center, pos=0.125}, draw=none, from=3-3, to=5-5]
          \arrow["{U_\ast}", from=3-5, to=5-5]
          \arrow["K"{pos=0.4}, from=4-2, to=4-4]
          \arrow["{\eta'}"', from=4-2, to=5-3]
          \arrow["\eta", from=4-4, to=5-5]
          \arrow["{K_\ast}"', from=5-3, to=5-5]
          \arrow["{H_\ast}"{pos=0.4}, from=3-3, to=3-5,crossing over]
          \arrow["{V_\ast}"{pos=0.7}, from=3-3, to=5-3,crossing over]
          \arrow["H"', from=2-2, to=2-4,crossing over]
        \end{tikzcd}\]
        where the bottom square comes from \autoref{adjointablesquaregivesunitscommute}.
        Now, as $\epsilon F$ is equivalent to $F \eta^{-1}$ we also get by pullback
        that $\epsilon' G$ is equivalent to $G (\eta')^{-1}$ and thus that $G$ really 
        is a left adjoint to $V$ with counit $\epsilon'$ and invertible unit $\eta'$.
        To see vertical adjointablility of the original pullback square 
        it is sufficient to see that $\epsilon HG$ is invertible as the 
        adjunction $G \adj V$ has invertible unit. But by the above diagram 
        we already know that $\epsilon HG \equiv H_\ast \circ \epsilon' G$ 
        and we just established invertibility of $\epsilon' G$.
    \end{proof}

    As a corollary we can immediately deduce that cocartesian fibrations are stable under pullback.
    \begin{corollary}
      \label{pullbackofcocartfibs}
      Let
      \begin{equation*}
          \begin{tikzpicture}[diagram]
              \matrix[objects]{%
                  |(a)| Y \& |(b)| X \\
                  |(c)| A \& |(d)| B \\
              };
              \path[maps,->]
                  (a) edge node[above] {$H$} (b)
                  (a) edge node[left] {$q$} (c)
                  (b) edge[->>] node[right] {$p$} (d)
                  (c) edge node[below] {$K$} (d)
              ;
              \node at (barycentric cs:a=0.8,b=0.3,c=0.3) (phi) {\mbox{\LARGE{$\lrcorner$}}};
          \end{tikzpicture}
      \end{equation*}
      be a pullback such that $p$ is a cocartesian fibration.
      Then its pullback $q$ is also a cocartesian fibration and the 
      functor $H$ is a cocartesian functor over $K$.
    \end{corollary}
    \begin{proof}
      From the pullback we get another pullback
      \begin{equation*}
          \begin{tikzpicture}[diagram]
              \matrix[objects]{%
                  |(a)| \Fun(\Delta^1,Y) \& |(b)| \Fun(\Delta^1,X) \\
                  |(c)| Y \times_A \Fun(\Delta^1,A) \& |(d)| X \times_B \Fun(\Delta^1,B) \\
              };
              \path[maps,->]
                  (a) edge node[above] {$H_\ast$} (b)
                  (a) edge node[left] {$(\ev_0,q_\ast)$} (c)
                  (b) edge node[right] {$(\ev_0,p_\ast)$} (d)
                  (c) edge node[below] {$H \times_K K_\ast$} (d)
              ;
              \node at (barycentric cs:a=0.8,b=0.3,c=0.3) (phi) {\mbox{\LARGE{$\lrcorner$}}};
          \end{tikzpicture}
      \end{equation*}
      to which we can apply \autoref{pullbackofadjunctions} 
      to get the desired statement.
    \end{proof}

    Another corollary of \autoref{pullbackofadjunctions} 
    is that the cocartesian fibrations $\ev_1 \colon A \times_B \Fun(\Delta^1,B) \to B$ inherit the 
    adjunction $\Delta_B \adj \ev_0$ from $\ev_1 \colon \Fun(\Delta^1,B) \to B$.

    \begin{corollary}
      \label{freecocartfibinheritsDeltaadjunction}
      For a functor $F \colon A \to B$ the functor $\pr_0 \colon A \times_B \Fun(\Delta^1,B) \to A$
      has as a fully-faithful left adjoint the functor 
      $\Delta_F \coloneqq (\id, \Delta_B \circ F) \colon A \to A \times_B \Fun(\Delta^1,B)$.
    \end{corollary}
    \begin{proof}
      We apply \autoref{pullbackofadjunctions} 
      to the defining pullback 
      \begin{equation*}
        \begin{tikzpicture}[diagram]
            \matrix[objects]{%
                |(a)| A \times_B \Fun(\Delta^1,B) \& |(b)| A \\
                |(c)| \Fun(\Delta^1,B) \& |(d)| B \\
            };
            \path[maps,->]
                (a) edge node[above] {$\pr_0$} (b)
                (a) edge node[left] {$$} (c)
                (b) edge node[right] {$$} (d)
                (c) edge node[below] {$\ev_0$} (d)
            ;
            \node at (barycentric cs:a=0.8,b=0.3,c=0.3) (phi) {\mbox{\LARGE{$\lrcorner$}}};
        \end{tikzpicture}
      \end{equation*}
    \end{proof}

    With these statements we can prove the equivalence of adjunctions given in terms of units/counits 
    and Hom-equivalences.

    \begin{lemma}
      \label{triangleidentityadjunctionsareequivalenttohomcatdef}
      Let $F \colon B \to A$ and $U \colon A \to B$ be two functors.
      Then the following pieces of data are equivalent.
      \begin{enumerate}
        \item A unit natural transformation $\eta \colon \id \twoto UF$,
          a counit natural transformation $\epsilon \colon FU \twoto \id$,
          such that the triangle-identities hold, \ie 
          $\epsilon F \circ F \eta \equiv \id$ and $U \epsilon \circ \eta U \equiv \id$

        \item An equivalence
          \begin{equation}
                \begin{tikzpicture}[diagram]
                    \matrix[objects] {%
                      |(a)| B \times_{A} \Fun(\Delta^1,A) \& \& |(b)| \Fun(\Delta^1,B) \times_{B} A \\
                      \& |(c)| B \times A  \\
                    };
                    \path[maps,->]
                    (a) edge node[above]  {$\equiv$} (b)
                    (a) edge node[below left]   {$\ev$} (c)
                    (b) edge node[below right]  {$\ev$} (c)
                    ;
                \end{tikzpicture}
          \end{equation}
          over the canonical functors to $B \times A$.
      \end{enumerate}
    \end{lemma}
    \begin{proof}
      We already deduced the second statement from the first in 
      \autoref{firsthalfoftriangleidentityadjunctionsareequivalenttohomcatdef}.

      For the converse implication we assume to be given inverse equivalences $\phi$ and $\psi$.
      Then we define the unit and the counit as follows.
      \begin{equation*}
        \eta \coloneqq B \xto{\Delta_F} B \times_A \Fun(\Delta^1,A) \xto{\phi} \Fun(\Delta^1,B) \times_B A \xto{\pr_0} \Fun(\Delta^1,B)
      \end{equation*}
      \begin{equation*}
        \epsilon \coloneqq A \xto{\Delta_U} \Fun(\Delta^1,B) \times_B A \xto{\psi} B \times_A \Fun(\Delta^1,A) \xto{\pr_1} \Fun(\Delta^1,A)
      \end{equation*}
      We know from \autoref{freecocartesianfibrationsonfunctorsandsquaresoffunctorsinducecocartesianfunctorsbetweenthose}
      that $\ev_1 \colon B \times_A \Fun(\Delta^1,A) \to A$ is a cocartesian fibration.
      Now we observe that the counit $e \colon \Delta_F \ev_0 \twoto \id$ of the adjunction $\Delta_F \adj \ev_0$, 
      from \autoref{freecocartfibinheritsDeltaadjunction},
      is the $\ev_1$-cocartesian lift of 
      \[\begin{tikzcd}[cramped]
        {B \times_A \Fun(\Delta^1,A)} & B & {B \times_A \Fun(\Delta^1,A)} \\
        && A
        \arrow["{\ev_0}"', from=1-1, to=1-2]
        \arrow[""{name=0, anchor=center, inner sep=0}, "{\ev_1}"', curve={height=12pt}, from=1-1, to=2-3]
        \arrow["{\Delta_F}", from=1-2, to=1-3]
        \arrow["F"', from=1-2, to=2-3]
        \arrow["{\ev_1}", from=1-3, to=2-3]
        \arrow[shorten <=4pt, shorten >=4pt, Rightarrow, from=1-2, to=0]
      \end{tikzcd}\]
      as it is exactly
      \begin{equation*}
        B \times_A \Fun(\Delta^1,A) \xto{(\ev_0 \Delta_F,\pr_1)} B \times_A \Fun(\Delta^1,A) \times_A \Fun(\Delta^1,A) \xto{\cocart_{\ev_1}} \Fun(\Delta^1, B \times_A \Fun(\Delta^1,A))
      \end{equation*}
      just as for $\ev_1 \colon \Fun(\Delta^1,A) \to A$ in \autoref{evaluationfunctorsarefibrations}.
      But as $\phi$ is an equivalence it is in particular a cocartesian functor between cocartesian
      fibrations over $A$ as follows.
      \[\begin{tikzcd}[cramped]
        {B \times_A \Fun(\Delta^1,A)} && {\Fun(\Delta^1,B) \times_B A} \\
        & {B \times A} \\
        & A
        \arrow["{\phi \equiv}", from=1-1, to=1-3]
        \arrow[from=1-1, to=2-2]
        \arrow["{\ev_1}"', curve={height=12pt}, from=1-1, to=3-2]
        \arrow[from=1-3, to=2-2]
        \arrow["{\pr_1}", curve={height=-12pt}, from=1-3, to=3-2]
        \arrow["{\pr_A}", from=2-2, to=3-2]
      \end{tikzcd}\]
      Hence it sends the $\ev_1$-cocartesian morphism $e$ to a $\pr_1$-cocartesian morphism $\phi e$.
      Thus $\phi$ itself is a $\pr_1$-cocartesian lift of
      \[\begin{tikzcd}[cramped]
        {B \times_A \Fun(\Delta^1,A)} & B & {B \times_A \Fun(\Delta^1,A)} & {\Fun(\Delta^1,B) \times_B A} \\
        &&& A
        \arrow["{\pr_0}", from=1-1, to=1-2]
        \arrow[""{name=0, anchor=center, inner sep=0}, "{\ev_1}"', curve={height=12pt}, from=1-1, to=2-4]
        \arrow["{\Delta_F}", from=1-2, to=1-3]
        \arrow["F", from=1-2, to=2-4]
        \arrow["\phi", from=1-3, to=1-4]
        \arrow["{\ev_1}", from=1-3, to=2-4]
        \arrow["{\pr_1}", from=1-4, to=2-4]
        \arrow[shorten <=11pt, shorten >=19pt, Rightarrow, from=1-2, to=0]
      \end{tikzcd}\]
      But cocartesian lifts for the pulled back cocartesian fibration $\Fun(\Delta^1,B) \times_B A \to A$ 
      are essentially given by applying $U$ and then composing.
      Hence $\phi$ itself is equivalent to the top horizontal composite in the 
      following commutative diagram.
      \[\begin{tikzcd}[cramped]
        & { \Fun(\Delta^1,B) \times_B A \times_A \Fun(\Delta^1,A)} \\
        {B \times_A \Fun(\Delta^1,A)} && {\Fun(\Delta^1,B) \times_B A} \\
        & {\Fun(\Delta^1,B) \times_B \Fun(\Delta^1,B) \times_B A}
        \arrow["{\comp \circ (\id \times_{\id} U_\ast)}", from=1-2, to=2-3]
        \arrow["{\id \times_{\id} (U_\ast, \ev_1)}"', from=1-2, to=3-2]
        \arrow["{( \phi \circ \Delta_F ) \times_{\id} \id}", from=2-1, to=1-2]
        \arrow["{\eta \times_U (U_\ast, \ev_1)}"', from=2-1, to=3-2]
        \arrow["{\comp \times_{\id} \id}"', from=3-2, to=2-3]
      \end{tikzcd}\]
      This we can now precompose with $\psi \circ \Delta_U$ to obtain
      \[\begin{tikzcd}[cramped]
        A & {\Fun(\Delta^1,B) \times_B A} \\
        \\
        {B \times_A \Fun(\Delta^1,A)} & {\Fun(\Delta^1,B) \times_B \Fun(\Delta^1,B) \times_B A} & {\Fun(\Delta^1,B) \times_B A}
        \arrow["{\Delta_U}"', from=1-1, to=1-2]
        \arrow["{(U, \varepsilon)}"', from=1-1, to=3-1]
        \arrow["\psi"', from=1-2, to=3-1]
        \arrow[Rightarrow, no head, from=1-2, to=3-3]
        \arrow["{\eta \times_U (U_\ast, \ev_1)}"', from=3-1, to=3-2]
        \arrow[""{name=0, anchor=center, inner sep=0}, "\phi", curve={height=-24pt}, from=3-1, to=3-3]
        \arrow["{\comp \times_{\id} \id}"'{pos=0.9}, from=3-2, to=3-3]
        \arrow["\equiv", draw=none, from=0, to=3-2]
      \end{tikzcd}\]
      \ie the desired triangle identity $U\epsilon \circ \eta U \equiv \id$.
      The other triangle identity can be proven similarly.
    \end{proof}

    The next statement gives a convenient description of the essential image of a fully-faithful 
    functor in the case that it is a right adjoint.

    \begin{lemma}
      \label{characterizationofessimoffullandfaithfulradj}
      Let $U \colon A \to B$ be a right adjoint with left adjoint $F$ and unit $\eta$ such that the counit $\epsilon$ is invertible.
      Then $U$ is fully-faithful and we can characterize its essential image as exactly 
      those objects $b$ in $B$ whose unit component $\eta_b \colon b \to UFb$ is invertible.
    \end{lemma}
    \begin{proof}
      If $\eta_b$ is invertible for some $b$ in $B$ then it is also in the essential image of $U$ 
      by this equivalence.
      Conversely, if we know $b \equiv U(a)$ for some object $a$ of $A$, then the naturality square of $\eta$ 
      for this equivalence in $B$ is 
      \begin{equation}
            \begin{tikzpicture}[diagram]
                \matrix[objects] {%
                |(a)| b \& |(b)| Ua \\
                |(c)| UFb \& |(d)| UFUa \\
                };
                \path[maps,->]
                (a) edge node[above]  {$\equiv$} (b)
                (c) edge node[below]  {$\equiv$} (d)
                (a) edge node[left]   {$\eta_b$} (c)
                (b) edge node[right]  {$\eta_{Ua}$} (d)
                ;
            \end{tikzpicture}
      \end{equation}
      But by invertibility of the counit $\epsilon$ and the triangle identities of the adjunction,
      $\eta_{Ua}$ is also invertible.
      Hence we can deduce from the naturality square that $\eta_b$ is also invertible.
    \end{proof}

    The following lemma is an explicit characterization of when functors factor through fully-faithful 
    left adjoints, in terms of the adjunction data.

    \begin{lemma}
      \label{liftingalongffladjisequivtohavingpostwhiskeringwithcounitinvertible}
      Let $U \colon A \to B$ be a right adjoint with fully-faithful left adjoint
      $F$ and counit $\epsilon \colon FU \twoto \id$, and $T \colon X \to A$
      another functor. Then the following are equivalent.
      \begin{enumerate}
        \item The whiskering $\epsilon T$ is invertible.
        \item There exists a lift
          \begin{equation}
            \begin{tikzpicture}[diagram]
                \matrix[objects] {%
                |(a)|  \& |(b)| B \\
                |(c)| X \& |(d)| A \\
                };
                \path[maps,->]
                (c) edge[dashed] node[above]  {$\exists l$} (b)
                (c) edge node[below]  {$T$} (d)
                (b) edge node[right]  {$F$} (d)
                ;
                \node at (barycentric cs:d=0.5,b=0.5,c=0.5) (phi) {$\equiv$};
                \node[maps,right] at (phi) {$\alpha$};
            \end{tikzpicture}
          \end{equation}
      \end{enumerate}
    \end{lemma}
    \begin{proof}
      To prove the second assertion from the first we can just choose $U \circ T$ and the invertible natural
      transformation $\epsilon T \colon F U T \equiv T$ to be the lift.
      If we are given a lift $l$ as in the second assertion, then we proceed in the following way.
      The triangle identity $\epsilon F \circ F \eta \equiv \id$ and the invertibility of $\eta$ 
      give us the invertibility of $\epsilon F$. But as $\epsilon F l \equiv \epsilon T$ by 
      $\alpha$ we deduce that $\epsilon T$ is also invertible.
      By \autoref{triangleidentityadjunctionsareequivalenttohomcatdef} 
      and the invertibility of the unit we can also deduce from $\alpha \colon F l \equiv T$ 
      that there is a unique $(U\alpha)(\eta l) \colon l \equiv U F l \equiv U T$, \ie that the lift 
      must be unique in this situation.
    \end{proof}

    The next statement is about an alternative characterisation of the adjointablility of squares, more precisely 
    a square between adjunctions is adjointable if and only if the units commute in the following sense.

    \begin{lemma}
      \label{adjointablilityequivtounitscommute}
      Let
      \begin{equation}
            \begin{tikzpicture}[diagram]
                \matrix[objects] {%
                |(a)| A_1 \& |(b)| A_2 \\
                |(c)| B_1 \& |(d)| B_2 \\
                };
                \path[maps,->]
                (a) edge node[above]  {$H$} (b)
                (c) edge node[below]  {$K$} (d)
                (a) edge node[left]   {$U_1$} (c)
                (b) edge node[right]  {$U_2$} (d)
                ;
            \end{tikzpicture}
      \end{equation}
      be commutative square such that we have adjoints $F_1 \adj U_1$ and $F_2 \adj U_2$.
      Then the following two statements are equivalent.
      \begin{enumerate}
        \item The square is adjointable.
        \item The composed natural transformation
          \begin{equation*}
            \begin{tikzcd}[cramped,column sep=tiny]
              {B_1} &&&& {B_2} \\
              & {B_2 \times_{A_2} \Fun(\Delta^1,A_2)} \\
              & {\Fun(\Delta^1,B_2) \times_{B_2} A_2} && {B_2 \times_{A_2} \Fun(\Delta^1,A_2)} \\
              {\Fun(\Delta^1,B_2) \times_{B_2} A_2} &&&& {\Fun(\Delta^1,B_2) \times_{B_2} A_2}
              \arrow["K", from=1-1, to=1-5]
              \arrow["{(\eta_1, F_1)}"', from=1-1, to=4-1]
              \arrow["{\Delta_{F_2}}", from=1-5, to=3-4]
              \arrow["{(\eta_2, F_2)}", from=1-5, to=4-5]
              \arrow["{\ev_0}", from=2-2, to=1-5]
              \arrow[""{name=0, anchor=center, inner sep=0}, Rightarrow, no head, from=2-2, to=3-4]
              \arrow["{\psi_2}", from=3-2, to=2-2]
              \arrow[Rightarrow, no head, from=3-2, to=4-5]
              \arrow["{\phi_2}", from=3-4, to=4-5]
              \arrow["{K_\ast \times_{K} H}", from=4-1, to=3-2]
              \arrow["{K_\ast \times_{K} H}"', from=4-1, to=4-5]
              \arrow["\epsilon", shorten <=69pt, shorten >=42pt, Rightarrow, from=1-5, to=0]
            \end{tikzcd}
          \end{equation*}
          is invertible.
      \end{enumerate}
    \end{lemma}
    \begin{proof}
      The first statement is by \autoref{triangleidentityadjunctionsareequivalenttohomcatdef}
      equivalent to the existence of the dashed functor in the top part of the following diagram.
      \[\begin{tikzcd}[cramped]
        {B_1} & {B_2} \\
        {\Fun(\Delta^1,B_1) \times^{(\ev_1,U_1)}_{B_1} A_1} & {B_2 \times^{(F_2,\ev_0)}_{A_2} \Fun(\Delta^1,A_2) } \\
        {B_1 \times^{(K,\ev_0)}_{B_2} \Fun(\Delta^1,B_2) \times^{(\ev_1,K U_1)}_{B_2} A_1} \\
        {B_1 \times^{(K,\ev_0)}_{B_2} \Fun(\Delta^1,B_2) \times^{(\ev_1,U_2 H)}_{B_2} A_1} & {\Fun(\Delta^1,B_2) \times^{(\ev_1,U_2)}_{B_2} A_2} \\
        {B_1 \times A_1} & {B_2 \times A_2}
        \arrow["\exists"', color={rgb,255:red,92;green,92;blue,214}, dashed, from=1-1, to=1-2]
        \arrow["{(\eta_1, F_1)}", from=1-1, to=2-1]
        \arrow["{\Delta_{F_2}}", hook', from=1-2, to=2-2]
        \arrow["{(\ev_0, K_\ast) \times_K \id}", from=2-1, to=3-1]
        \arrow[draw={rgb,255:red,92;green,214;blue,92}, curve={height=-30pt}, from=2-2, to=5-2]
        \arrow["{\id \times_{\id} \alpha_\ast \times_{\id} \id}", from=3-1, to=4-1]
        \arrow[draw={rgb,255:red,92;green,214;blue,92}, curve={height=30pt}, from=3-1, to=5-1]
        \arrow["{(\pr_1, H \pr_2)}"', from=4-1, to=4-2]
        \arrow[draw={rgb,255:red,92;green,214;blue,92}, from=4-1, to=5-1]
        \arrow["{\psi_2}"', from=4-2, to=2-2]
        \arrow[draw={rgb,255:red,92;green,214;blue,92}, from=4-2, to=5-2]
        \arrow["{K \times H}"', color={rgb,255:red,92;green,214;blue,92}, from=5-1, to=5-2]
      \end{tikzcd}\]
      By the bottom colored part of the diagram such a lift, if it exists, must be
      equivalent to the functor $K$.
      By \autoref{liftingalongffladjisequivtohavingpostwhiskeringwithcounitinvertible}
      the existence of such dashed lift along the fully-faithful left adjoint $\Delta_{F_2}$
      is equivalent to the post-whiskering with the counit of $\Delta_{F_2} \adj \ev_0$ 
      being invertible.
      As $\phi_2$ and $\psi_2$ are by definition inverse equivalences this invertibility
      is itself equivalent to the second assertion above.
    \end{proof}

    An easy corollary of the equivalence of the unit/counit description and the Hom-equivalence description 
    of adjunctions is that a functor $F$ has a right adjoint if and only 
    the cocartesian fibration $\ev_1 \colon B \times_{A} \Fun(\Delta^1,A) \to A$ admits 
    a fully-faithful right adjoint. 
    This gives a way of going back and forth between general adjunctions and adjunctions with invertible counit.

    \begin{lemma}
      \label{havingaradjointisequivtofreecocartfibhavingaffradj}
      For any functor $F \colon B \to A$ the following two statements are equivalent.
      \begin{enumerate}
        \item $F$ has a right adjoint.
        \item The functor $\ev_1 \colon B \times_{A} \Fun(\Delta^1,A) \to A$
          has a fully-faithful right adjoint.
      \end{enumerate}
    \end{lemma}
    \begin{proof}
      Let us start by assuming the first assertion. 
      By \autoref{triangleidentityadjunctionsareequivalenttohomcatdef}
      we have an equivalence
      \begin{equation}
            \begin{tikzpicture}[diagram]
                \matrix[objects] {%
                  |(a)| B \times^{(F,\ev_0)}_{A} \Fun(\Delta^1,A) \& \& |(b)| \Fun(\Delta^1,B) \times^{(\ev_1,U)}_{B} A \\
                  \& |(c)| B \times A  \\
                  \& |(d)| A  \\
                };
                \path[maps,->]
                (a) edge node[above]  {$\equiv$} (b)
                (a) edge node[below left]  {$\ev$} (c)
                (b) edge node[above left]   {$\ev$} (c)
                (c) edge node[left]   {$\pr_A$} (d)
                (b) edge node[below right]   {$\pr_1$} (d)
                (a) edge node[below left]   {$\ev_1$} (d)
                ;
            \end{tikzpicture}
      \end{equation}
      But we also know that the functor $\pr_1$ above always has the fully-faithful 
      right adjoint $\Delta_U$. Hence $\ev_1$ is the composite of an equivalence
      and a functor which has a fully-faithful right adjoint, thus it also
      has one.
      To show the backwards implication, let us consider the commutative triangle
      \begin{equation}
            \begin{tikzpicture}[diagram]
                \matrix[objects] {%
                  |(a)| B \& \& |(b)| B \times^{(F,\ev_0)}_{A} \Fun(\Delta^1,A) \\
                  \& |(c)| A \\
                };
                \path[maps,->]
                (a) edge node[above]  {$\Delta_F$} (b)
                (a) edge node[below left]  {$F$} (c)
                (b) edge node[below right]   {$\ev_1$} (c)
                ;
            \end{tikzpicture}
      \end{equation}
      Furthermore, it is always true that $\Delta_F$ is left adjoint to 
      $\pr_1 \colon B \times^{(F,\ev_0)}_{A} \Fun(\Delta^1,A) \to B$.
      Now we have factorized $F$ into two left adjoints, hence it also is one.
    \end{proof}

    The next lemma is about an instance of the so-called mate correspondence, the passage back and forth 
    between natural transformations $F \to G$ between left adjoints $F \adj U$, $G \adj V$ and 
    natural transformations in the reversed directions $V \to U$ between their right adjoints.
    We first construct the equivalence.

    \begin{construction}
      \label{matecorrespondencebetweenadjoints}
      Let $U, V \colon A \to B$ be two right adjoints with left adjoints
      $F \adj U$ and $G \adj V$. Let us denote the Hom-space of $\Fun(A,B)$ from $U$ to $V$ by $U \to V$.
      We can then use the Hom-equivalences from \autoref{triangleidentityadjunctionsareequivalenttohomcatdef} 
      of the adjunctions ${}_\ast U \adj {}_\ast F$ and $G_\ast \adj V_\ast$ to get the 
      \define{mate correspondence equivalence} $\matecorrbetwadjs$
      \begin{align*}
        (U \to V) = ({}_\ast U (\id) \to V) &\equiv (\id \to {}_\ast F (V)) \\ 
        &= (\id \to V F) \\ 
        &= (\id \to V_\ast (F)) \equiv (G_\ast (\id) \to F) = (G \to F)
      \end{align*}
      Note that by associativity of whiskering and natural transformation composition, 
      we could have also define it by first invoking $G_\ast \adj V_\ast$ and then ${}_\ast U \adj {}_\ast F$.
    \end{construction}

    This mate correspondence equivalence is functorial in the following sense.
    \begin{lemma}
      \label{matecorrespondanceisfunctorial}
      Let $U, V, W \colon A \to B$ be three right adjoints with left adjoints
      $F \adj U$, $G \adj V$ and $H \adj W$. 
      The mate correspondence equivalences constructed in \autoref{matecorrespondencebetweenadjoints}
      commute with composition of natural transformations, \ie
      \begin{equation*}
            \begin{tikzpicture}[diagram]
                \matrix[objects]{%
                    |(a)| (U \to V) \times (V \to W) \& |(b)| (U \to W) \\
                    |(c)| (G \to F) \times (H \to G) \& |(d)| (H \to F) \\
                };
                \path[maps,->]
                    (a) edge node[above] {$\text{comp}$} (b)
                    (a) edge node[left] {$\matecorrbetwadjs \times \matecorrbetwadjs$} (c)
                    (b) edge node[right] {$\matecorrbetwadjs$} (d)
                    (c) edge node[below] {$\text{comp}$} (d)
                ;
            \end{tikzpicture}
      \end{equation*}
      commutes.
      Furthermore, they are unital, \ie the functors $\matecorrbetwadjs \colon (U \to U) \to (F \to F)$ send identities to 
      identities.
      In particular we have that the following dashed lift always exists.
      \begin{equation*}
            \begin{tikzpicture}[diagram]
                \matrix[objects]{%
                    |(a)| (U \equiv V) \& |(b)| (G \equiv F) \\
                    |(c)| (U \to V) \& |(d)| (G \to F) \\
                };
                \path[maps,->]
                    (a) edge[dashed] node[above] {$\matecorrbetwadjs$} (b)
                    (a) edge (c)
                    (b) edge (d)
                    (c) edge node[below] {$\matecorrbetwadjs$} (d)
                ;
            \end{tikzpicture}
      \end{equation*}
    \end{lemma}
    \begin{proof}
      This is follows from repeated application of \autoref{middlefourinterchangefornattrafos}
      and the triangle identities of $G \adj V$.
      The second claim is an immediate consequence of the triangle identity of $F \adj U$.
      The third claim follows from the characterisation of invertible maps as having a 
      section and a retraction.
    \end{proof}

    The following lemma is a functorial version of the fact that cocartesian lifts are equivalences if and 
    only if their projection to the base category is so.
    \begin{lemma}
      \label{cocartesianliftsofequivsareequivs}
      Let $p \colon X \fibration B$ be a cocartesian fibration.
      Let $f \colon x \to \tilde{x}$ be a morphism in $X$, or more generally a 
      natural transformation between functors $A \to X$ for some category $A$,
      such that $f$ is a $p$-cocartesian lift.
      If its image $p(f)$ under $p$ is an equivalence, then $f$ is 
      also an equivalence.
    \end{lemma}
    \begin{proof}
      We have the commutative triangle
      \begin{equation*}
            \begin{tikzpicture}[diagram]
                \matrix[objects]{%
                  |(a)| \Fun(\Delta^1,X) \& |(b)| X \times_{B} \Fun(\Delta^1,B) \\
                  |(c)| X \& |(d)|  \\
                };
                \path[maps,->]
                    (a) edge node[above] {$(\ev_0, p_\ast)$} (b)
                    (a) edge node[left] {$\ev_0$} (c)
                    (b) edge node[below right] {$\pr_0$} (c)
                ;
            \end{tikzpicture}
      \end{equation*}
      of right adjoints. By \autoref{matecorrespondanceisfunctorial}
      we may pass to left adjoints and obtain the still commutative triangle
      \begin{equation*}
            \begin{tikzpicture}[diagram]
                \matrix[objects]{%
                  |(a)| \Fun(\Delta^1,X) \& |(b)| X \times_{B} \Fun(\Delta^1,B) \\
                  |(c)| X \& |(d)|  \\
                };
                \path[maps,->]
                    (b) edge node[above] {$\text{cocart}$} (a)
                    (c) edge node[left] {$\Delta_X$} (a)
                    (c) edge node[below right] {$\Delta_p$} (b)
                ;
            \end{tikzpicture}
      \end{equation*}
      from which the desired statement follows immediately.
    \end{proof}

    The next statement is a functorial version of the diagrammatic left cancellation property of 
    cocartesian morphisms in cocartesian fibrations.

    \pagebreak 

    \begin{lemma}
      \label{cocartesianliftscanbediagrammaticallyleftcancelled}
      Let $p \colon X \fibration B$ be a cocartesian fibration.
      Let 
      \begin{equation*}
            \begin{tikzpicture}[diagram]
                \matrix[objects]{%
                  |(a)| x_0 \& |(b)| x_1 \\
                  |(c)| \& |(d)| x_2 \\
                };
                \path[maps,->]
                    (a) edge node[below] {$\text{cocart}$} (b)
                    (a) edge node[above] {$f$} (b)
                    (b) edge node[right] {$g$} (d)
                    (a) edge node[below left] {$g \circ f$} (d)
                ;
            \end{tikzpicture}
      \end{equation*}
      be a commutative triangle of morphisms in $X$, or more general natural 
      transformations between functors $A \to X$ for some category $A$,
      such that $f$ is a $p$-cocartesian lift.
      Then $g$ is a cocartesian lift if and only if $g \circ f$ is a 
      cocartesian lift.
    \end{lemma}
    \begin{proof}
      We have the following commutative square.
      \begin{equation*}
            \begin{tikzpicture}[diagram]
              \matrix[objects,column sep={7em,between origins},row sep=2em]{%
                  |(a)| \Fun(\Delta^1,\Fun(\Delta^1,X)) \\ 
                  \& |(b)| \Fun(\Delta^1,X) \times^{(p_\ast,(\ev_0)_\ast)}_{\Fun(\Delta^1,B)} \Fun(\Delta^1,\Fun(\Delta^1,B)) \\
                  |(c)| \Fun(\Delta^1,X) \times^{(p_\ast,\ev_0)}_{\Fun(\Delta^1,B)} \Fun(\Delta^1,\Fun(\Delta^1,B)) \\ 
                  \& |(d)| X \times_{B} \Fun(\Delta^1,\Fun(\Delta^1,B)) \\
                };
                \path[maps,->]
                    (a) edge node[above right] {$((\ev_0)_\ast, p_\ast)$} (b)
                    (a) edge node[left] {$(\ev_0, p_\ast)$} (c)
                    (b) edge node[right] {$\ev_0 \times_{\ev_0} \id$} (d)
                    (c) edge node[below left] {$\ev_0 \times_{\ev_0} \id$} (d)
                ;
            \end{tikzpicture}
      \end{equation*}
      The left hand vertical functor is a right adjoint because
      $p_\ast \colon \Fun(\Delta^1,X) \to \Fun(\Delta^1,B)$
      is a cocartesian fibration.
      The upper diagonal functor is a right adjoint because it is 
      $\Fun(\Delta^1, -)$ applied to the adjunction witnessing that $p$
      is a cocartesian fibration.
      The lower diagonal functor is right adjoint as it is the pullback
      of the adjunction witnessing cocartesianness of $p$, which is an 
      adjunction over $\Fun(\Delta^1,B)$, along
      $\ev_0 \colon \Fun(\Delta^1,\Fun(\Delta^1,B)) \to \Fun(\Delta^1,B)$.
      The right hand vertical functor is right adjoint as it is the pullback
      of the adjunction witnessing cocartesianness of $p$, which is an 
      adjunction over $\Fun(\Delta^1,B)$, along
      $(\ev_0)_\ast \colon \Fun(\Delta^1,\Fun(\Delta^1,B)) \to \Fun(\Delta^1,B)$.
      Now we now that we have an commutative square consisting
      of right adjoints, hence 
      \autoref{matecorrespondanceisfunctorial} tells us 
      that the fully adjointed square, where we pass to left adjoints 
      both vertically and horizontally, commutes again.
      Now to deduce our desired statement we just need to take the functor
      \begin{equation*}
        X \times_B \Fun(\Delta^1,B) \times_B \Fun(\Delta^1,B) \xto{\id_X \times_{\id_B} \text{cocart}_{\ev_1}} X \times_{B} \Fun(\Delta^1,\Fun(\Delta^1,B))
      \end{equation*}
      and precompose it onto our commutative square of left adjoints to obtain 
      the desired claim.
    \end{proof}

    In the following lemma we concern ourselves with squares that are horizontally right adjointable as well as 
    vertically left adjointable. We prove that in this case the order of adjointing along the two 
    different orientations does not matter.
    \begin{lemma}
      \label{bothwaysadjointablegivessametotalmate}
      Let 
      \begin{equation*}
            \begin{tikzpicture}[diagram]
                \matrix[objects]{%
                    |(a)| A \& |(b)| B \\
                    |(c)| C \& |(d)| D \\
                };
                \path[maps,->]
                    (a) edge node[above] {$F_0$} (b)
                    (a) edge node[left] {$V_0$} (c)
                    (b) edge node[right] {$V_1$} (d)
                    (c) edge node[below] {$F_1$} (d)
                ;
                \node[rotate=45] at ($ (c) ! 0.5 ! (b) $) (psi) {$\Rightarrow$};
                \node[maps,right] at (psi) {$\alpha$};
            \end{tikzpicture}
      \end{equation*}
      be a natural transformation such that we have adjoints
      $F_0 \adj U_0$, $F_1 \adj U_1$, $G_0 \adj V_0$ and $G_1 \adj V_1$.
      Let us furthermore assume that the vertically adjointed mate as below on the left
      as well as the horizontally adjointed mate as below on the right
      \begin{equation*}
            \begin{tikzpicture}[diagram]
                \matrix[objects]{%
                    |(a)| A \& |(b)| B \\
                    |(c)| C \& |(d)| D \\
                };
                \path[maps,->]
                    (a) edge node[above] {$F_0$} (b)
                    (c) edge node[left] {$G_0$} (a)
                    (d) edge node[right] {$G_1$} (b)
                    (c) edge node[below] {$F_1$} (d)
                ;
                \node[rotate=135] at ($ (c) ! 0.5 ! (b) $) (psi) {$\Rightarrow$};
                \node[maps,right] at (psi) {$\beta$};
            \end{tikzpicture}
            \qquad
            \begin{tikzpicture}[diagram]
                \matrix[objects]{%
                    |(a)| A \& |(b)| B \\
                    |(c)| C \& |(d)| D \\
                };
                \path[maps,->]
                    (b) edge node[above] {$U_0$} (a)
                    (a) edge node[left] {$V_0$} (c)
                    (b) edge node[right] {$V_1$} (d)
                    (d) edge node[below] {$U_1$} (c)
                ;
                \node[rotate=-45] at ($ (c) ! 0.5 ! (b) $) (psi) {$\Rightarrow$};
                \node[maps,right] at (psi) {$\gamma$};
            \end{tikzpicture}
      \end{equation*}
      are invertible.
      Then the horizontally adjointed mate of $\beta^{-1}$ and the vertically adjointed mate
      of $\gamma^{-1}$ as depicted in
      \begin{equation*}
            \begin{tikzpicture}[diagram]
                \matrix[objects]{%
                    |(a)| A \& |(b)| B \\
                    |(c)| C \& |(d)| D \\
                };
                \path[maps,->]
                    (b) edge node[above] {$U_0$} (a)
                    (c) edge node[left] {$G_0$} (a)
                    (d) edge node[right] {$G_1$} (b)
                    (d) edge node[below] {$U_1$} (c)
                ;
                \node[rotate=45] at ($ (a) ! 0.6 ! (d) $) (psi) {$\Rightarrow$};
                \node[maps,right] at (psi) {$\beta^{-1}$};
                \node[rotate=45] at ($ (a) ! 0.4 ! (d) $) (phi) {$\Rightarrow$};
                \node[maps,left] at (phi) {$\gamma^{-1}$};
            \end{tikzpicture}
      \end{equation*}
      are equivalent.
    \end{lemma}
    \begin{proof}
      In the diagram
      \[\begin{tikzcd}[cramped]
        {(F_1 V_0 \to V_1 F_0 )} & { ( V_0 U_0 \to U_1 V_1 )} & {( V_0 U_0 \equiv U_1 V_1 )} \\
        {( F_1 G_1 \to G_0 F_0 )} && {( U_1 V_1 \to V_0 U_0 )} \\
        {( F_1 G_1 \equiv G_0 F_0 )} & {( G_0 F_0 \to F_1 G_1 )} & {( G_0 U_1 \to U_0 G_1 )}
        \arrow["\equiv"', from=1-1, to=2-1]
        \arrow["\equiv"', from=1-2, to=1-1]
        \arrow["\equiv", from=1-2, to=2-1]
        \arrow[hook, from=1-3, to=1-2]
        \arrow[hook, from=1-3, to=2-3]
        \arrow["\equiv"{description}, from=1-3, to=3-1]
        \arrow["\equiv"', from=2-3, to=3-2]
        \arrow["\equiv", from=2-3, to=3-3]
        \arrow[hook, from=3-1, to=2-1]
        \arrow[hook, from=3-1, to=3-2]
        \arrow["\equiv", from=3-3, to=3-2]
      \end{tikzcd}\]
      all functors that are decorated by $\equiv$ are instances of 
      mate correspondence functors.
      Here the upper left hand triangle and the lower right hand triangle commute 
      because of the fact that the mate correspondence for composed adjoints 
      can be decomposed by the mate correspondences for the factor adjoints,
      because units and counits of composed adjoints can be built from the 
      units and counits of their factor adjoints.
      The two middle squares of the diagram commute by \autoref{matecorrespondanceisfunctorial}.
      Now the proof of the desired statement is a simple diagram chase in 
      this diagram.
    \end{proof}

    The next statement is a functorial version of cocartesian pushforward preserving cocartesianness of 
    morphisms.
    \begin{lemma}
      \label{cocartiscocartfun}
      Let $p \colon X \fibration B$ be a cocartesian fibration.
      Then 
      \begin{equation*}
            \begin{tikzpicture}[diagram]
                \matrix[objects]{%
                  |(a)| X \times_B \Fun(\Delta^1,B) \& \& |(b)| \Fun(\Delta^1,X) \\
                  |(c)| \& |(d)| \Fun(\Delta^1,B) \\
                };
                \path[maps,->]
                    (a) edge node[above] {$\text{cocart}$} (b)
                    (b) edge[->>] node[below right] {$p_\ast$} (d)
                    (a) edge[->>] node[below left] {$\pr_1$} (d)
                ;
            \end{tikzpicture}
      \end{equation*}
      is a cocartesian functor between cocartesian fibrations.
    \end{lemma}
    \begin{proof}
      The triangle commutes as $\text{cocart}$ is left adjoint to $(\ev_0, p_\ast)$ with 
      invertible unit.
      $\pr_1$ is a cocartesian fibration as it is the pullback of $p$.
      To show that $\text{cocart}$ is cocartesian we want to show that 
      \begin{equation*}
            \begin{tikzpicture}[diagram]
                \matrix[objects]{%
                  |(b)| \Fun(\Delta^1,X \times_B \Fun(\Delta^1,B)) \& |(d)| X \times_{B} \Fun(\Delta^1,B) \times_{\Fun(\Delta^1,B)} \Fun(\Delta^1,\Fun(\Delta^1,B)) \\
                  |(a)| \Fun(\Delta^1,\Fun(\Delta^1,X)) \& |(c)| \Fun(\Delta^1,X) \times^{(p_\ast,\ev_0)}_{\Fun(\Delta^1,B)} \Fun(\Delta^1,\Fun(\Delta^1,B)) \\
                };
                \path[maps,->]
                    (b) edge node[right] {$\text{cocart}_\ast$} (a)
                    (a) edge node[below] {$(\ev_0, p_\ast)$} (c)
                    (b) edge node[above] {$(\ev_0, (\pr_1)_\ast)$} (d)
                    (d) edge node[right] {$\text{cocart} \times_{\id} \id$} (c)
                ;
            \end{tikzpicture}
      \end{equation*}
      is vertically adjointable.
      But this is exactly the horizontally adjointed square to 
      \begin{equation*}
            \begin{tikzpicture}[diagram]
                \matrix[objects]{%
                  |(a)| \Fun(\Delta^1,\Fun(\Delta^1,X)) \& |(c)| \Fun(\Delta^1,X) \times^{(p_\ast,\ev_0)}_{\Fun(\Delta^1,B)} \Fun(\Delta^1,\Fun(\Delta^1,B)) \\
                  |(b)| \Fun(\Delta^1,X \times_B \Fun(\Delta^1,B)) \& |(d)| X \times_{B} \Fun(\Delta^1,B) \times_{\Fun(\Delta^1,B)} \Fun(\Delta^1,\Fun(\Delta^1,B)) \\
                };
                \path[maps,->]
                    (a) edge node[right] {$(\ev_0,p_\ast)_\ast$} (b)
                    (a) edge node[above] {$(\ev_0, p_\ast)$} (c)
                    (b) edge node[below] {$(\ev_0, (\pr_1)_\ast)$} (d)
                    (c) edge node[right] {$(\ev_0, p_\ast) \times_{\id} \id$} (d)
                ;
            \end{tikzpicture}
      \end{equation*}
      As explained in the proof of \autoref{bothwaysadjointablegivessametotalmate},
      to first horizontally adjoint and then vertically adjoint the result is equivalent to 
      fully passing to left adjoints with all the functors.
      But as the above square is, as explained in the proof of 
      \autoref{cocartesianliftscanbediagrammaticallyleftcancelled},
      a commutative square of right adjoints we can deduce as we did there that the 
      fully adjointed square also commutes, which proves our claim.
    \end{proof}

  \subsection{The directed Left-Lifting Property of Left Adjoints against Cocartesian Fibrations}
    \label{subsectiondirectedllpofladjagainstcocartfibs}

    In this subsection we will discuss a generalisation of the unique left lifting property of cofinal functors
    against left fibrations.
    Informally, we will concern ourselves with laxly commuting squares of the form
    \[\begin{tikzcd}[cramped]
      A & X \\
      B & Y
      \arrow[from=1-1, to=1-2]
      \arrow["F"', from=1-1, to=2-1]
      \arrow["\alpha"', shorten <=6pt, shorten >=6pt, Rightarrow, from=1-2, to=2-1]
      \arrow["p", two heads, from=1-2, to=2-2]
      \arrow[from=2-1, to=2-2]
    \end{tikzcd}\]
    where $F$ is a left adjoint and $p$ a cocartesian fibration.
    We will show that in this situation one can always construct a lax lift $l$ of this square, \ie 
    \[\begin{tikzcd}[cramped]
      A & X \\
      B & Y
      \arrow[""{name=0, anchor=center, inner sep=0}, from=1-1, to=1-2]
      \arrow["F"', from=1-1, to=2-1]
      \arrow["p", two heads, from=1-2, to=2-2]
      \arrow["l"', dashed, from=2-1, to=1-2]
      \arrow[from=2-1, to=2-2]
      \arrow["{\tilde{\alpha}}"', shorten <=8pt, shorten >=8pt, Rightarrow, from=0, to=2-1]
    \end{tikzcd}\]
    together with an identification $p \tilde{\alpha} \equiv \alpha$. 
    This is achieved by employing the adjunction $F \adj U$ and cocartesian lifting for $p$ and letting them
    work together.
    Moreover this lax lift will be special in the sense that it will be initial among all lifts that 
    recover the original lax square via postwhiskering with $p$.
    More precisely the strategy is to precompose the original lax square with the counit $\epsilon \colon FU \twoto \id$ 
    to obtain 
    \[\begin{tikzcd}[cramped]
      B & A & X \\
      & B & Y
      \arrow["U", from=1-1, to=1-2]
      \arrow[""{name=0, anchor=center, inner sep=0}, equals, from=1-1, to=2-2]
      \arrow[from=1-2, to=1-3]
      \arrow["F", from=1-2, to=2-2]
      \arrow["\alpha", shorten <=8pt, shorten >=8pt, Rightarrow, from=1-3, to=2-2]
      \arrow["p", two heads, from=1-3, to=2-3]
      \arrow[from=2-2, to=2-3]
      \arrow["\epsilon", shorten <=4pt, shorten >=4pt, Rightarrow, from=1-2, to=0]
    \end{tikzcd}\]
    and then then cocartesian lift the top horizontal composite functor along this natural transformation 
    to get 
    \[\begin{tikzcd}[sep=large]
      B & A & X \\
      & B & Y
      \arrow["U", from=1-1, to=1-2]
      \arrow[""{name=0, anchor=center, inner sep=0}, "{\exists_! l}"', curve={height=24pt}, dashed, from=1-1, to=1-3]
      \arrow[equals, from=1-1, to=2-2]
      \arrow[from=1-2, to=1-3]
      \arrow["p", two heads, from=1-3, to=2-3]
      \arrow[from=2-2, to=2-3]
      \arrow["{\exists_! \theta}", shorten <=3pt, shorten >=3pt, Rightarrow, from=1-2, to=0]
    \end{tikzcd}\]
    This we can then in turn precompose with the unit $\eta \colon \id \twoto UF$ 
    \[\begin{tikzcd}[sep=large]
      A \\
      B & A & X \\
      & B & Y
      \arrow["F"', from=1-1, to=2-1]
      \arrow[""{name=0, anchor=center, inner sep=0}, equals, from=1-1, to=2-2]
      \arrow["U", from=2-1, to=2-2]
      \arrow[""{name=1, anchor=center, inner sep=0}, "{\exists_! l}"', curve={height=24pt}, dashed, from=2-1, to=2-3]
      \arrow[equals, from=2-1, to=3-2]
      \arrow[from=2-2, to=2-3]
      \arrow["p", two heads, from=2-3, to=3-3]
      \arrow[from=3-2, to=3-3]
      \arrow["\eta"', shorten <=3pt, shorten >=2pt, Rightarrow, from=0, to=2-1]
      \arrow["{\exists_! \theta}", shorten <=3pt, shorten >=3pt, Rightarrow, from=2-2, to=1]
    \end{tikzcd}\]
    to get our lax lift $(l, \tilde{\alpha})$ as indicated in above.

    This directed lifting property itself we be a corollary to the following general lemma 
    applied to the adjointable square obtained by precomposing with $F$ and postcomposing with $p$.

    \begin{lemma}
      \label{weirdadjointablecocartfiblemma}
      Let
      \begin{equation}
            \begin{tikzpicture}[diagram]
                \matrix[objects] {%
                |(a)| A \& |(b)| C \\
                |(c)| B \& |(d)| D \\
                };
                \path[maps,->]
                (a) edge node[above]  {$U$} (b)
                (c) edge node[below]  {$V$} (d)
                (a) edge[->>] node[left]   {$p$} (c)
                (b) edge node[right]  {$q$} (d)
                ;
                \node[rotate=45] at ($ (c) ! 0.5 ! (b) $) (psi) {$\equiv$};
                \node[maps,right] at (psi) {$\alpha$};
            \end{tikzpicture}
      \end{equation}
      be a commutative square of functors with adjunctions
      $F \adj U$, $G \adj V$ such that the square is adjointable.
      Let us furthermore assume that $p$ is a cocartesian fibration.
      Then the functor 
      \begin{equation*}
        (\ev_0, \alpha \circ q_\ast, p) \colon \Fun(\Delta^1,C) \times_{C} A \to C \times_{D} \Fun(\Delta^1,D) \times_{D} B
      \end{equation*}
      has a fully-faithful left adjoint.
    \end{lemma}
    \begin{proof}
      By definition of cocartesianness we know that 
      \begin{equation*}
        (\ev_0, p_\ast) \colon \Fun(\Delta^1,A) \to A \times_B \Fun(\Delta^1,B)
      \end{equation*}
      has a fully-faithful left adjoint.
      Now we pull back along $F$ and obtain
      \begin{equation}
            \begin{tikzpicture}[diagram]
                \matrix[objects] {%
                |(a)| C \times^{(F,\ev_0)}_A \Fun(\Delta^1,A) \& |(b)| \Fun(\Delta^1,A) \\
                |(c)| C \times^{(pF,\ev_0)}_B \Fun(\Delta^1,B) \& |(d)| A \times^{(p,\ev_0)}_B \Fun(\Delta^1,B) \\
                |(e)| C \& |(f)| A \\
                };
                \path[maps,->]
                (a) edge node[above]  {$\pr_1$} (b)
                (c) edge node[below]  {$F \times_{\id} \id$} (d)
                (e) edge node[below]  {$F$} (f)
                (a) edge node[left]   {$(\id,p_\ast,p)$} (c)
                (c) edge node[left]   {$\pr_0$} (e)
                (b) edge node[right]  {$(\ev_0,p_\ast)$} (d)
                (d) edge node[right]  {$\pr_0$} (f)
                ;
            \end{tikzpicture}
      \end{equation}
      The composite rectangle is a pullback as well as the bottom square, thus
      by pullback cancellation also the top square.
      By \autoref{pullbackofadjunctions} 
      we conclude that the functor $(\id, p_\ast, p)$ also has a fully-faithful left adjoint.
      Next we observe that the diagram
      \[\begin{tikzcd}[cramped]
        {C \times^{(F,\ev_0)}_A \Fun(\Delta^1,A)} & {\Fun(\Delta^1,C) \times^{(\ev_1,U)}_C A} & {C \times^{(q,\ev_0)}_D \Fun(\Delta^1,D) \times^{(\ev_1,V)}_D B} \\
        & {C \times^{(pF,\ev_0)}_B \Fun(\Delta^1,B)} & {C \times^{(Gq,\ev_0)}_B \Fun(\Delta^1,B)}
        \arrow["{\phi^{F \dashv U}}", from=1-1, to=1-2]
        \arrow["{(\id,p_\ast,p)}"', from=1-1, to=2-2]
        \arrow["{(\ev_0, \alpha \circ q_\ast, p)}"', from=1-2, to=1-3]
        \arrow["{C \times_D \psi^{G \dashv V}}", from=1-3, to=2-3]
        \arrow["{\_ \circ \text{mate}_\alpha}"', from=2-2, to=2-3]
      \end{tikzcd}\]
      where the functors denoted $\phi$ and $\psi$ are the Hom-equivalences given by
      \autoref{triangleidentityadjunctionsareequivalenttohomcatdef} 
      from the adjunctions $F \adj U$ and $G \adj V$,
      is commutative by naturality of $G\alpha$ and the counit of $G \adj V$.
      As all the other functors are equivalences we deduce that 
      $(\ev_0, \alpha \circ q_\ast, p)$ also has a fully-faithful left adjoint.
    \end{proof}

    \laxliftingpropertyofladjwrtcocartfib*

    We now discuss special cases of the general directed lifting of left adjoints against cocartesian fibrations,
    namely what happens if the starting lax square is actually already commutative, \ie the natural 
    transformation is invertible, and recover an actual non-lax lifting in the case that $F$ is 
    fully-faithful.

    \liftingpropertyofffladjwrtcocartfib*
    \begin{proof}
      In the situation of \autoref{laxliftingpropertyofleftadjointsagainstcocartesianfibrations} 
      we have that the square
      \begin{equation}
            \begin{tikzpicture}[diagram,row sep=small]
                \matrix[objects,wide origins] {%
                |(a)| \Fun(B,X) \\ 
                \& |(c)| \Fun(\Delta^1,\Fun(A,X)) \times_{\Fun(A,X)} \Fun(B,X) \\
                |(b)| \Fun(A,X) \times_{\Fun(A,Y)} \Fun(B,Y) \\ 
                \& |(d)| \Fun(A,X) \times_{\Fun(A,Y)} \Fun(\Delta^1,\Fun(A,Y)) \times_{\Fun(A,Y)} \Fun(B,Y) \\
                };
                \path[maps,->]
                (a) edge node[left]  {$({}_\ast F , p_\ast)$} (b)
                (c) edge node[right]  {$$} (d)
                (a) edge node[above right]   {$\Delta_{\Fun(F,X)}$} (c)
                (b) edge node[below left]  {$(\id, \Delta_{\Fun(A,Y)}, \id)$} (d)
                ;
            \end{tikzpicture}
      \end{equation}
      commutes.
      By the proof of \autoref{weirdadjointablecocartfiblemma} and the invertibility
      of the unit of the adjunction we then deduce that the left adjoint of 
      \autoref{laxliftingpropertyofleftadjointsagainstcocartesianfibrations}
      also factors through 
      \begin{equation*}
        \Delta_{({}_\ast F)} \colon \Fun(B,X) \to \Fun(\Delta^1,\Fun(A,X)) \times_{\Fun(A,X)} \Fun(B,X).
      \end{equation*}

    \end{proof}

    Informally, this adjunction now tells us that we can lift any commutative square
    \begin{equation}
          \begin{tikzpicture}[diagram]
              \matrix[objects] {%
              |(a)| A \& |(b)| X \\
              |(c)| B \& |(d)| Y \\
              };
              \path[maps,->]
              (a) edge node[above]  {$H$} (b)
              (c) edge node[below]  {$K$} (d)
              (a) edge node[left]   {$F$} (c)
              (b) edge[->>] node[right]  {$p$} (d)
              (c) edge[dashed] node[right]  {$$} (b)
              ;
          \end{tikzpicture}
    \end{equation}
    via pre-whiskering with the counit of $F \adj U$ and then cocartesian lifting along $p$.
    But the adjunction tells us even more, the lift constructed in the explained manner is 
    even initial in $\Fun(B,X)$ among all other lifts of this fixed square.

    To close this subsection we restrict further to the left adjoints 
    $\Delta_G \colon C \mono C \times_Y \Fun(\Delta^1,Y)$. In this situation we can give an even more useful 
    characterisation of the lifts obtained through this lax lifting procedure.

    \begin{corollary}
      \label{liftsofDeltaFagainstcocartesianfibrationsincommutativesquares}
      If we specialize this lifting situation of 
      \autoref{liftsoffullandfaithfulleftadjointsagainstcocartesianfibrationsincommutativesquares}
      even further by taking as the left adjoint $F$ the functor $\Delta_G \colon C \mono C \times_Y \Fun(\Delta^1,Y)$ 
      for some arbitrary functor
      $G \colon C \to Y$, and also take $K$ to be the functor $\ev_1 \colon C \times_Y \Fun(\Delta^1,Y) \to Y$,
      then the prescribed lifting strategy for commutative squares of the form
      \begin{equation}
            \begin{tikzpicture}[diagram]
                \matrix[objects] {%
                |(a)| C \& |(b)| X \\
                |(c)| C \times_Y \Fun(\Delta^1,Y) \& |(d)| Y \\
                };
                \path[maps,->]
                (a) edge node[above]  {$H$} (b)
                (c) edge[->>] node[below]  {$\ev_1$} (d)
                (a) edge node[left]   {$\Delta_G$} (c)
                (b) edge[->>] node[right]  {$p$} (d)
                (c) edge[dashed] (b)
                ;
            \end{tikzpicture}
      \end{equation}
      produces even stronger results then the characterisation of the 
      essential image of the lifting left adjoint. Indeed, a lift
      of such a square is in the essential image of the lifting left
      adjoint if and only if it is a cocartesian functor between the 
      indicated cocartesian fibrations.
    \end{corollary}
    \begin{proof}
      To prove this let us start with an arbitrary $\ev_1$-cocartesian pushforward
      of an arbitrary object $(c, f \colon F(c) \to y)$ in 
      $C \times_Y \Fun(\Delta^1,Y)$ along an arbitrary morphism $g \colon y \to y'$ 
      in $Y$.
      Such a cocartesian lift is given by

      \[\begin{tikzcd}
        {F(c)} & {F(c)} \\
        y & {y'}
        \arrow["{F(\id_c)}", equals, from=1-1, to=1-2]
        \arrow["f"', from=1-1, to=2-1]
        \arrow["gf", from=1-2, to=2-2]
        \arrow["g"', from=2-1, to=2-2]
      \end{tikzcd}\]
      We can precompose this with the component of the counit $\epsilon$ of the adjunction 
      $\Delta_G \adj \ev_0$ at the object $(c,f)$ to get

      \[\begin{tikzcd}
        {F(c)} & {F(c)} & {F(c)} \\
        {F(c)} & y & {y'}
        \arrow["{F(\id_c)}", equals, from=1-1, to=1-2]
        \arrow["{F(\id_c)}"', equals, from=1-1, to=2-1]
        \arrow["{F(\id_c)}", equals, from=1-2, to=1-3]
        \arrow["f"', from=1-2, to=2-2]
        \arrow["gf", from=1-3, to=2-3]
        \arrow["f"', from=2-1, to=2-2]
        \arrow["g"', from=2-2, to=2-3]
      \end{tikzcd}\]
      This composite is itself the component of $\epsilon$ at the object
      $(b, g f)$.
      Furthermore, $\epsilon$ is itself an $\ev_1$-cocartesian lift.
      Hence if such a functor $C \times_Y \Fun(\Delta^1,Y) \to X$
      lifting the square sends the components of $\epsilon$ to cocartesian
      lifts then by the cancellation property of cocartesian lifts
      \autoref{cocartesianliftscanbediagrammaticallyleftcancelled}
      also arbitrary $\ev_1$-cocartesian lifts.
    \end{proof}

    \begin{remark}
      The only things that we really need to know in the previous corollary about 
      the functors 
      \begin{equation*}
        C \xto{\Delta_G} C \times_Y \Fun(\Delta^1,Y) \xfibration{\ev_1} Y
      \end{equation*} 
      is the following.
      We need to know that $\Delta_G$ has a retraction $r$, \ie we have $\eta \colon \id \equiv r \circ \Delta_G$,
      that we have a natural transformation $\epsilon \colon \Delta_G \circ r \twoto \id$ such that 
      $\epsilon \Delta_G \equiv \Delta_G \eta^{-1}$, that $\ev_1$ is a cocartesian fibration, that 
      $\epsilon$ is a cocartesian lift of $\ev_1 \epsilon$ and that $\Delta_G \circ r$ 
      sends cocartesian lifts to equivalences. From the last property one can then conclude that 
      $r \equiv r \circ \Delta_G \circ r$ also sends cocartesian lifts to equivalences.
      As $\epsilon$ is a particular such a cocartesian lift we get that $r \epsilon$ is invertible.
      From this one can deduce that $\Delta_G$ is left adjoint to $r$, with invertible unit.
    \end{remark}

    \begin{corollary}
      In particular \autoref{liftsofDeltaFagainstcocartesianfibrationsincommutativesquares}
      proves that we can restrict the adjunction of 
      \autoref{liftsoffullandfaithfulleftadjointsagainstcocartesianfibrationsincommutativesquares}
      to its essential image and obtain an equivalence
      \begin{equation*}
        {}_\ast (\Delta_G) \colon \Fun^\text{cocart}_{/ Y} (\ev_1, p) \xto{\equiv} \Fun_{/ Y}(G,p)
      \end{equation*}
      where we denote by $\Fun^\text{cocart}_{/ Y} (\ev_1, p)$ the full subcategory of 
      $\Fun_{/ Y} (\ev_1, p)$ on the cocartesian functors over $Y$,
      \ie we proved that $\ev_1 \colon C \times_Y \Fun(\Delta^1,Y) \fibration Y$ is the 
      \define{free cocartesian fibration on $G$}.
    \end{corollary}

    \subsection{More advanced Stability Properties of Adjunctions and Cocartesian Fibrations}
      \label{subsectionmoreadvancedstabilityproperties}

    In this subsection we address more intricate stability properties of adjunctions, like
    \autoref{cubepullbackofladjwithadjfacesisladj} the stability of pullbacks of left adjoints seen as objects in $\Fun(\Delta^1,\Cat)$,
    under the assumption that the cospan legs are adjointable squares.
    From this one can immediately deduce that pullbacks of cocartesian fibrations seen as 
    objects in $\Fun(\Delta^1,\Cat)$ are cocartesian fibrations again if the assume that the 
    cospan legs are cocartesian functors.
    After that we turn our attention to \autoref{pullbackofrightadjointalongcocartesianfibrationisrightadjoint} the stability of right adjoints under pullback along 
    cocartesian fibrations.

    We start of with an extension of the pullback stability of right adjoints which have fully-faithful 
    left adjoints.

    \begin{lemma}
        \label{enhancedpullbackofadjunctions}
        Let
        \begin{equation*}
          \begin{tikzpicture}[diagram]
              \matrix[objects]{%
                  |(a)| Y \& |(b)| X \\
                  |(c)| A \& |(d)| B \\
              };
              \path[maps,->]
                  (a) edge node[above] {$V$} (b)
                  (a) edge[->>] node[left] {$q$} (c)
                  (b) edge[->>] node[right] {$p$} (d)
                  (c) edge node[below] {$U$} (d)
              ;
              \node at (barycentric cs:a=0.8,b=0.3,c=0.3) (phi) {\mbox{\LARGE{$\lrcorner$}}};
          \end{tikzpicture}
        \end{equation*}
        be a pullback square such that $U$ has a fully-faithful left adjoint 
        $F$ and $p$ is a cocartesian fibration.
        Then its pullback $q$ is also a cocartesian fibration and the 
        functor $V$ is a cocartesian functor over $U$.
        Furthermore, the pullback $V$ of $U$ also has a fully-faithful left adjoint 
        $G$ and the pullback square is furthermore horizontally adjointable and 
        in this adjointed square $G$ is a cocartesian functor over $F$.
    \end{lemma}
    \begin{proof}
      The first part of the statement is just \autoref{pullbackofcocartfibs}.
      From the proof of \autoref{pullbackofadjunctions} we know
      that the left adjoint is produced by the pullback factorization
      \begin{equation*}
          \begin{tikzpicture}[diagram]
              \matrix[objects]{%
                  |(x)| X \& |(a)| Y \& |(b)| X \\
                  |(y)| B \& |(c)| A \& |(d)| B \\
              };
              \path[maps,->]
                  (a) edge node[above] {$V$} (b)
                  (a) edge node[left] {$q$} (c)
                  (b) edge node[right] {$p$} (d)
                  (c) edge node[below] {$U$} (d)

                  (x) edge node[right] {$p$} (y)
                  (x) edge node[above] {$G$} (a)
                  (y) edge node[below] {$F$} (c)
              ;
              \node at (barycentric cs:a=0.8,b=0.3,c=0.3) (phi) {\mbox{\LARGE{$\lrcorner$}}};
              \node at (barycentric cs:x=0.8,y=0.3,a=0.3) (pi) {\mbox{\LARGE{$\lrcorner$}}};

              \path[maps,-] 
              (x) edge[double distance=0.2em,bend left] (b)
              (y) edge[double distance=0.2em,bend right] (d)
              ;
          \end{tikzpicture}
      \end{equation*}
      which in turn induces the pullback factorization
      \begin{equation*}
          \begin{tikzpicture}[diagram]
              \matrix[objects]{%
                  |(x)| \Fun(\Delta^1,X) \& |(a)| \Fun(\Delta^1,Y) \& |(b)| \Fun(\Delta^1,X) \\
                  |(y)| X \times_B \Fun(\Delta^1,B) \& |(c)| Y \times_A \Fun(\Delta^1,A) \& |(d)| X \times_B \Fun(\Delta^1,B) \\
              };
              \path[maps,->]
                  (a) edge node[above] {$V_\ast$} (b)
                  (a) edge node[left] {$(\ev_0,q_\ast)$} (c)
                  (b) edge node[right] {$(\ev_0,p_\ast)$} (d)
                  (c) edge node[below] {$V \times_U U_\ast$} (d)

                  (x) edge node[right] {$(\ev_0,p_\ast)$} (y)
                  (x) edge node[above] {$G_\ast$} (a)
                  (y) edge node[below] {$G \times_F F_\ast$} (c)
              ;
              \node at (barycentric cs:a=0.8,b=0.3,c=0.3) (phi) {\mbox{\LARGE{$\lrcorner$}}};
              \node at (barycentric cs:x=0.8,y=0.3,a=0.3) (pi) {\mbox{\LARGE{$\lrcorner$}}};

              \path[maps,-] 
              (x) edge[double distance=0.2em,out=20,in=160] (b)
              (y) edge[double distance=0.2em,out=-20,in=-160] (d)
              ;
          \end{tikzpicture}
      \end{equation*}
      hence both squares are also vertically adjointable by 
      \autoref{pullbackofadjunctions}.
      In particular, we see that $G$ is cocartesian over $F$.
    \end{proof}

    In order to prove the full version of the pullback stability of left adjoints in $\Fun(\Delta^1,\Cat)$ our 
    strategy will be to translate having a right adjoint into the Hom-equivalence characterisation of 
    adjunctions, see \autoref{triangleidentityadjunctionsareequivalenttohomcatdef}, 
    and then use that commutative squares induce cocartesian functors between the so-called
    free cocartesian fibrations of the left and right hand vertical functors of the square, from 
    \autoref{freecocartesianfibrationsonfunctorsandsquaresoffunctorsinducecocartesianfunctorsbetweenthose}.

    We first produce the appropriate commutative squares of Hom-equivalences from adjointable squares.

    \begin{lemma}
      \label{adjointablesquaresgivecommutativesquaresofhomcatequivalences}
      Let
      \begin{equation}
            \begin{tikzpicture}[diagram]
                \matrix[objects] {%
                |(a)| A_0 \& |(b)| A_1 \\
                |(c)| B_0 \& |(d)| B_1 \\
                };
                \path[maps,->]
                (a) edge node[above]  {$H$} (b)
                (c) edge node[below]  {$K$} (d)
                (a) edge node[left]   {$U_0$} (c)
                (b) edge node[right]  {$U_1$} (d)
                ;
            \end{tikzpicture}
      \end{equation}
      be a commutative square of functors such that the $U_i$ both have
      left adjoints $F_i$ and the square is adjointable.
      The by \autoref{adjointablilityequivtounitscommute} commutative square
      \begin{equation}
            \begin{tikzpicture}[diagram]
                \matrix[objects] {%
                |(a)| B_0 \& |(b)| B_1 \\
                |(c)| \Fun(\Delta^1,B_0) \times_{B_0} A_0 \& |(d)| \Fun(\Delta^1,B_1) \times_{B_1} A_1 \\
                };
                \path[maps,->]
                (a) edge node[above]  {$K$} (b)
                (c) edge node[below]  {$K_\ast \times_{K} H$} (d)
                (a) edge node[left]   {$(\eta_0,F_0)$} (c)
                (b) edge node[right]  {$(\eta_1,F_1)$} (d)
                ;
            \end{tikzpicture}
      \end{equation}
      can be factorized vertically into commutative squares

      \begin{equation}
        \label{adjunctionfactorizationofunitscommutesquare}
        \begin{tikzcd}[cramped]
          {B_0} & {B_1} \\
          {B_0 \times_{A_0} \Fun(\Delta^1,A_0)} & {B_1 \times_{A_1} \Fun(\Delta^1,A_1)} \\
          {  \Fun(\Delta^1,B_0) \times_{B_0} A_0} & {  \Fun(\Delta^1,B_1) \times_{B_1} A_1}
          \arrow["K", from=1-1, to=1-2]
          \arrow["{\Delta_{F_0}}", from=1-1, to=2-1]
          \arrow["{(\eta_0, F_0)}"', shift right=7, curve={height=30pt}, from=1-1, to=3-1]
          \arrow["{\Delta_{F_0}}", from=1-2, to=2-2]
          \arrow["{(\eta_1, F_1)}", shift left=7, curve={height=-30pt}, from=1-2, to=3-2]
          \arrow["{K \times_H H_\ast}", from=2-1, to=2-2]
          \arrow["{\phi_0}", from=2-1, to=3-1]
          \arrow["{\phi_1}", from=2-2, to=3-2]
          \arrow["{H \times_K K_\ast}"', from=3-1, to=3-2]
        \end{tikzcd}
      \end{equation}
    \end{lemma}
    \begin{proof}
      As 
      \begin{equation}
            \begin{tikzpicture}[diagram]
                \matrix[objects] {%
                  |(a)| B_0 \times_{A_0} \Fun(\Delta^1,A_0) \& |(b)| B_1 \times_{A_1} \Fun(\Delta^1,A_1) \\
                |(c)| A_0 \& |(d)| A_1 \\
                };
                \path[maps,->]
                (a) edge node[above]  {$H \times_{K} K_\ast$} (b)
                (c) edge node[below]  {$K$} (d)
                (a) edge[->>] node[left]   {$\ev_1$} (c)
                (b) edge[->>] node[right]  {$\ev_1$} (d)
                ;
            \end{tikzpicture}
      \end{equation}
      and by pullback also
      \begin{equation}
            \begin{tikzpicture}[diagram]
                \matrix[objects] {%
                  |(a)| \Fun(\Delta^1,B_0) \times_{B_0} A_0 \& |(b)|  \Fun(\Delta^1,B_1) \times_{B_1} A_1\\
                |(c)| A_0 \& |(d)| A_1 \\
                };
                \path[maps,->]
                (a) edge node[above]  {$ K_\ast \times_{K} H$} (b)
                (c) edge node[below]  {$H$} (d)
                (a) edge[->>] node[left]   {$\ev_1$} (c)
                (b) edge[->>] node[right]  {$\ev_1$} (d)
                ;
            \end{tikzpicture}
      \end{equation}
      are cocartesian functors by 
      \autoref{freecocartesianfibrationsonfunctorsandsquaresoffunctorsinducecocartesianfunctorsbetweenthose}
      both composites of the bottom square in \autoref{adjunctionfactorizationofunitscommutesquare}
      are cocartesian functors over the functor $H \colon A_0 \to A_1$.
      As the outer rectangle in \autoref{adjunctionfactorizationofunitscommutesquare} 
      commutes both composites are lifts of the lower rectangle in
      \[\begin{tikzcd}[cramped,column sep=small]
        & {B_0 \times_{A_0} \Fun(\Delta^1,A_0)} & {\Fun(\Delta^1,B_0) \times_{B_0} A_0} \\
        {B_0} & {B_1} & {B_1 \times_{A_1} \Fun(\Delta^1,A_1)} & {\Fun(\Delta^1,B_1) \times_{B_1} A_1} \\
        {B_0 \times_{A_0} \Fun(\Delta^1,A_0)} && {A_0} & {A_1}
        \arrow[""{name=0, anchor=center, inner sep=0}, "{\phi_0}", from=1-2, to=1-3]
        \arrow["{K_\ast \times_K H}"{pos=0.8}, from=1-3, to=2-4]
        \arrow["{\Delta_{F_0}}", hook, from=2-1, to=1-2]
        \arrow["K", from=2-1, to=2-2]
        \arrow["{\Delta_{F_0}}"', hook, from=2-1, to=3-1]
        \arrow[""{name=1, anchor=center, inner sep=0}, "{\Delta_{F_1}}", hook, from=2-2, to=2-3]
        \arrow["{\phi_1}", from=2-3, to=2-4]
        \arrow["{\pr_1}", two heads, from=2-4, to=3-4]
        \arrow[draw={rgb,255:red,153;green,92;blue,214}, dashed, from=3-1, to=2-4]
        \arrow["{\ev_1}"', two heads, from=3-1, to=3-3]
        \arrow["H"', from=3-3, to=3-4]
        \arrow["\equiv"{pos=0.4}, draw=none, from=0, to=1]
      \end{tikzcd}\]
      that are even cocartesian functors over $H$.
      But cocartesian functors in particular preserve the $\ev_1$-cocartesian lift
      \[\begin{tikzcd}[cramped]
        {B \times_A \Fun(\Delta^1,A)} & B & {B \times_A \Fun(\Delta^1,A)} \\
        && A
        \arrow["{\ev_0}", from=1-1, to=1-2]
        \arrow[""{name=0, anchor=center, inner sep=0}, curve={height=24pt}, Rightarrow, no head, from=1-1, to=1-3]
        \arrow["{\Delta_F}", hook, from=1-2, to=1-3]
        \arrow["{\ev_1}", two heads, from=1-3, to=2-3]
        \arrow["{\text{counit}}"{pos=0.3}, shorten <=1pt, shorten >=3pt, Rightarrow, from=1-2, to=0]
      \end{tikzcd}\]
      and by the adjunction of \autoref{laxliftingpropertyofleftadjointsagainstcocartesianfibrations}
      there exists a unique such lift, hence both lifts must be equivalent.

      Alternatively, one could also argue in the following more structured way.
      By pulling back along $H$ we can reduce without loss 
      of generality to talking about our two lifts of the bottem left square in
      \[\begin{tikzcd}[cramped]
        & {B_1} & {B_1 \times_{A_1} \Fun(\Delta^1,A_1)} \\
        {B_0} && {\Fun(\Delta^1,B_1) \times_{B_1} A_0} & {\Fun(\Delta^1,B_1) \times_{B_1} A_1} \\
        {B_0 \times_{A_0} \Fun(\Delta^1,A_0)} && {A_0} & {A_1}
        \arrow["{\Delta_{F_1}}", hook, from=1-2, to=1-3]
        \arrow["{\phi_1}", from=1-3, to=2-4]
        \arrow["K", from=2-1, to=1-2]
        \arrow["{\exists !}", from=2-1, to=2-3]
        \arrow["{\Delta_{F_0}}"', hook, from=2-1, to=3-1]
        \arrow[from=2-3, to=2-4]
        \arrow["{\pr_0}", two heads, from=2-3, to=3-3]
        \arrow["{\pr_1}", two heads, from=2-4, to=3-4]
        \arrow[draw={rgb,255:red,153;green,92;blue,214}, dashed, from=3-1, to=2-3]
        \arrow["{\ev_1}"', two heads, from=3-1, to=3-3]
        \arrow[""{name=0, anchor=center, inner sep=0}, "H"', from=3-3, to=3-4]
        \arrow["\lrcorner"{anchor=center, pos=0.125}, draw=none, from=2-3, to=0]
      \end{tikzcd}\]
      which now have been factorized into cocartesian functors over $A_0$.
      But by \autoref{laxliftingpropertyofleftadjointsagainstcocartesianfibrations}
      and \autoref{liftsofDeltaFagainstcocartesianfibrationsincommutativesquares},
      there already exists a unique such cocartesian lift, hence our two lifts must be equivalent.
    \end{proof}

    Now we are prepared to prove that the pullback of a cospan of left adjoints in $\Fun(\Delta^1,\Cat)$ with 
    legs being adjointable squares produces again a left adjoint and the pullback structure maps will 
    automatically also be adjointable.

    \cubepullbackofladj*
    \begin{proof}
        By \autoref{adjointablesquaresgivecommutativesquaresofhomcatequivalences}
        we have commutative squares
      \[\begin{tikzcd}[column sep=2.25em]
        {B_0} & {B_2} & {B_1} \\
        {B_0 \times_{A_0} \Fun(\Delta^1,A_0)} & {B_2 \times_{A_2} \Fun(\Delta^1,A_2)} & {B_1 \times_{A_1} \Fun(\Delta^1,A_1)} \\
        {\Fun(\Delta^1,B_0) \times_{B_0} A_0} & {\Fun(\Delta^1,B_2) \times_{B_2} A_2} & {\Fun(\Delta^1,B_1) \times_{B_1} A_1}
        \arrow["K"{pos=0.4}, from=1-1, to=1-2]
        \arrow["{\Delta_{F_0}}"', hook, from=1-1, to=2-1]
        \arrow["{(\eta_0, F_0)}"', shift right=5, curve={height=30pt}, from=1-1, to=3-1]
        \arrow["{\Delta_{F_2}}", hook, from=1-2, to=2-2]
        \arrow["{(\eta_2, F_2)}"{pos=0.7}, shift left=5, curve={height=-30pt}, from=1-2, to=3-2]
        \arrow["M", from=1-3, to=1-2]
        \arrow["{\Delta_{F_1}}", hook, from=1-3, to=2-3]
        \arrow["{(\eta_1, F_1)}", shift left=5, curve={height=-30pt}, from=1-3, to=3-3]
        \arrow["{K \times_H H_\ast}", from=2-1, to=2-2]
        \arrow["{\phi_0}"', from=2-1, to=3-1]
        \arrow["{\phi_2}", from=2-2, to=3-2]
        \arrow["{M \times_L L_\ast}"', from=2-3, to=2-2]
        \arrow["{\phi_1}", from=2-3, to=3-3]
        \arrow["{K_\ast \times_K H}"', from=3-1, to=3-2]
        \arrow["{M_\ast \times_M L}", from=3-3, to=3-2]
      \end{tikzcd}\]
      Hence we can now pull back to obtain the commutative diagram
      \[\begin{tikzcd}[cramped,column sep=tiny,row sep=small]
        B && {B_1} \\
        & {B_0} && {B_2} \\
        {B \times_{A} \Fun(\Delta^1,A)} && {B_1 \times_{A_1} \Fun(\Delta^1,A_1)} \\
        & {B_0 \times_{A_0} \Fun(\Delta^1,A_0)} && {B_2 \times_{A_2} \Fun(\Delta^1,A_2)} \\
        {\Fun(\Delta^1,B) \times_{B} A} && {\Fun(\Delta^1,B_1) \times_{B_1} A_1} \\
        & {\Fun(\Delta^1,B_0) \times_{B_0} A_0} && {\Fun(\Delta^1,B_2) \times_{B_2} A_2} \\
        {B \times A} && {B_1 \times A_1} \\
        & {B_0 \times A_0} && {B_2 \times A_2}
        \arrow[from=1-1, to=1-3]
        \arrow[from=1-1, to=2-2]
        \arrow["{\exists ! \Delta_{F}}"', color={rgb,255:red,92;green,214;blue,92}, dashed, hook, from=1-1, to=3-1]
        \arrow["M", from=1-3, to=2-4]
        \arrow["{\Delta_{F_1}}"{pos=0.7}, hook, from=1-3, to=3-3]
        \arrow["K"{pos=0.4}, from=2-2, to=2-4]
        \arrow["{\Delta_{F_0}}"'{pos=0.7}, hook, from=2-2, to=4-2]
        \arrow["{\Delta_{F_2}}", hook, from=2-4, to=4-4]
        \arrow[from=3-1, to=3-3]
        \arrow[from=3-1, to=4-2]
        \arrow["{\exists ! \equiv}"', color={rgb,255:red,92;green,92;blue,214}, dashed, from=3-1, to=5-1]
        \arrow["{\exists !}"', shift right=6, curve={height=30pt}, from=3-1, to=7-1]
        \arrow["{M \times_L L_\ast}"{pos=0.7}, from=3-3, to=4-4]
        \arrow["{\phi_1}"{pos=0.6}, from=3-3, to=5-3]
        \arrow["\equiv"'{pos=0.6}, from=3-3, to=5-3]
        \arrow["{K \times_H H_\ast}"{pos=0.3}, from=4-2, to=4-4]
        \arrow["{\phi_0}"'{pos=0.7}, from=4-2, to=6-2]
        \arrow["\equiv"{pos=0.7}, from=4-2, to=6-2]
        \arrow["{\phi_2}", from=4-4, to=6-4]
        \arrow["\equiv"', from=4-4, to=6-4]
        \arrow[from=5-1, to=5-3]
        \arrow[from=5-1, to=6-2]
        \arrow["{\exists !}"', from=5-1, to=7-1]
        \arrow["{M_\ast \times_M L}"{pos=0.8}, from=5-3, to=6-4]
        \arrow[from=5-3, to=7-3]
        \arrow["{K_\ast \times_K H}"{pos=0.3}, from=6-2, to=6-4]
        \arrow[from=6-2, to=8-2]
        \arrow[from=6-4, to=8-4]
        \arrow[from=7-1, to=7-3]
        \arrow[from=7-1, to=8-2]
        \arrow["{M \times L}", from=7-3, to=8-4]
        \arrow["{K \times H}"', from=8-2, to=8-4]
      \end{tikzcd}\]
      in which each of the horizontal squares is a pullback and the left most vertical
      functors are the unique ones induces by pullback.
      In particular we obtain the equivalence
      \begin{equation}
                \begin{tikzpicture}[diagram]
                    \matrix[objects] {%
                      |(a)| B \times_{A} \Fun(\Delta^1,A) \& \& |(b)| \Fun(\Delta^1,B) \times_{B} A \\
                      \& |(c)| B \times A  \\
                    };
                    \path[maps,->]
                    (a) edge[dashed] node[above]  {$\equiv$} (b)
                    (b) edge node[below right]  {$\ev$} (c)
                    (a) edge node[below left]   {$\ev$} (c)
                    ;
                \end{tikzpicture}
        \end{equation}
        proving that $F$ has a right adjoint.
        As $\phi \circ \Delta_{F} = (\eta, F)$ we can also deduce invertibility of 
        $\eta$ if all the $\eta_i$ are invertible.
        For the counits one can apply a dual argument reversing the adjunction 
        equivalences in the diagrams above.
        From the composed commutative diagram
        \begin{equation}
            \begin{tikzpicture}[diagram]
                \matrix[objects] {%
                  |(a)| B \& |(b)| B_i \& |(x)| \\
                  |(c)| B \times_{A} \Fun(\Delta^1,A)  \& |(d)| B_i \times_{A_i} \Fun(\Delta^1,A_i) \& |(y)|  \\
                  |(e)| \Fun(\Delta^1,B) \times_{B} A \& |(f)| \Fun(\Delta^1,B_i) \times_{B_i} A_i \& |(z)| B_i \times_{A_i} \Fun(\Delta^1,A_i) \\
                };
                \path[maps,->]
                (a) edge node[above]  {$$} (b)
                (e) edge node[below]  {$$} (f)
                (c) edge node[below]  {$$} (d)
                (a) edge[bend right=80] node[above left]   {$(\eta,F)$} (e)
                (b) edge[bend right=80] node[above left]  {$(\eta_i,F_i)$} (f)
                (a) edge node[right] {$\Delta_{F}$} (c)
                (b) edge node[above right] {$\Delta_{F_i}$} (z)
                (b) edge node[left] {$\Delta_{F_i}$} (d)
                (c) edge node[right] {$\phi$} (e)
                (d) edge node[left] {$\phi_i$} (f)
                (f) edge node[below]  {$\psi_i$} (z)
                ;
            \end{tikzpicture}
      \end{equation}
      one can then also deduce the adjointablility claim as in the proof of 
      \autoref{adjointablilityequivtounitscommute}.
    \end{proof}

    The following is an immediate corollary of \autoref{cubepullbackofladjwithffradjandadjfacesisladj} with our definition of cocartesian fibrations.
    \begin{lemma}
        \label{pullbackofcocartfunsgivescocartfib}
        Consider the commutative cube of categories such that the top and bottom faces are
        pullbacks.
        \begin{equation}
            \begin{tikzpicture}[diagram]
                \matrix[objects,narrow] {%
                    |(a)| X_2 \& \& |(b)| X_1 \\
                    \& |(c)| X_0 \& \& |(d)| X \\
                    |(e)| A_2 \& \& |(f)| A_1 \\
                    \& |(g)| A_0 \& \& |(h)| A \\
                };
                \path[maps,->]
                (b) edge (a)
                (f) edge (e)

                (d) edge (c)
                (h) edge (g)

                (c) edge node[below left]   {$$} (a)
                (g) edge (e)

                (d) edge (b)
                (h) edge (f)

                (a) edge[->>] node[left]   {$F_2$} (e)
                (c) edge[->>] node[below left]   {$F_0$} (g)
                (b) edge[->>] node[below left]   {$F_1$} (f)
                (d) edge node[left]   {$F$} (h)
                ;
            \end{tikzpicture}
        \end{equation}
        If the $F_i$ are cocartesian fibrations and $X_0 \to X_2$ and $X_1 \to X_2$ are cocartesian functors,
        then the pullback induced functor $F$ is also a cocartesian fibration
        and $X \to X_0$, $X \to X_1$ are also cocartesian functors.
    \end{lemma}
    \begin{proof}
        The commutative cube gives us the following commutative cube in which
        the top and bottom faces are again pullbacks.
        \begin{equation*}
            \begin{tikzpicture}[diagram]
                \matrix[objects,column sep=0.5em,row sep=2em] {%
                    |(a)| \Fun(\Delta^1 , X_2) \& \& |(b)| \Fun(\Delta^1 , X_1) \\
                    \& |(c)| \Fun(\Delta^1 , X_0) \& \& |(d)| \Fun(\Delta^1 , X) \\
                    |(e)| X_2 \times_{A_2} \Fun(\Delta^1,A_2) \& \& |(f)| X_1 \times_{A_1} \Fun(\Delta^1,A_1) \\
                    \& |(g)| X_0 \times_{A_0} \Fun(\Delta^1,A_0) \& \& |(h)| X \times_{A} \Fun(\Delta^1,A) \\
                };
                \path[maps,->]
                (b) edge node[above]  {$$} (a)
                (f) edge  (e)

                (d) edge  (c)
                (h) edge  (g)

                (c) edge node[below left]   {$$} (a)
                (g) edge  (e)

                (d) edge  (b)
                (h) edge  (f)

                (a) edge node[left]   {$$} (e)
                (c) edge node[below left]   {$$} (g)
                (b) edge node[below left]   {$$} (f)
                (d) edge node[left]   {$$} (h)
                ;
            \end{tikzpicture}
        \end{equation*}
        By definition of cocartesian fibrations we know that all the vertical
        functors except the most right one have fully-faithful left adjoints.
        By definition of cocartesian functors we have that the back face and the
        left face are adjointable. Now \autoref{pullbackofadjunctions} tells us that
        also the right most vertical functor has a fully-faithful left adjoint
        and that the front face and right face of the cube are also adjointable.
        Thus we get that $F$ is a cocartesian fibration and the
        $A \to A_i$ are cocartesian functors.
    \end{proof}

    At last we are equipped to prove that general right adjoints are stable under pullbacks along 
    cocartesian fibrations.
    Indeed, the bare fact that the pullback of the right adjoint is again a right adjoint can be obtained 
    from the adjunction characterisation \autoref{havingaradjointisequivtofreecocartfibhavingaffradj} and several applications of \autoref{pullbackofadjunctions}.
    But we actually also recover the additional conclusions from \autoref{enhancedpullbackofadjunctions}, \ie that the pullback square is 
    adjointable and that the pulled back right adjoint as well as its left adjoint constitute 
    cocartesian functors.

    \pagebreak 

    \superpullbackradjalongcocartfib*
    \begin{proof}
      From \autoref{pullbackofcocartfibs} we already know that $V$ is cocartesian over $U$.
      We have the following commutative diagram
      \[\begin{tikzcd}
        { \Fun(\Delta^1,X) \times^{(\ev_1,V)}_X Y} & {\Fun(\Delta^1,X)} \\
        {X \times^{(p,\ev_0)}_B \Fun(\Delta^1,B) \times^{(\ev_1,U)}_B A} & {X \times^{(p,\ev_0)}_B \Fun(\Delta^1,B)} \\
        {\Fun(\Delta^1,B) \times^{(\ev_1,U)}_B A} & {\Fun(\Delta^1,B)}
        \arrow["{\pr_0}", from=1-1, to=1-2]
        \arrow["{(\ev_0,p_\ast, q)}"', from=1-1, to=2-1]
        \arrow["{p_\ast \times_p q}"', shift right=20, curve={height=30pt}, from=1-1, to=3-1]
        \arrow["{(\ev_0,p_\ast)}", from=1-2, to=2-2]
        \arrow["{p_\ast}", shift left=10, curve={height=-30pt}, from=1-2, to=3-2]
        \arrow[""{name=0, anchor=center, inner sep=0}, "{(\pr_0,\pr_1)}", from=2-1, to=2-2]
        \arrow["{(\pr_1,\pr_2)}"', from=2-1, to=3-1]
        \arrow["{\pr_1}", from=2-2, to=3-2]
        \arrow[""{name=1, anchor=center, inner sep=0}, "{\pr_0}"', from=3-1, to=3-2]
        \arrow["\lrcorner"{anchor=center, pos=0.125}, draw=none, from=1-1, to=0]
        \arrow["\lrcorner"{anchor=center, pos=0.125}, draw=none, from=2-1, to=1]
      \end{tikzcd}\]
      in which the bottom square is a pullback by definition and the composed rectangle
      by the defining pullback of $Y$.
      Hence by pullback cancellation we deduce that the top square is also a pullback.
      But by $p$ being a cocartesian fibration we know that $(\ev_0,p_\ast)$ has a 
      fully-faithful left adjoint. Thus we can infer from 
      \autoref{pullbackofadjunctions} that $(\ev_0, p_\ast, q)$ also has one.
      On the other hand, $U$ having a left adjoint gives us by the dual of
      \autoref{havingaradjointisequivtofreecocartfibhavingaffradj}
      that the right hand vertical functor in the pullback 
      \begin{equation*}
            \begin{tikzpicture}[diagram]
                \matrix[objects]{%
                  |(a)| X \times^{(p,\ev_0)}_B \Fun(\Delta^1,B) \times^{(\ev_1,U)}_B A \& |(b)| \Fun(\Delta^1,B) \times^{(\ev_1,U)}_B A \\
                  |(c)| X \& |(d)| B \\
                };
                \path[maps,->]
                    (a) edge node[above] {$$} (b)
                    (a) edge node[left] {$\pr_0$} (c)
                    (b) edge node[right] {$\ev_0$} (d)
                    (c) edge node[below] {$p$} (d)
                ;
                \node at (barycentric cs:a=0.8,b=0.3,c=0.3) (phi) {\mbox{\LARGE{$\lrcorner$}}};
            \end{tikzpicture}
      \end{equation*}
      has a fully-faithful left adjoint. By \autoref{pullbackofadjunctions}
      this means that also its pullback, the left hand vertical functor in the above
      diagram has a fully-faithful left adjoint.
      Now the commutative triangle
      \begin{equation}
            \begin{tikzpicture}[diagram]
                \matrix[objects] {%
                  |(a)| \Fun(\Delta^1,X) \times^{(\ev_1,V)}_X Y \& |(b)| X \times^{(p,\ev_0)}_B \Fun(\Delta^1,B) \times^{(\ev_1,U)}_B A \\
                  \& |(c)| X \\
                };
                \path[maps,->]
                (a) edge node[above]  {$(\ev_0,p_\ast, q)$} (b)
                (a) edge node[below left]  {$\ev_0$} (c)
                (b) edge node[below right]   {$\pr_0$} (c)
                ;
            \end{tikzpicture}
      \end{equation}
      tells us that we factorized $\ev_0$ into two functors that both have 
      fully-faithful left adjoints, hence it itself has one.

      By carefully tracing through the these steps one can also conclude that if the
      counit of $F \adj U$ is invertible, that the counit of $G \adj V$ is invertible too.

      To show adjointability of the pullback with respect to this newly constructed adjoint
      it suffices by \autoref{adjointablilityequivtounitscommute} to show that
      \begin{equation}
        \label{equivcharofadjointabilityforpbofradjalongcocartfib}
        \begin{tikzcd}[cramped]
        X &&&& B \\
        & {B \times_{A} \Fun(\Delta^1,A)} \\
        & {\Fun(\Delta^1,B) \times_{B} A} && {B \times_{A} \Fun(\Delta^1,A)} \\
        {\Fun(\Delta^1,X) \times_{X} Y} &&&& {\Fun(\Delta^1,B) \times_{B} A}
        \arrow["p", from=1-1, to=1-5]
        \arrow["{(\tilde{\eta}, G)}"', from=1-1, to=4-1]
        \arrow["{\Delta_{F}}", color={rgb,255:red,92;green,92;blue,214}, from=1-5, to=3-4]
        \arrow["{(\eta, F)}", color={rgb,255:red,92;green,92;blue,214}, from=1-5, to=4-5]
        \arrow["{\ev_0}", color={rgb,255:red,92;green,92;blue,214}, from=2-2, to=1-5]
        \arrow[""{name=0, anchor=center, inner sep=0}, color={rgb,255:red,92;green,92;blue,214}, Rightarrow, no head, from=2-2, to=3-4]
        \arrow["\psi", color={rgb,255:red,92;green,92;blue,214}, from=3-2, to=2-2]
        \arrow[color={rgb,255:red,92;green,92;blue,214}, Rightarrow, no head, from=3-2, to=4-5]
        \arrow["\phi", color={rgb,255:red,92;green,92;blue,214}, from=3-4, to=4-5]
        \arrow["{p_\ast \times_{p} q}", from=4-1, to=3-2]
        \arrow["{p_\ast \times_{p} q}"', from=4-1, to=4-5]
        \arrow["\epsilon", color={rgb,255:red,92;green,92;blue,214}, shorten <=63pt, shorten >=38pt, Rightarrow, from=1-5, to=0]
      \end{tikzcd}
      \end{equation}
      is invertible, where $G$ is the fully-faithful left adjoint of $V$ and $\tilde{\eta}$ the 
      unit of this adjunction.
      We furthermore know from the application of \autoref{pullbackofadjunctions}
      above that 
      \begin{equation*}
          \begin{adjustbox}{scale=0.9}
            \begin{tikzpicture}[diagram]
                \matrix[objects,column sep=3em]{%
                  |(x)| X \& |(a)| X \times^{(p,\ev_0)}_B \Fun(\Delta^1,B) \times^{(\ev_1,U)}_B A \& |(b)| \Fun(\Delta^1,B) \times^{(\ev_1,U)}_B A \\
                  \& |(c)| X \& |(d)| B \& |(y)| \Fun(\Delta^1,B) \times^{(\ev_1,U)}_B A \\
                };
                \path[maps,->]
                    (x) edge node[above] {$(\id , \eta ,F)$} (a)
                    (a) edge node[above] {$(\pr_1,\pr_2)$} (b)
                    (a) edge node[left] {$\pr_0$} (c)
                    (b) edge[color=pastellviolett] node[right] {$\ev_0$} (d)
                    (c) edge node[below] {$p$} (d)
                    (d) edge[color=pastellviolett] node[below] {$(\eta, F)$} (y)
                ;
                \path[maps,-] 
                (x) edge[double distance=0.2em] node[below left] (idx) {$$} (c)
                (b) edge[color=pastellviolett,double distance=0.2em] node[above right] (idrest) {$$} (y)
                ;
                \node[rotate=45] at ($ (idx) ! 0.5 ! (a) $) (psi) {$\Rightarrow$};
                \node[maps,right] at (psi) {$\bar{\eta}$};

                \node[color=pastellviolett,rotate=45] at ($ (idrest) ! 0.5 ! (d) $) (eps) {$\Rightarrow$};
                \node[color=pastellviolett,maps,right] at (eps) {$\epsilon'$};
            \end{tikzpicture}
          \end{adjustbox}
      \end{equation*}
      is invertible, but as $\bar{\eta}$ was invertible to begin with, this is equivalent to 
      the whiskering of $\epsilon'$ being invertible.
      Note also that the color-highlighted natural transformations in these two diagrams
      coincide.
      We also know by \autoref{havingaradjointisequivtofreecocartfibhavingaffradj}
      that the functor $(\tilde{\eta},G)$ is exactly the fully-faithful left adjoint  of $\ev_0 \colon \Fun(\Delta^1,X) \times^{(\ev_1,V)}_X Y \to X$ .
      By construction of this left adjoint we have also the commutative diagram
      \begin{equation*}
            \begin{tikzpicture}[diagram]
                \matrix[objects,column sep=3em]{%
                  |(a)| X \\
                  |(c)| \Fun(\Delta^1,X) \times^{(\ev_1,V)}_X Y \& |(d)| X \times^{(p,\ev_0)}_B \Fun(\Delta^1,B) \times^{(\ev_1,U)}_B A \& |(y)| \Fun(\Delta^1,B) \times^{(\ev_1,U)}_B A \\
                };
                \path[maps,->]
                    (a) edge node[left] {$(\tilde{\eta},G)$} (c)
                    (a) edge node[above right] {$(\id , \eta ,F)$} (d)
                    (c) edge node[below] {$(\ev_0,p_\ast, q)$} (d)
                    (d) edge node[below] {$(\pr_1,\pr_2)$} (y)
                    (c) edge[out=-20,in=-160] node[above] {$p_\ast \times_p q$} (y)
                ;
            \end{tikzpicture}
      \end{equation*}
      That means we can actually deduce from the invertibility of 
      $\epsilon' (\pr_1,\pr_2) (\id, \eta, F)$
      the desired invertibility of \autoref{equivcharofadjointabilityforpbofradjalongcocartfib}.

      Let us now look into why $G$ is a cocartesian functor.
      In the first step we produced a fully-faithful left adjoint for $(ev_0,p_\ast,q)$ by 
      applying \autoref{pullbackofadjunctions}.
      One can picture the pullback definition of this left adjoint $L$ as in 
      diagram of pullbacks
      \[\begin{tikzcd}[cramped]
        {X \times^{(p,\ev_0)}_B \Fun(\Delta^1,B) \times^{(\ev_1,U)}_B A} & {X \times^{(p,\ev_0)}_B \Fun(\Delta^1,B)} \\
        { \Fun(\Delta^1,X) \times^{(\ev_1,V)}_X Y} & {\Fun(\Delta^1,X)} \\
        {X \times^{(p,\ev_0)}_B \Fun(\Delta^1,B) \times^{(\ev_1,U)}_B A} & {X \times^{(p,\ev_0)}_B \Fun(\Delta^1,B)} \\
        {\Fun(\Delta^1,B) \times^{(\ev_1,U)}_B A} & {\Fun(\Delta^1,B)}
        \arrow["{(\pr_0,\pr_1)}"', from=1-1, to=1-2]
        \arrow["L"', hook, from=1-1, to=2-1]
        \arrow["{(\pr_1, \pr_2)}"', shift right=30, curve={height=50pt}, from=1-1, to=4-1]
        \arrow["{\text{cocart}}", hook, from=1-2, to=2-2]
        \arrow["{\pr_1}", shift left=20, curve={height=-50pt}, from=1-2, to=4-2]
        \arrow[""{name=0, anchor=center, inner sep=0}, "{\pr_0}", from=2-1, to=2-2]
        \arrow["{(\ev_0,p_\ast, q)}"', from=2-1, to=3-1]
        \arrow["{p_\ast \times_p q}"', shift right=20, curve={height=30pt}, from=2-1, to=4-1]
        \arrow["{(\ev_0,p_\ast)}", from=2-2, to=3-2]
        \arrow["{p_\ast}", shift left=10, curve={height=-30pt}, from=2-2, to=4-2]
        \arrow[""{name=1, anchor=center, inner sep=0}, "{(\pr_0,\pr_1)}", from=3-1, to=3-2]
        \arrow["{(\pr_1,\pr_2)}"', from=3-1, to=4-1]
        \arrow["{\pr_1}", from=3-2, to=4-2]
        \arrow[""{name=2, anchor=center, inner sep=0}, "{\pr_0}"', from=4-1, to=4-2]
        \arrow["\lrcorner"{anchor=center, pos=0.125}, draw=none, from=1-1, to=0]
        \arrow["\lrcorner"{anchor=center, pos=0.125}, draw=none, from=2-1, to=1]
        \arrow["\lrcorner"{anchor=center, pos=0.125}, draw=none, from=3-1, to=2]
      \end{tikzcd}\]
      But as we proved in \autoref{cocartiscocartfun}, the functor 
      $\text{cocart}$ is in fact a cocartesian functor from $\pr_1$ to $p_\ast$.
      Hence by \autoref{pullbackofcocartfunsgivescocartfib} the functor 
      $L$ is also a cocartesian functor between its pulled back cocartesian fibrations
      $(\pr_1, \pr_2)$ and $p_\ast \times_p q$.
      In the second step we applied \autoref{pullbackofadjunctions} to the pullback 
      \begin{equation*}
            \begin{tikzpicture}[diagram]
                \matrix[objects]{%
                  |(a)| X \times^{(p,\ev_0)}_B \Fun(\Delta^1,B) \times^{(\ev_1,U)}_B A \& |(b)| \Fun(\Delta^1,B) \times^{(\ev_1,U)}_B A \\
                  |(c)| X \& |(d)| B \\
                };
                \path[maps,->]
                    (a) edge node[above] {$$} (b)
                    (a) edge node[left] {$\pr_0$} (c)
                    (b) edge node[right] {$\ev_0$} (d)
                    (c) edge node[below] {$p$} (d)
                ;
                \node at (barycentric cs:a=0.8,b=0.3,c=0.3) (phi) {\mbox{\LARGE{$\lrcorner$}}};
            \end{tikzpicture}
      \end{equation*}
      But as $p$ is a cocartesian fibration we can deduce from 
      \autoref{enhancedpullbackofadjunctions} that the pulled 
      back fully-faithful left adjoint $T$ of the functor $\pr_0$ in the 
      pullback above is in fact cocartesian over the left adjoint $(\eta, F)$ 
      of $\ev_0$ above.
      Let us now put these two cocartesianness observations together
      with the way we composed $L$ and $T$ to get $(\tilde{\eta}, G)$
      \begin{equation*}
          \begin{tikzpicture}[diagram]
              \matrix[objects,wide]{%
                |(x)| X \& |(a)| X \times^{(p,\ev_0)}_B \Fun(\Delta^1,B) \times^{(\ev_1,U)}_B A \& |(b)| \Fun(\Delta^1,X) \times^{(\ev_1,V)}_X Y \& |(z)| Y \\
                  |(y)| B \& |(c)| \Fun(\Delta^1,B) \times^{(\ev_1,U)}_B A \& \& |(d)| A \\
              };
              \path[maps,->]
                  (a) edge node[above] {$L$} (b)
                  (a) edge[->>] node[left] {$$} (c)
                  (b) edge[->>] node[below right] {$p_\ast \times_p q$} (c)
                  (c) edge node[below] {$\pr_A$} (d)
                  (b) edge node[below] {$\pr_Y$} (z)

                  (x) edge[->>] node[right] {$p$} (y)
                  (z) edge[->>] node[right] {$q$} (d)
                  (x) edge node[above] {$T$} (a)
                  (y) edge node[below] {$(\eta, F)$} (c)
                  (x) edge[out=20,in=160] node[above] {$G$} (z)
                  (y) edge[out=-20,in=-160] node[below] {$F$} (d)
              ;
          \end{tikzpicture}
      \end{equation*}
      to see that $G$ is in fact cocartesian over $F$, as we can just post compose these to cocartesian
      functors with the indicated cocartesian functor $\pr_Y$ over $\pr_A$.
    \end{proof}

  \subsection{Relative Adjunctions}
    \label{subsectionrelativeadjoints}

    In this subsection we give an elementary proof of the statement that a cocartesian functor 
    \begin{equation}
          \begin{tikzpicture}[diagram]
              \matrix[objects] {%
                |(a)| A \& \& |(b)| B \\
                \& |(c)| I \\
              };
              \path[maps,->]
              (a) edge node[above]  {$F$} (b)
              (a) edge[->>] node[below left]   {$p$} (c)
              (b) edge[->>] node[below right]  {$q$} (c)
              ;
          \end{tikzpicture}
    \end{equation}
    between cocartesian fibrations over a base category $I$ has a relative right adjoint $U$ over $I$ 
    if and only if it does so fiberwise in the objects of $I$, \ie the functors $F_i$ on the fibers.
    We will furthermore enhance this result by \autoref{cocartesiannessofradjisequivtoadjointabilityofcocartpushforwardsquareforladj}, a characterisation of the cocartesianness 
    of the right adjoint $U$ in terms of the squares 
    \begin{equation}
        \begin{tikzpicture}[diagram]
            \matrix[objects]{%
                |(a)| A_i \& |(b)| B_i \\
                |(c)| A_j \& |(d)| B_j \\
            };
            \path[maps,->]
                (a) edge node[above] {$F_i$} (b)
                (a) edge node[left] {$k_!$} (c)
                (b) edge node[right] {$k_!$} (d)
                (c) edge node[below] {$F_j$} (d)
            ;
        \end{tikzpicture}
    \end{equation}
    being adjointable.

    The hard part is to show that the fiberwise given right adjoints glue together to give a right adjoint 
    on the total categories, but we can alternatively characterise the existence of both fiberwise right 
    adjoints and the total category right adjoint in terms of their relative slice categories 
    \begin{equation*}
      A_{q(b)} \times_{B_{q(b)}} (B_{q(b)})_{/ b} \qquad \text{and} \qquad A \times_{B} B_{/ b} \\
    \end{equation*}
    having terminal objects.
    Our strategy will be to exhibit the canonical inclusion functor 
    \begin{equation*}
      A_{q(b)} \times_{B_{q(b)}} (B_{q(b)})_{/ b} \mono A \times_{B} B_{/ b} \\
    \end{equation*}
    as a right adjoint, thus preserving terminal objects.

    We will obtain this by pulling back, via \autoref{pullbackofrightadjointalongcocartesianfibrationisrightadjoint}, the canonical fully-faithful right adjoint 
    $\Delta_I \colon I \mono \Fun(\Delta^1,I)$ along a certain cocartesian fibration that we 
    are going to construct in the following lemma.

    \begin{lemma}
      \label{cocartesianfunctorcommacategorycocartesianfibration}
      For any cocartesian functor
      \begin{equation}
            \begin{tikzpicture}[diagram]
                \matrix[objects] {%
                  |(a)| A \& \& |(b)| B \\
                  \& |(c)| I \\
                };
                \path[maps,->]
                (a) edge node[above]  {$F$} (b)
                (a) edge node[below left]   {$p$} (c)
                (b) edge node[below right]  {$q$} (c)
                ;
            \end{tikzpicture}
      \end{equation}
      between cocartesian fibrations $p$ and $q$ the functor
      $q_\ast \circ \pr_1 \colon A \times_{B} \Fun(\Delta^1,B) \to \Fun(\Delta^1,I)$
      is also a cocartesian fibration.
    \end{lemma}
    \begin{proof}
      We can apply \autoref{pullbackofcocartfunsgivescocartfib} to the
      cube
      \[\begin{tikzcd}[cramped]
        {A \times_B \Fun(\Delta^1,B)} && {\Fun(\Delta^1,B)} \\
        & A && B \\
        {\Fun(\Delta^1,I)} && {\Fun(\Delta^1,I)} \\
        & I && I
        \arrow[from=1-1, to=1-3]
        \arrow[from=1-1, to=2-2]
        \arrow[two heads, from=1-1, to=3-1]
        \arrow["{\ev_0}", from=1-3, to=2-4]
        \arrow["{q_\ast}"{pos=0.2}, two heads, from=1-3, to=3-3]
        \arrow["F"{pos=0.3}, from=2-2, to=2-4]
        \arrow["q", two heads, from=2-4, to=4-4]
        \arrow[Rightarrow, no head, from=3-1, to=3-3]
        \arrow["{\ev_0}", from=3-1, to=4-2]
        \arrow["{\ev_0}", from=3-3, to=4-4]
        \arrow[Rightarrow, no head, from=4-2, to=4-4]
        \arrow["p"'{pos=0.3}, two heads, from=2-2, to=4-2]
      \end{tikzcd}\]
      because we know by assumption that $F$ is a cocartesian functor and by
      \autoref{evaluationfunctorsarecocartesian} that $q_\ast$ is a 
      cocartesian fibration and $\ev_0$ is also a cocartesian functor.
    \end{proof}

    \begin{corollary}
      \label{relativecommacategoryreflectivelyincommacategory}
      Let $F$ be a cocartesian functor
      \begin{equation}
            \begin{tikzpicture}[diagram]
                \matrix[objects] {%
                  |(a)| A \& \& |(b)| B \\
                  \& |(c)| I \\
                };
                \path[maps,->]
                (a) edge node[above]  {$F$} (b)
                (a) edge node[below left]   {$p$} (c)
                (b) edge node[below right]  {$q$} (c)
                ;
            \end{tikzpicture}
      \end{equation}
      between cocartesian fibrations $p$ and $q$.
      Applying \autoref{pullbackofrightadjointalongcocartesianfibrationisrightadjoint} 
      to the cocartesian fibration
      \begin{equation*}
        A \times_{B} \Fun(\Delta^1,B) \fibration \Fun(\Delta^1,I) 
      \end{equation*}
      obtained from
      \autoref{cocartesianfunctorcommacategorycocartesianfibration} and the adjunction
      $\ev_1 \adj \Delta_I$ we can conclude that the top
      horizontal functor in the pullback
      \begin{equation*}
            \begin{tikzpicture}[diagram]
                \matrix[objects]{%
                  |(a)| A \times_{B} \Fun_{/I}(\Delta^1,B) \& |(b)| A \times_{B} \Fun(\Delta^1,B) \\
                    |(c)| I \& |(d)| \Fun(\Delta^1,I) \\
                };
                \path[maps,->]
                    (a) edge node[above] {$$} (b)
                    (a) edge[->>] node[left] {$$} (c)
                    (b) edge[->>] node[right] {$q_\ast$} (d)
                    (c) edge node[below] {$\Delta_I$} (d)
                ;
                \node at (barycentric cs:a=0.8,b=0.3,c=0.3) (phi) {\mbox{\LARGE{$\lrcorner$}}};
            \end{tikzpicture}
      \end{equation*}
      is a fully-faithful right adjoint.
    \end{corollary}

    Now the only thing left to do is to exhibit the adjunction we obtained as an adjunction over 
    $B$, so that we can deduce from it also the fiberwise adjunctions.

    \begin{lemma}
      \label{relativecommaadjunctionlivesoverB}
      The adjunction 
      \begin{equation*}
        \begin{tikzpicture}[diagram]
          \matrix[objects] {
            |(a)| A \times_{B} \Fun_{/I}(\Delta^1,B) \& |(b)| A \times_{B} \Fun(\Delta^1,B) \\
          };
          \path[maps,,->] 
            (b) edge[bend right] node[above] (f) {$$} (a)
            (a) edge node[below] (g) {$\incl$}  (b)
          ;
          \node[rotate=-90] at ($ (f) ! 0.5 ! (g) $) (adj) {$\adj$};
        \end{tikzpicture}
      \end{equation*}
      from \autoref{relativecommacategoryreflectivelyincommacategory}
      lives over $B$ via the functors
      \begin{equation}
            \begin{tikzpicture}[diagram]
                \matrix[objects] {%
                  |(a)| A \times_{B} \Fun_{/I}(\Delta^1,B) \& \&  |(b)| A \times_{B} \Fun(\Delta^1,B) \\
                  \& |(c)| B \\
                };
                \path[maps,->]
                (a) edge node[above]  {$\incl$} (b)
                (a) edge node[below left]   {$\ev_1$} (c)
                (b) edge node[below right]  {$\ev_1$} (c)
                ;
            \end{tikzpicture}
      \end{equation}
      \ie the unit and the counit are invertible after postcomposing with the respective functors $\ev_1$.
    \end{lemma}
    \begin{proof}
      The counit is already invertible by construction.
      Unpacking the proof of \autoref{pullbackofrightadjointalongcocartesianfibrationisrightadjoint}
      one sees that the unit of this adjunction is actually the cocartesian lift of 
      \[\begin{tikzcd}[cramped]
        {A \times^{(F,\ev_0)}_B \Fun(\Delta^1,B)} && {A \times^{(F,\ev_0)}_B \Fun(\Delta^1,B)} \\
        {\Fun(\Delta^1,I)} \\
        & I & {\Fun(\Delta^1,I)}
        \arrow[equals, from=1-1, to=1-3]
        \arrow[two heads, from=1-1, to=2-1]
        \arrow[two heads, from=1-3, to=3-3]
        \arrow["{\ev_1}"', from=2-1, to=3-2]
        \arrow[""{name=0, anchor=center, inner sep=0}, curve={height=-12pt}, equals, from=2-1, to=3-3]
        \arrow["{\Delta_I}"', hook, from=3-2, to=3-3]
        \arrow["\eta"', shorten <=3pt, Rightarrow, from=0, to=3-2]
      \end{tikzcd}\]
      where $\eta$ is the unit of the adjunction $\Delta_I \adj \ev_1$.
      Note that this whiskered natural transformation becomes invertible after
      post-whiskering with $\ev_1 \colon \Fun(\Delta^1,I) \to I$.
      Recall furthermore from the applied \autoref{pullbackofcocartfunsgivescocartfib}
      that the top horizontal functors in
      \begin{equation*}
            \begin{tikzpicture}[diagram]
                \matrix[objects,column sep=3em]{%
                  |(a)| A \times^{(F,\ev_0)}_B \Fun(\Delta^1,B) \& |(b)| \Fun(\Delta^1,B) \& |(x)| B \\
                  |(c)| \Fun(\Delta^1,I) \& |(d)| \Fun(\Delta^1,I) \& |(y)| I \\
                };
                \path[maps,->]
                (a) edge[->>] node[left] {$q_\ast \circ \pr_1$} (c)
                    (a) edge node[above] {$\pr_1$} (b)
                    (b) edge[->>] node[right] {$q_\ast$} (d)
                    (x) edge[->>] node[right] {$q$} (y)
                    (b) edge node[below] {$\ev_1$} (x)
                    (d) edge node[below] {$\ev_1$} (y)
                    (c) edge[bend right] node[above] {$\ev_1$} (y)
                    (a) edge[bend left] node[above] {$\ev_1$} (x)
                ;\path[maps,-] 
                (c) edge[double distance=0.2em] (d)
                ;
            \end{tikzpicture}
      \end{equation*}
      are cocartesian.
      Hence they preserve cocartesian lifts. The statement now follows from
      \autoref{cocartesianliftsofequivsareequivs}.
    \end{proof}

    Next we make precise what a relative adjunction, \ie an adjunction over a base category $I$ is.
    \begin{definition}
      Let
      \[\begin{tikzcd}[cramped]
        A && B \\
        & I
        \arrow["F", shift left, from=1-1, to=1-3]
        \arrow["G"', shift right, from=1-1, to=1-3]
        \arrow["p"', from=1-1, to=2-2]
        \arrow["q", from=1-3, to=2-2]
      \end{tikzcd}\]
      be two functors over $I$. We define a \define{natural transformation over $I$} a to be natural transformation $\alpha \colon F \twoto G$
      if as a functor $\alpha \colon A \to \Fun(\Delta^1,B)$ 
      we have a commutative diagram
      \[\begin{tikzcd}[cramped]
        A && {\Fun_{/I}(\Delta^1,B)} & {\Fun(\Delta^1,B)} \\
        && I & {\Fun(\Delta^1,I)}
        \arrow[dashed, from=1-1, to=1-3]
        \arrow["\alpha", curve={height=-30pt}, from=1-1, to=1-4]
        \arrow["p"', from=1-1, to=2-3]
        \arrow[from=1-3, to=1-4]
        \arrow[from=1-3, to=2-3]
        \arrow["{q_\ast}", from=1-4, to=2-4]
        \arrow[""{name=0, anchor=center, inner sep=0}, "{\Delta_I}", from=2-3, to=2-4]
        \arrow["\lrcorner"{anchor=center, pos=0.125}, draw=none, from=1-3, to=0]
      \end{tikzcd}\]
      Furthermore, we say that a functor $F$ over $I$ as below on the left is 
      \define{left adjoint over $I$} to a functor $U$ over $I$ as below on the right in
      \[\begin{tikzcd}[cramped]
        A && B & B && A \\
        & I &&& I
        \arrow["F", from=1-1, to=1-3]
        \arrow["p"', from=1-1, to=2-2]
        \arrow["q", from=1-3, to=2-2]
        \arrow["U", from=1-4, to=1-6]
        \arrow["q"', from=1-4, to=2-5]
        \arrow["p", from=1-6, to=2-5]
      \end{tikzcd}\]
      if we have natural transformations $\eta \colon \id \twoto UF$ and $\epsilon \colon FU \twoto \id$ over $I$,
      such that $\eta$ and $\epsilon$ are unit and counit of an adjunction $F \adj U$.
    \end{definition}

    The final step is to show that relative adjunctions are in fact stable under changing the base category 
    via pulling back, in order to deduce the desired fiberwise adjunctions.
    \begin{lemma}
      \label{pullbackpreservesadjunctionsoverbase}
      Let
      \begin{equation}
            \begin{tikzpicture}[diagram]
                \matrix[objects] {%
                  |(a)| A \& \& |(b)| B \\
                  \& |(c)| J \\
                };
                \path[maps,->]
                (a) edge node[above]  {$F$} (b)
                (a) edge node[below left]   {$p$} (c)
                (b) edge node[below right]  {$q$} (c)
                ;
            \end{tikzpicture}
      \end{equation}
      be a functor over a base category $J$ such that it is left adjoint to $U$ over $J$.
      Let us also be given another functor $t \colon I \to J$.
      Then
      \begin{equation}
            \begin{tikzpicture}[diagram]
                \matrix[objects] {%
                  |(a)| I \times_J A \& \& |(b)| I \times_J B \\
                  \& |(c)| I \\
                };
                \path[maps,->]
                (a) edge node[above]  {$I \times_J F$} (b)
                (a) edge node[below left]   {$I \times_J p$} (c)
                (b) edge node[below right]  {$I \times_J q$} (c)
                ;
            \end{tikzpicture}
      \end{equation}
      is left adjoint to $I \times_J U$ over $I$.
    \end{lemma}
    \begin{proof}
      Let us denote the unit of the adjunction $F \adj U$ by 
      $\eta \colon A \to \Fun(\Delta^1,A)$
      and the counit by 
      $\epsilon \colon B \to \Fun(\Delta^1,B)$.
      The fact that $\eta$ becomes invertible when postwhiskered
      with the functor $p$ to $J$ can be expressed as
      the commutativity of the following square.
      \begin{equation}
            \begin{tikzpicture}[diagram]
                \matrix[objects] {%
                  |(a)| A \& |(b)| \Fun(\Delta^1,A) \\
                  |(c)| J \& |(d)| \Fun(\Delta^1,J) \\
                };
                \path[maps,->]
                (a) edge node[above]  {$\eta$} (b)
                (a) edge node[left]   {$p$} (c)
                (b) edge node[right]  {$p_\ast$} (d)
                (c) edge node[below]  {$\Delta_J$} (d)
                ;
            \end{tikzpicture}
      \end{equation}
      By definition of the relative functor category $\Fun_{/ J}(\Delta^1, A)$ 
      this induces us by pullback a functor $A \to \Fun_{/ J}(\Delta^1, A)$ 
      over $J$, which we will also call $\eta$ by abuse of notation.
      Similarly, we can do this for $\epsilon$ to get a functor
      $B \to \Fun_{/ J}(\Delta^1, B)$.
      Note that the identity natural transformation between on the 
      functor $\id_A$ can also be upgraded to a functor
      $\Delta_A \colon A \to \Fun_{/ J}(\Delta^1,A)$ over $J$, similarly for $B$.
      As morphism composition 
      $\Fun(\Delta^1,A) \times_A \Fun(\Delta^1,A) \to \Fun(\Delta^1,A)$ also
      pulls back to a functor
      $\Fun_{/ J}(\Delta^1,A) \times_A \Fun_{/ J}(\Delta^1,A) \to \Fun_{/ J}(\Delta^1,A)$,
      over $J$, one can also express the triangle identities of $F \adj U$ 
      using these functors over $J$.
      As we know that in the commutative cube
      \[\begin{tikzcd}[cramped]
        {I \times_J \Fun_{/ J}(\Delta^1,A)} && {\Fun_{/ J}(\Delta^1,A)} \\
        & {\Fun(\Delta^1,I \times_J A)} && {\Fun(\Delta^1,A)} \\
        I && J \\
        & {\Fun(\Delta^1,I)} && {\Fun(\Delta^1,J)}
        \arrow[from=1-1, to=1-3]
        \arrow["{\exists !}"', from=1-1, to=2-2]
        \arrow[from=1-1, to=3-1]
        \arrow[from=1-3, to=2-4]
        \arrow[from=1-3, to=3-3]
        \arrow[from=2-2, to=2-4]
        \arrow[from=2-2, to=4-2]
        \arrow["{p_\ast}", from=2-4, to=4-4]
        \arrow["t"', from=3-1, to=3-3]
        \arrow["{\Delta_I}"', from=3-1, to=4-2]
        \arrow["{\Delta_I}"', from=3-3, to=4-4]
        \arrow["{t_\ast}"', from=4-2, to=4-4]
      \end{tikzcd}\]
      the back, the right and the front face are pullbacks, we are able to 
      deduce by pullback cancellation that also the left hand face is a 
      pullback, \ie that we have an equivalence
      \begin{equation*}
        I \times_J \Fun_{/ J}(\Delta^1, A) \xto{\equiv} \Fun_{/ I}(\Delta^1,I \times_J A)
      \end{equation*}
      over $I$.
      Hence by pulling back our relative version of $\eta$ 
      for example we get
      \begin{equation*}
        I \times_J A \xto{I \times_J \eta} I \times_J \Fun_{/ J}(\Delta^1, A) \xto{\equiv} \Fun_{/ I}(\Delta^1,I \times_J A)
      \end{equation*}
      \ie a relative natural transformation over $I$ between the functors 
      \begin{equation*}
        \id_{(I \times_J A)} \text{ and } (I \times_J U)(I \times_J F) \colon I \times_J A \to I \times_J A
      \end{equation*}
      over $I$.
      Now one just needs to contemplate that the all the squares in 
      the commutative prisms
      \[\begin{tikzcd}[cramped]
        {\Fun_{/ I}(\Delta^1,I \times_J A) \times_{(I \times_J A)} \Fun_{/ I}(\Delta^1,I \times_J A)} && {\Fun_{/ I}(\Delta^1,I \times_J A)} \\
        & I \\
        {\Fun_{/ J}(\Delta^1,A) \times_A \Fun_{/ J}(\Delta^1,A)} && {\Fun_{/ J}(\Delta^1,A)} \\
        & J
        \arrow["{\text{comp}}", from=1-1, to=1-3]
        \arrow[from=1-1, to=2-2]
        \arrow[from=1-1, to=3-1]
        \arrow[from=1-3, to=2-2]
        \arrow[from=1-3, to=3-3]
        \arrow["t"{pos=0.3}, from=2-2, to=4-2]
        \arrow["{\text{comp}}"{pos=0.3}, from=3-1, to=3-3]
        \arrow[from=3-1, to=4-2]
        \arrow[from=3-3, to=4-2]
      \end{tikzcd}\]
      and
      \[\begin{tikzcd}[cramped]
        {I \times_J A} && {\Fun_{/ I}(\Delta^1,I \times_J A)} \\
        & I \\
        A && {\Fun_{/ J}(\Delta^1,A)} \\
        & J
        \arrow["{\Delta_{(I \times_J A)}}", from=1-1, to=1-3]
        \arrow[from=1-1, to=2-2]
        \arrow[from=1-1, to=3-1]
        \arrow[from=1-3, to=2-2]
        \arrow[from=1-3, to=3-3]
        \arrow["t"{pos=0.3}, from=2-2, to=4-2]
        \arrow["{\Delta_A}"{pos=0.3}, from=3-1, to=3-3]
        \arrow["p"', from=3-1, to=4-2]
        \arrow[from=3-3, to=4-2]
      \end{tikzcd}\]
      are in fact pullbacks, to see that compositions of relative 
      natural transformations over $J$ and identity relative natural 
      transformations over $J$ pull back to compositions of and identity 
      natural transformations over $I$.
      The same is also true for whiskerings of relative natural transformations, 
      and also if we replace $p \colon A \to J$ by $q \colon B \to J$.
      Thus also the triangle identities of $F \adj U$, which live over 
      $J$ pull back to triangle identities for $I \times_J F$ and 
      $I \times_J U$.
    \end{proof}

    \begin{corollary}
      \label{objectwiserelativesliceadjunction}
      Let $F$ be a cocartesian functor
      \begin{equation}
            \begin{tikzpicture}[diagram]
                \matrix[objects] {%
                  |(a)| A \& \& |(b)| B \\
                  \& |(c)| I \\
                };
                \path[maps,->]
                (a) edge node[above]  {$F$} (b)
                (a) edge node[below left]   {$p$} (c)
                (b) edge node[below right]  {$q$} (c)
                ;
            \end{tikzpicture}
      \end{equation}
      Then pulling back the adjunction over $B$ from \autoref{relativecommaadjunctionlivesoverB}
      for each object $b$ of $B$ tells us that the inclusions
      \begin{equation*}
        A_{q(b)} \times_{B_{q(b)}} (B_{q(b)})_{/b} \mono A \times_{B} B_{/b}
      \end{equation*}
      are all fully-faithful right adjoints.
    \end{corollary}
    \begin{proof}
      This follows from \autoref{relativecommaadjunctionlivesoverB} and 
      \autoref{pullbackpreservesadjunctionsoverbase}.
    \end{proof}

    We are now equipped to conclude the fact that relative right adjoints can be given fiberwise.

    \reladjcanbegivenobjwinbase*
    \begin{proof}
        We employ the equivalent characterisation of a functor $G \colon X \to Y$ being left adjoint if and
        only if for all objects $y$ of $Y$ the relative slice categories
        $X \times_{Y} Y_{/y}$ have terminal objects, for both our statements.
        By \autoref{objectwiserelativesliceadjunction}
        now if $A_{q(b)} \times_{B_{q(b)}} (B_{q(b)})_{/b}$ admits a terminal object,
        then this will be preserved by the right adjoint
        \begin{equation*}
          A_{q(b)} \times_{B_{q(b)}} (B_{q(b)})_{/b} \mono A \times_{B} B_{/b}
        \end{equation*}
        Thus we have that also $A \times_{B} B_{/b}$ admits a terminal object.
        The other way around, if we know that $A \times_{B} B_{/b}$ has a terminal object
        and that the counit map $\epsilon_b$ witnessing this terminality actually
        already lives in the fiber over $q(b)$, then the terminal object already lives in and is terminal
        in $A_{q(b)} \times_{B_{q(b)}} (B_{q(b)})_{/b}$.
    \end{proof}

    For our purposes this is actually, as mentioned in the beginning of this subsection, not 
    enough, we also need to control the cocartesianness of the glued-together right adjoint $U$.
    To this end, the following characterisation of cocartesianness of $U$ in terms of just the 
    fiberwise adjunctions $F_i \adj U_i$ is useful.

    \cocartnessofreladj*
    \begin{proof}

      Cocartesianness of $F$ means by definition that
      \begin{equation}
        \label{defofcocartfunbetweencocartfib}
            \begin{tikzpicture}[diagram]
                \matrix[objects]{%
                    |(a)| \Fun(\Delta^1, A) \& |(b)| \Fun(\Delta^1, B) \\
                    |(c)| A \times_{I} \Fun(\Delta^1, I) \& |(d)| B \times_{I} \Fun(\Delta^1, I) \\
                };
                \path[maps,->]
                    (a) edge node[above] {$F_\ast$} (b)
                    (a) edge node[left] {$\ev_0$} (c)
                    (b) edge node[right] {$\ev_0$} (d)
                    (c) edge node[below] {$F \times_{\id} \id$} (d)
                ;
            \end{tikzpicture}
      \end{equation}
      is vertically adjointable, \ie the vertical mate is invertible and thus makes the square
      \begin{equation}
        \label{vertadjsquare}
            \begin{tikzpicture}[diagram]
                \matrix[objects]{%
                    |(a)| \Fun(\Delta^1, A) \& |(b)| \Fun(\Delta^1, B) \\
                    |(c)| A \times_{I} \Fun(\Delta^1, I) \& |(d)| B \times_{I} \Fun(\Delta^1, I) \\
                };
                \path[maps,->]
                    (a) edge node[above] {$F_\ast$} (b)
                    (c) edge node[left] {$\text{cocart}$} (a)
                    (d) edge node[right] {$\text{cocart}$} (b)
                    (c) edge node[below] {$F \times_{\id} \id$} (d)
                ;
            \end{tikzpicture}
      \end{equation}
      commutative. But by the assumption that $U$ is a relative right adjoint over $I$ we also know that its 
      horizontal mate is invertible, \ie make the square
      \begin{equation}
        \label{horadjsquare}
            \begin{tikzpicture}[diagram]
                \matrix[objects]{%
                    |(a)| \Fun(\Delta^1, A) \& |(b)| \Fun(\Delta^1, B) \\
                    |(c)| A \times_{I} \Fun(\Delta^1, I) \& |(d)| B \times_{I} \Fun(\Delta^1, I) \\
                };
                \path[maps,->]
                    (b) edge node[above] {$U_\ast$} (a)
                    (a) edge node[left] {$\ev_0$} (c)
                    (b) edge node[right] {$\ev_0$} (d)
                    (d) edge node[below] {$U \times_{\id} \id$} (c)
                ;
            \end{tikzpicture}
      \end{equation}
      commutative. Now we can apply \autoref{bothwaysadjointablegivessametotalmate}
      to conclude that the horizontal mate of \autoref{vertadjsquare} is equivalent to the
      vertical mate of \autoref{horadjsquare}.
      But vertical adjointability of \autoref{horadjsquare} is by definition 
      cocartesianness of the functor $U$, \ie the first statement.
      On the other hand, the composite diagram
      \begin{equation}
            \begin{tikzpicture}[diagram]
                \matrix[objects]{%
                    |(x)| A \& |(y)| B \\
                    |(a)| \Fun(\Delta^1, A) \& |(b)| \Fun(\Delta^1, B) \\
                    |(c)| A \times_{I} \Fun(\Delta^1, I) \& |(d)| B \times_{I} \Fun(\Delta^1, I) \\
                };
                \path[maps,->]
                    (x) edge node[above] {$F$} (y)
                    (a) edge node[left] {$\ev_0$} (x)
                    (b) edge node[right] {$\ev_0$} (y)

                    (a) edge node[above] {$F_\ast$} (b)
                    (c) edge node[left] {$\text{cocart}$} (a)
                    (d) edge node[right] {$\text{cocart}$} (b)
                    (c) edge node[below] {$F \times_{\id} \id$} (d)
                ;
            \end{tikzpicture}
      \end{equation}
      is always horizontally adjointable as we have for every cocartesian fibration 
      by invertibility of the unit of the adjunction $\text{cocart} \adj (\ev_0, p_\ast)$ that
      \begin{equation}
            \begin{tikzpicture}[diagram]
                \matrix[objects]{%
                    |(a)| A \times_I \Fun(\Delta^1,I) \& |(b)| \Fun(\Delta^1,A) \\
                    |(c)|  \& |(d)| A \times_I \Fun(\Delta^1,A) \& |(x)| A \\
                };
                \path[maps,->]
                    (a) edge node[above] {$\text{cocart}$} (b)
                    (d) edge node[below] {$\pr_0$} (x)
                    (b) edge node[left] {$(\ev_0, p_\ast)$} (d)
                    (b) edge node[above right] {$\ev_0$} (x)
                ;
                \path[maps,-]
                    (a) edge[double distance=0.2em] (d)
                ;
            \end{tikzpicture}
      \end{equation}
      commutes.
      Now as $(\ev_0, \ev_1) \colon \Fun(\Delta^1, A) \to A \times A$ 
      is conservative, 
      the horizontal adjointability of \autoref{vertadjsquare}
      is equivalent to horizontal adjointability of
      \begin{equation}
            \begin{tikzpicture}[diagram]
                \matrix[objects]{%
                    |(x)| A \& |(y)| B \\
                    |(a)| \Fun(\Delta^1, A) \& |(b)| \Fun(\Delta^1, B) \\
                    |(c)| A \times_{I} \Fun(\Delta^1, I) \& |(d)| B \times_{I} \Fun(\Delta^1, I) \\
                };
                \path[maps,->]
                    (x) edge node[above] {$F$} (y)
                    (a) edge node[left] {$\ev_1$} (x)
                    (b) edge node[right] {$\ev_1$} (y)

                    (a) edge node[above] {$F_\ast$} (b)
                    (c) edge node[left] {$\text{cocart}$} (a)
                    (d) edge node[right] {$\text{cocart}$} (b)
                    (c) edge node[below] {$F \times_{\id} \id$} (d)
                ;
            \end{tikzpicture}
      \end{equation}
      The equivalence of this and the second statement is due to the fact that we can 
      check invertibility of the mate natural transformation objectwise.
    \end{proof}

  \printbibliography

@unpublished{formalizationbookproject,
  author = {Cisinski, Denis-Charles and Cnossen, Bastiaan and Nguyen, Kim and
            Walde, Tashi},
  title = {Formalization of Higher Categories},
  year = {2025},
  month = {04},
  note = {In progress. Available online at \url{
          https://drive.google.com/file/d/1lKaq7watGGl3xvjqw9qHjm6SDPFJ2-0o/view
          }},
}

@unpublished{ha,
  author = {Lurie, Jacob},
  title = {Higher Algebra},
  year = {2017},
  month = {09},
  note = {Unpublished. Available online at \url{https://www.math.ias.edu/~lurie/
          }},
}

@book{GR,
  AUTHOR = {Gaitsgory, Dennis and Rozenblyum, Nick},
  TITLE = {A study in derived algebraic geometry. {V}ol. {I}. {C}orrespondences
           and duality},
  SERIES = {Mathematical Surveys and Monographs},
  VOLUME = {221},
  PUBLISHER = {American Mathematical Society, Providence, RI},
  YEAR = {2017},
  PAGES = {xl+533pp},
  ISBN = {978-1-4704-3569-1},
  MRCLASS = {14F05 (18D05 18G55)},
  MRNUMBER = {3701352},
  MRREVIEWER = {Adrian Langer},
  DOI = {10.1090/surv/221.1},
  URL = {https://doi.org/10.1090/surv/221.1},
}

@article{gepnerhaugseng,
  AUTHOR = {Gepner, David and Haugseng, Rune},
  TITLE = {Enriched {$\infty$}-categories via non-symmetric {$\infty$}-operads},
  JOURNAL = {Adv. Math.},
  FJOURNAL = {Advances in Mathematics},
  VOLUME = {279},
  YEAR = {2015},
  PAGES = {575--716},
  ISSN = {0001-8708,1090-2082},
  MRCLASS = {18D20 (18D10 18D50)},
  MRNUMBER = {3345192},
  MRREVIEWER = {Christopher\ L.\ Rogers},
  DOI = {10.1016/j.aim.2015.02.007},
  URL = {https://doi.org/10.1016/j.aim.2015.02.007},
}

@article{AMGR,
  title = {Factorization homology of enriched $\infty$-categories},
  author = {David Ayala and John Francis and Aaron Mazel-Gee and Nick Rozenblyum
            },
  year = {2024},
  eprint = {1710.06414},
  archivePrefix = {arXiv},
  primaryClass = {math.AT},
  url = {https://arxiv.org/abs/1710.06414},
}

@article{barwickschommerpries,
  AUTHOR = {Barwick, Clark and Schommer-Pries, Christopher},
  TITLE = {On the unicity of the theory of higher categories},
  JOURNAL = {J. Amer. Math. Soc.},
  FJOURNAL = {Journal of the American Mathematical Society},
  VOLUME = {34},
  YEAR = {2021},
  NUMBER = {4},
  PAGES = {1011--1058},
  ISSN = {0894-0347,1088-6834},
  MRCLASS = {18N60 (18N65)},
  MRNUMBER = {4301559},
  MRREVIEWER = {Niall\ Taggart},
  DOI = {10.1090/jams/972},
  URL = {https://doi.org/10.1090/jams/972},
}

@article{riehlshulman,
  AUTHOR = {Riehl, Emily and Shulman, Michael},
  TITLE = {A type theory for synthetic {$\infty$}-categories},
  JOURNAL = {High. Struct.},
  FJOURNAL = {Higher Structures},
  VOLUME = {1},
  YEAR = {2017},
  NUMBER = {1},
  PAGES = {147--224},
  ISSN = {2209-0606},
  MRCLASS = {18D05 (03G30 18G35 18G55)},
  MRNUMBER = {3912054},
  MRREVIEWER = {Thomas\ Streicher},
  DOI = {10.1007/s42001-017-0005-6},
  URL = {https://doi.org/10.1007/s42001-017-0005-6},
}

@article{buchholtzweinberger,
  AUTHOR = {Buchholtz, Ulrik and Weinberger, Jonathan},
  TITLE = {Synthetic fibered {$(\infty,1)$}-category theory},
  JOURNAL = {High. Struct.},
  FJOURNAL = {Higher Structures},
  VOLUME = {7},
  YEAR = {2023},
  NUMBER = {1},
  PAGES = {74--165},
  ISSN = {2209-0606},
  MRCLASS = {03B38 (18D30 18D40 18D70 18N45 18N60)},
  MRNUMBER = {4600458},
  MRREVIEWER = {Zoran\ \v Skoda},
}

@article{blomstraighteningofeveryfunctor,
  title = {On the straightening of every functor},
  author = {Thomas Blom},
  year = {2025},
  eprint = {2408.16539},
  archivePrefix = {arXiv},
  primaryClass = {math.CT},
  url = {https://arxiv.org/abs/2408.16539},
}

@article{Martini2021,
  title = {Yoneda's lemma for internal higher categories},
  author = {Louis Martini},
  year = {2021},
  eprint = {2103.17141},
  journal = {arXiv preprint},
  archivePrefix = {arXiv},
  primaryClass = {math.CT},
}

@article{MW2021,
  author = {Louis Martini and Sebastian Wolf},
  title = {Limits and colimits in internal higher category theory},
  year = {2021},
  archiveprefix = {arXiv},
  eprint = {2111.14495},
  primaryclass = {math.CT},
  journal = {arXiv preprint},
}

@misc{gaitsgoryrozenblyumconjectures,
  title = {On the squares functor and the Gaitsgory-Rozenblyum conjectures},
  author = {Félix Loubaton and Jaco Ruit},
  year = {2025},
  eprint = {2507.07807},
  archivePrefix = {arXiv},
  primaryClass = {math.CT},
  url = {https://arxiv.org/abs/2507.07807},
}

\end{document}